\documentclass[a4paper,11pt]{article}

\usepackage[utf8x]{inputenc}
\usepackage[T1]{fontenc}
\usepackage{lmodern}

\usepackage[colorlinks]{hyperref}
\hypersetup{linkcolor=blue, citecolor=magenta}
\usepackage[frenchb]{babel}
\usepackage{fullpage}

\usepackage{bbm}
\usepackage[bbgreekl]{mathbbol}
\usepackage{nccrules}
\usepackage[nottoc]{tocbibind}
\usepackage[intlimits,leqno]{amsmath}
\usepackage{amsthm}
\usepackage{amssymb}
\usepackage{mathrsfs}
\usepackage{stmaryrd}
\usepackage{xspace}
\usepackage[all]{xy}
\newdir{ >}{{}*!/-10pt/\dir{>}}
\newdir{|>}{!/4.5pt/@{|}*:(1,-.2)@^{>}*:(1,+.2)@_{>}}

\newcommand{\Z}{\ensuremath{\mathbb{Z}}}

\newcommand{\R}{\ensuremath{\mathbb{R}}}
\newcommand{\C}{\ensuremath{\mathbb{C}}}

\newcommand{\A}{\ensuremath{\mathbb{A}}}

\newcommand{\Tr}{\ensuremath{\mathrm{tr}\,}}

\newcommand{\Resprod}{\ensuremath{{\prod}'}}

\newcommand{\dd}{\ensuremath{\,\mathrm{d}}}

\newcommand{\angles}[1]{\ensuremath{\langle #1 \rangle}}
\newcommand{\mes}{\ensuremath{\mathrm{mes}}}

\newcommand{\identity}{\ensuremath{\mathrm{id}}}

\newcommand{\Hom}{\ensuremath{\mathrm{Hom}}}

\newcommand{\rightiso}{\ensuremath{\stackrel{\sim}{\rightarrow}}}
\newcommand{\leftiso}{\ensuremath{\stackrel{\sim}{\leftarrow}}}

\newcommand{\Ker}{\ensuremath{\mathrm{Ker}\,}}

\newcommand{\Gm}{\ensuremath{\mathbb{G}_\mathrm{m}}}

\newcommand{\Supp}{\ensuremath{\mathrm{Supp}}}

\newcommand{\GL}{\ensuremath{\mathrm{GL}}}

\newcommand{\U}{\ensuremath{\mathrm{U}}}

\newcommand{\SL}{\ensuremath{\mathrm{SL}}}
\newcommand{\Sp}{\ensuremath{\mathrm{Sp}}}

\newcommand{\Mp}{\ensuremath{\widetilde{\mathrm{Sp}}}}

\newcommand{\bmu}{\ensuremath{\bbmu}}

\newcommand{\noyau}{\ensuremath{\boldsymbol{\varepsilon}}} 
\newcommand{\rev}{\ensuremath{\mathbf{p}}} 
\newcommand{\asp}{\ensuremath{\dashrule[.7ex]{2 2 2 2}{.4}}} 

\renewcommand{\Re}{\ensuremath{\mathrm{Re}\xspace}}
\renewcommand{\Im}{\ensuremath{\mathrm{Im}\xspace}}

\newtheorem{definition-fr}{D\'{e}finition}[section]{\bf}{\rm}
\newtheorem{proposition-fr}[definition-fr]{Proposition}{\bf}{\it}
\newtheorem{theoreme}[definition-fr]{Th\'{e}or\`{e}me}{\bf}{\it}
\newtheorem{corollaire}[definition-fr]{Corollaire}{\bf}{\it}
\newtheorem{lemme}[definition-fr]{Lemme}{\bf}{\it}
\newtheorem{remarque}[definition-fr]{Remarque}{\bf}{\rm}
\newtheorem{hypothese}[definition-fr]{Hypoth\`{e}se}{\bf}{\rm}

\title{La formule des traces pour les revêtements de groupes réductifs connexes. IV. \\ Distributions invariantes}
\author{Wen-Wei Li}
\date{}

\begin{document}

\maketitle

\begin{abstract}
  Nous établissons la formule des traces invariante à la Arthur pour les revêtements adéliques des groupes réductifs connexes sur un corps de nombres, sous l'hypothèse que le théorème de Paley-Wiener invariant soit vérifié pour tout sous-groupe de Lévi en les places archimédiennes réelles. Cette hypothèse est vérifiée pour les revêtements métaplectiques de $\GL(n)$ et ceux de $\Sp(2n)$ à deux feuillets, par exemple. La démonstration est basée sur les articles antérieurs et sur les idées d'Arthur. Nous donnons également des formes simples de la formule des traces lorsque la fonction test satisfait à certaines propriétés de cuspidalité.
\end{abstract}



\tableofcontents

\section{Introduction}\label{sec:intro}
Cet article fait partie d'un programme \cite{Li14a,Li12b,Li13} visant à établir la formule des traces invariante à la Arthur pour les revêtements des groupes réductifs connexes sur un corps de nombres $F$. Nous le complétons dans cet article sous l'hypothèse technique (l'Hypothèse \ref{hyp:PW}) que le Théorème de Paley-Wiener invariant soit satisfait en les places archimédiennes réelles.

Les revêtements interviennent dans des problèmes importants en arithmétique, par exemple les formes modulaires à poids demi-entiers \cite{Weil64}, la correspondance métaplectique \cite{KF86} de Flicker et Kazhdan et les séries de Dirichlet multiples \cite{BBCFH}, pour n'en citer que quelque uns. Au vu de la puissance de la formule des traces d'Arthur-Selberg pour les groupes réductifs connexes, une formule des traces pour les revêtements est très souhaitable.

Pour commencer, on note $\A$ l'anneau d'adèles de $F$ et on considère un $F$-groupe réductif connexe $G$. Grosso modo, un revêtement de $G(\A)$ à $m$ feuillets est une extension centrale de groupes topologiques localement compacts
\begin{align*}
  1 \to \bmu_m \to \tilde{G} \xrightarrow{\rev} G(\A) \to 1, \\
  \text{où} \quad \bmu_m := \{ z \in \C^\times : z^m=1 \}.
\end{align*}
Afin d'étudier les formes et représentations automorphes de $\tilde{G}$, il faut imposer d'autres conditions naturelles sur $\tilde{G}$, eg. l'existence d'une section $G(F) \to \tilde{G}$, ce que l'on fixe, et la commutativité de l'algèbre de Hecke sphérique en presque toute place non-archimédienne. On renvoie à \cite{Li14a} pour les détails. De plus, on peut se limiter à l'étude de représentations automorphes spécifiques de $\tilde{G}$, c'est-à-dire les représentations automorphes $\pi$ qui satisfont à $\pi(\noyau) = \noyau \cdot \identity$ pour tout $\noyau \in \bmu_m$. Pour cela, il suffit de considérer les fonctions test $f: \tilde{G} \to \C$ anti-spécifiques, c'est-à-dire celles satisfaisant à $f(\noyau \tilde{x})=\noyau^{-1}f(\tilde{x})$ pour tous $\noyau \in \bmu_m$, $\tilde{x} \in \tilde{G}$.

On note $C^\infty_{c,\asp}(\tilde{G}^1)$ l'espace des fonctions $C^\infty_c$ anti-spécifiques sur $\tilde{G}^1$. Comme dans le cas de groupes réductifs connexes, on a déjà la formule des traces dite grossière pour $\tilde{G}$, de la forme suivante
$$ \sum_{\chi \in \mathfrak{X}} J_\chi(\mathring{f}^1) = J(\mathring{f}^1) = \sum_{\mathfrak{o} \in \mathcal{O}}J_{\mathfrak{o}}(\mathring{f}^1), \quad \mathring{f}^1 \in C^\infty_c(\tilde{G}^1) $$
où $\tilde{G}^1 := \rev^{-1}(G(\A)^1)$, et $G(\A)^1$ est le groupe des éléments $x \in G(\A)$ tel que $|\chi(x)|=1$ pour tout $\chi \in \Hom_{F-\text{grp}}(G, \Gm)$. La forme linéaire $J$ est une trace convenablement régularisée de l'opérateur de convolution par $\mathring{f}^1$ agissant sur $L^2(G(F) \backslash \tilde{G}^1)$, et les $J_\chi$ (resp. $J_\mathfrak{o}$) sont des objets géométriques (resp. spectraux) indexés par les classes de conjugaison semi-simples dans $G(F)$ (resp. données cuspidales automorphes de $\tilde{G}$). Ces formes linéaires s'appellent souvent ``distributions'', bien que leur continuité est toujours un problème délicat. De plus, on peut raffiner les termes $J_\chi$, $J_\mathfrak{o}$; cela conduit aux développements fins géométrique et spectral établis dans \cite{Li14a} et \cite{Li13}, respectivement.

Or $J$ n'est pas une distribution invariante lorsque $G$ n'est pas anisotrope modulo le centre, par conséquent la formule des traces grossière n'est pas assez satisfaisante. L'idée d'Arthur est de la rendre invariante de la manière suivante.
\begin{enumerate}
  \item On déduit de $J$ une distribution sur $\tilde{G}_V := \rev^{-1}(G(F_V))$, notée toujours $J$, où $V$ est un ensemble fini de places de $F$ contenant les places ramifiées ou archimédiennes.
  \item Pour tout sous-groupe de Lévi $L$ de $G$, on choisit un espace convenable de fonctions test lisses anti-spécifiques $\mathcal{H}_{\tilde{L}}$ avec une application surjective $\mathcal{H}_{\tilde{L}} \to I\mathcal{H}_{\tilde{L}}$.
  \item Trouver une bonne application $\phi^1_{\tilde{L}}: \mathcal{H}_{\tilde{G}} \to I\mathcal{H}_{\tilde{L}}$, satisfaisant à des propriétés sous conjugaison qui sont analogues à celles de $J_\chi$, $J_\mathfrak{o}$, etc. Pour $L=G$, on doit retrouver $\mathcal{H}_{\tilde{G}} \to I\mathcal{H}_{\tilde{G}}$.
  \item Définir une distribution $I=I^{\tilde{G}}$ par récurrence de sorte que $I$ est invariante (cela résultera des propriétés de $\phi^1_{\tilde{L}}$),  $I$ se factorise par $I\mathcal{H}_{\tilde{G}}$, et on a la décomposition
    $$ J(f) = \sum_{L \in \mathcal{L}(M_0)} \frac{|W^L_0|}{|W^G_0|} I^{\tilde{L}}(\phi^1_{\tilde{L}}(f)), \quad f \in \mathcal{H}_{\tilde{G}}, $$
    où $\mathcal{L}(M_0)$ est l'ensemble des sous-groupes de Lévi contenant un sous-groupe de Lévi minimal $M_0$ choisi, et $W^L_0$ est le groupe de Weyl de $L$ par rapport à $M_0$.
  \item Les termes avec $L=G$ doivent s'exprimer en termes des objets qui nous intéressent, disons les caractères des représentations ou les intégrales orbitales invariantes. 
\end{enumerate}
En fait, on doit aussi appliquer une variante de cette procédure aux distributions dans les développements fins, à savoir les intégrales orbitales pondérées et les caractères pondérés. Une fois réussie, cette procédure donne des distributions locales invariantes $I_{\tilde{M}}(\cdots)$ par lesquelles $I$ s'exprime, pour tout $M \in \mathcal{L}(M_0)$.

Malgré les difficultés énormes, Arthur parvient à compléter ce programme dans \cite{Ar88-TF1,Ar88-TF2} d'une façon légèrement différente pour les groupes réductifs connexes. Il obtient ainsi la première formule des traces invariante générale. Dans un article plus récent \cite{Ar02}, il propose une autre façon d'obtenir une formule des traces invariante, visant à la stabilisation. Indiquons une différence substantielle entre ces deux approches. Les distributions locales spectrales dans \cite{Ar02} sont définies à l'aide de fonctions $\mu$ de Harish-Chandra, ce qu'Arthur appelle la normalisation canonique dans \cite{Ar98}, tandis que celles dans \cite{Ar88-TF1,Ar88-TF2} dépendent d'un choix de facteurs normalisants pour les opérateurs d'entrelacement locaux.

Revenons aux revêtements. Il s'avère que des modifications non-triviales et parfois assez techniques interviennent lorsque l'on envisage d'adapter les formules des traces invariante aux revêtements. Heureusement, la plupart des obstacles sont enlevés dans \cite{Li14a,Li12b,Li13}, où on a obtenu une formule des traces raffinée non-invariante; on renvoie aux introductions desdites références pour une discussion plus approfondie. Dans cet article, on rassemble les briques déjà obtenues et énonce une formule des traces invariante. Notre approche est modelée sur \cite{Ar02}. Il y a cependant quelques différences. Expliquons.

\begin{itemize}
  \item Dans \cite{Ar02}, Arthur définit les données centrales de la forme $(Z,\zeta)$ où $Z$ est un tore central induit dans $G$ et $\zeta$ est un caractère automorphe de $Z(\A)$. Cette complication est due à l'endoscopie. Pour les revêtements, on ne voit pas encore le besoin de telles constructions. Par contre, notre version de donnée centrale sera une paire $(A,\mathfrak{a})$ où $A \subset A_{G,\infty}$ est un sous-groupe de Lie connexe, isomorphe à $\mathfrak{a} \subset \mathfrak{a}_G$ par l'application $H_{\tilde{G}}$ de Harish-Chandra. On s'intéresse finalement au cas $A=A_{G,\infty}$, ce qui correspond pour l'essentiel au choix $Z=\{1\}$ au sens de \cite{Ar02}. Or les autres données centrales sont aussi nécessaires pour la démonstration.
  \item Dans \cite[\S 5]{Ar88-TF2}, Arthur complète sa construction des distributions invariantes locales à l'aide d'un argument global. Dans notre cadre, l'argument global paraît problématique mais on peut l'achever de façon locale. Plus précisément, notre démonstration reposera sur une généralisation du ``Théorème 0'' de Kazhdan \cite[Corollaire 5.8.9]{Li12b}, un sous-produit de la formule des traces locale invariante pour les revêtements.
  \item Enfin, nous essayons de réconcilier les choix différents de fonctions test pour les formules des traces invariantes dans \cite{Ar88-TF1,Ar88-TF2} et \cite{Ar02}. La transition est formelle pourvu que le formalisme de données centrales et de leurs espaces de fonctions associés soit mis en place, ce que nous ferons soigneusement dans \S\ref{sec:PW}.
\end{itemize}

Donnons quelques remarques sur le Théorème de Paley-Wiener (l'Hypothèse \ref{hyp:PW}). Nous le mettons comme une hypothèse en les places réelles. Le cas des groupes réductifs connexes est résolu dans \cite{CD84,CD90} mais l'auteur ne sait pas comment traiter le cas de revêtements. Néanmoins, c'est possible de vérifier cette hypothèse dans certains cas (voir \S\ref{sec:PW-exemples}), par exemple le revêtement métaplectique de Weil de $\Sp(2n,\R)$ ou les revêtements de $\GL(n,\R)$. En particulier, notre résultat justifie les travaux de Mezo \cite{MZ01,MZ02} sur la correspondance métaplectique, ce qui présume la formula des traces invariante pour de tels revêtements.

Le choix de l'espace de fonctions test est une affaire formelle, cependant il mérite quelques remarques. Il y en a au moins quatre candidats. 
\renewcommand{\labelenumi}{\Alph{enumi}.}\begin{enumerate}
  \item L'espace $\mathcal{H}_{\asp}(\tilde{G}^1_V)$ des fonctions dans $C^\infty_c(\tilde{G}^1_V)$  qui sont anti-spécifiques et $\tilde{K}_V$-finies sous la translation bilatérale par un sous-groupe compact maximal $\tilde{K}_V$ de $\tilde{G}_V$ en bonne position relativement à $M_0$. C'est le choix fait dans \cite{Ar88-TF1,Ar88-TF2}, et aussi dans \cite{Ar02} pour le cas $Z=\{1\}$. En fait, Arthur définit des distributions en choisissant des fonctions $f^\star \in \mathcal{H}_{\asp}(\tilde{G}_V)$ dont la restriction à $\tilde{G}^1_V$ est égale à une fonction donnée $f^1 \in \mathcal{H}_{\asp}(\tilde{G}^1_V)$, ensuite il construit des distributions en $f^\star$ et montre qu'elles ne dépendent que de $f^1$.
  \item L'espace $\mathcal{H}_{\asp}(\tilde{G}_V, A_{G,\infty})$ des fonctions dans $C^\infty_c(\tilde{G}_V/A_{G,\infty})$ qui sont anti-spécifiques et $\tilde{K}_V$-finies. Cet espace s'identifie visiblement à $\mathcal{H}_{\asp}(\tilde{G}^1_V)$ par restriction. Toutefois nous préférons le point de vue $A_{G,\infty}$-équivariant dans cet article car il se conforme mieux aux conventions adoptées dans le cas tordu \cite{Lab04}.
  \item L'espace $\mathcal{H}_{\asp}(\tilde{G}_V)$. Intégration le long de $A_{G,\infty}$ donne une application surjective $\mathcal{H}_{\asp}(\tilde{G}_V) \to \mathcal{H}_{\asp}(\tilde{G}_V, A_{G,\infty})$, notée $f \mapsto f^1$. Si l'on a défini la distribution invariante $I$ de $\mathcal{H}_{\asp}(\tilde{G}_V, A_{G,\infty})$, alors on peut en déduire une distribution de $\mathcal{H}_{\asp}(\tilde{G}_V)$ par la formule $I(f) = I(f^1)$. C'est la distribution que l'on obtient si l'on part de l'action sur $L^2(G(F) A_{G,\infty} \backslash \tilde{G})$ des éléments de $C^\infty_{c,\asp}(\tilde{G})$ au lieu de $C^\infty_{c,\asp}(\tilde{G}/A_{G,\infty})$. Cet espace est aussi utilisé dans \cite{Lab04} et c'est le plus commode à manipuler dans les démonstrations.
  \item Dans les démonstrations, on aura aussi besoin de passer à un espace $\mathcal{H}_{\text{ac},\asp}(\tilde{G}_V)$ de fonctions test anti-spécifiques ``à support presque compact'', abrégé par ``ac'' en omettant l'accent grave. Il contient tous les espaces ci-dessus.
\end{enumerate}\renewcommand{\labelenumi}{\arabic{enumi}.}

On obtient les variantes des espaces ci-dessus en remplaçant $A_{G,\infty}$ par $A$ qui fait partie d'une donnée centrale générale $(A,\mathfrak{a})$. Le cas $A=\{1\}$ donne $\mathcal{H}_{\asp}(\tilde{G}_V)$ et $\mathcal{H}_{\text{ac},\asp}(\tilde{G})$. En principe, toute propriété des espaces de fonctions test peut se réduire au cas $A=\{1\}$; en particulier, les propriétés de l'application cruciale $\phi^1_{\tilde{M}}$ se déduit de celles de son avatar $\phi_{\tilde{M}}$ associé à la donnée centrale $A=\{1\}$. Le Théorème de Paley-Wiener invariant intervient dans la description de l'espace $I\mathcal{H}_{\asp}(\tilde{G}_V)$.

La formule des traces invariante est donnée dans le Théorème \ref{prop:I}. C'est de la forme
\begin{multline*}
  \sum_{M \in \mathcal{L}^G(M_0)} \frac{|W^M_0|}{|W^G_0|} \sum_{\tilde{\gamma} \in \Gamma(\tilde{M}^1, V)} a^{\tilde{M}}(\tilde{\gamma}) I_{\tilde{M}}(\tilde{\gamma}, f^1) = I(f^1) \\
  = \sum_{t \geq 0} \sum_{M \in \mathcal{L}^G(M_0)} \frac{|W^M_0|}{|W^G_0|} \int_{\Pi_{t,-}(\tilde{M}^1,V)} a^{\tilde{M}}(\pi) I_{\tilde{M}}(\pi,0,f^1) \dd\pi ,
\end{multline*}
pour tout $f^1 \in \mathcal{H}_{\asp}(\tilde{G}_V, A_{G,\infty})$. Ici $\Gamma(\tilde{M}^1, V)$ est un ensemble de bonne classes de conjugaison dans $\tilde{M}^1_V$ (voir \cite[Définition 2.6.1]{Li14a}), et $\Pi_{t,-}(\tilde{M}^1,V)$ est un espace de représentations irréductibles unitaires spécifiques de $\tilde{M}^1_V$ dont le caractère central est trivial sur $A_{M,\infty}$, qui est muni d'une mesure; on renvoie à la Définition \ref{def:coeff-spec} et la Remarque \ref{rem:coeff-spec} pour les détails techniques sur cette mesure. Le côté géométrique est une somme absolument convergente tandis que le côté spectral est regardé comme une intégrale itérée. On renvoie à la Remarque \ref{rem:convergence-spec} pour une discussion sur le paramètre gênant $t$.

En pratique, on utilise souvent les formes simples de ladite formule des traces données dans le Théorème \ref{prop:formules-simples}, pourvu que la fonction test $f^1$ satisfasse à certaines conditions de cuspidalité (voir la Définition \ref{def:cuspidale}).

Enfin, indiquons que la formule des traces invariante est une étape intermédiaire mais nécessaire pour la formule des traces stable, si l'on est ambitieux d'envisager la stabilisation complète pour les revêtements. Il paraît que le cas où $\tilde{G}=\Mp(2n)$ est le revêtement métaplectique de Weil de $G=\Sp(2n)$ est abordable.

\paragraph{Organisation de cet article}
Dans \S\ref{sec:preliminaires}, on récapitule les conventions et notations essentielles. On définit aussi les principaux espaces de fonctions test qui suffisent pour énoncer la formule des traces invariante. Une étude plus approfondie de ces espaces fait l'objet de la section suivante.

Le \S\ref{sec:PW} est consacré à l'étude des espaces de fonctions test plus généraux ainsi que leur espace de Paley-Wiener. On énonce l'Hypothèse \ref{hyp:PW} et on établit la liaison entre les espaces associés à des données centrales différentes. Dans \S\ref{sec:PW-exemples}, on vérifie l'Hypothèse \ref{hyp:PW} dans certains cas spéciaux.

Dans \S\ref{sec:dists-locales}, on définit tout d'abord les intégrales orbitales pondérées et les caractères pondérés unitaires qui sont canoniquement normalisés. Malgré l'importance des caractères pondérés non-unitaires, ils n'apparaissent pas dans cet article. L'application mentionnée plus haut
$$ \phi^1_{\tilde{L}}: \mathcal{H}_{\asp}(\tilde{G}_V, A_{G,\infty}) \to I\mathcal{H}_{\asp}(\tilde{L}_V, A_{L,\infty}) $$
est définie par caractères pondérés tempérés. Observons que les espaces à l'indice ``ac'' n'apparaissent pas ici. Ces ingrédients permettront de démarrer la machine d'Arthur, et nous fournissent des distributions invariantes locales
$$ I_{\tilde{M}}(\cdots): I\mathcal{H}_{\asp}(\tilde{G}_V, A) \to \C$$
pour toute donnée centrale $(A,\mathfrak{a})$. On établit leurs existence et propriétés de façon purement locale.

Les deux sections \S\ref{sec:PW} et \S\ref{sec:dists-locales} concernent la théorie locale. En particulier, l'ensemble fini $V$ de places n'est sujet qu'à la condition d'adhérence (la Définition \ref{def:adherence}).

Dans \S\ref{sec:developpements}, on effectue la ``compression'' faite dans \cite[\S\S 2-3]{Ar02} pour les développements fins, afin d'absorber tout renseignement global aux coefficients $a^{\tilde{M}}(\cdots)$. Cette procédure introduit des intégrales orbitales pondérées non ramifiées $r^L_M(k)$ dans les coefficients géométriques, et des facteurs $r^L_M(c)$ provenant de facteurs normalisants non ramifiés dans le coefficients spectraux. On s'attend à ce que $r^L_M(c)$ s'exprime en termes de fonctions $L$ partielles et facteurs $\varepsilon$, mais cette interprétation n'est pas nécessaire pour cet article.

Dans \S\ref{sec:I}, les résultats des deux sections précédentes sont combinées et la formule des traces invariante (le Théorème \ref{prop:I}) en découle doucement. On donne également les variantes Théorème \ref{prop:I-1} et \ref{prop:I-2} pour différents usages de fonctions test. À la fin, on donne les formes simples de la formule des traces, en supposant que la fonction test est cuspidale en une ou deux places. C'est une paraphrase simple de \cite[\S 7]{Ar88-TF2}.

\paragraph{Remerciements}
L'auteur tient à remercier Jean-Loup Waldspurger pour sa lecture d'une version antérieure de ce texte. Il remercie également Hùng M\d{a}nh Bùi et Paul Mezo pour des discussions utiles.

\section{Préliminaires}\label{sec:preliminaires}
\subsection{Revêtements}\label{sec:rev}
Nous adoptons systématiquement le formalisme dans \cite{Li14a}. Pour la commodité du lecteur, rappelons brièvement les définitions concernant les revêtements.

Soient $m \in \Z_{\geq 1}$, $F$ un corps de nombres et $\A=\A_F$ son anneau d'adèles. Soit $G$ un $F$-groupe réductif connexe. Le centre de $G$ est noté par $Z_G$. On considère un revêtement adélique à $m$ feuillets de $G$ au sens de \cite{Li14a}, qui est une extension centrale de groupes localement compacts
\begin{gather}\label{eqn:ext-globale}
  1 \to \bmu_m \to \tilde{G} \xrightarrow{\rev} G(\A) \to 1,
\end{gather}
où $\bmu_m := \{z \in \C^\times : z^m=1 \}$.

À chaque ensemble fini $V$ de places de $F$, on note $F_V := \prod_{v \in V} F_v$ et on en déduit une extension centrale
$$ 1 \to \bmu_m \to \tilde{G}_V \xrightarrow{\rev_S} G(F_V) \to 1 $$
qui est la tirée-en-arrière de \eqref{eqn:ext-globale} par $G(F_V) \hookrightarrow G(\A)$. On l'appelle un revêtement local. Lorsque $V=\{v\}$, on écrit tout simplement $\tilde{G}_v \xrightarrow{\rev_v} G(F_v)$. De la même façon, on définit $\rev^V: \tilde{G}^V \to G(F^V)$ où $F^V := \Resprod_{v \notin S} F_v$.

Les éléments dans $\tilde{G}$ sont affectés d'un $\sim$, eg. $\tilde{x}$; leur image dans $G(\A)$ est notée $x$, etc.

Rappelons qu'un revêtement adélique est muni des structures suivantes \cite[\S 3.3]{Li14a}.
\begin{itemize}
  \item Un ensemble fini $V_\text{ram} = V_\text{ram}(\tilde{G})$ de places de $V$ contenant toutes les places archimédiennes.
  \item Un modèle lisse et connexe de $G$ sur l'anneau des $V_\text{ram}$-entiers dans $F$.
  \item Pour toute $v \notin V_\text{ram}$, on pose $K_v := G(\mathfrak{o}_v)$ et il existe un scindage $s_v: K_v \hookrightarrow \tilde{G}_v$, ce que l'on fixe. On suppose que $(\rev_v, K_v, s_v)$ est un revêtement non ramifié au sens de \cite[\S 3.1]{Li14a}. On exige que $(s_v)_{v \notin V_\text{ram}}: \prod_{v \notin V_\text{ram}} K_v \to \tilde{G}$ induise un homomorphisme continu $K^V \hookrightarrow \tilde{G}$.
  \item Il existe un scindage $i: G(F) \to \tilde{G}$ de $\rev$, ce que l'on fixe.
  \item Les éléments unipotents de $G(\A)$ se relèvent de façon canonique à $\tilde{G}$. Cela induit un scindage $U(\A) \to \tilde{G}$ pour tout radical unipotent d'un sous-groupe parabolique de $G$. Idem pour les revêtements locaux. Cf. \cite[\S 2.2]{Li14a}.
\end{itemize}

On choisit de plus un sous-groupe de Levi minimal $M_0$ de $G$, un sous-groupe compact maximal $K = \prod_v K_v$ de $G(\A)$ en bonne position relativement à $M_0$, et on suppose que pour tout $v \notin V_\text{ram}$, $K_v$ est le sous-groupe hyperspécial dans la donnée du revêtement non ramifié ci-dessus, ce qui est toujours possible quitte à agrandir $V_\text{ram}$. On pose $\tilde{K}_v := \rev^{-1}(K_v)$, $\tilde{K} := \rev^{-1}(K)$. Si $M \in \mathcal{L}(M_0)$, alors $K^M := K \cap M(\A)$ satisfait encore à ces propriétés avec $\tilde{M}$ au lieu de $\tilde{G}$. Les données pour un revêtement adélique sont donc héritées par les sous-groupes de Lévi contenant $M_0$.

À proprement parler, $\tilde{G}$ n'est pas le produit restreint des composantes locales $\tilde{G}_v$ par rapport à $\{K_v : v \notin V_\text{ram}\}$, mais seulement un quotient de ce produit restreint. Pour $f \in C^\infty_c(\tilde{G})$, on écrira $f = \prod_v f_v$ avec $f_v \in C^\infty_c(\tilde{G}_v)$ si la tirée-en-arrière de $f$ à $\Resprod_v \tilde{G}_v$ admet cette factorisation. De même, on peut parler des factorisations $\pi = \bigotimes'_v \pi_v$ de représentations irréductibles de $\tilde{G}$.

Une représentation $\pi$ de $\tilde{G}$ est dite spécifique si $\pi(\noyau)=\noyau \cdot \identity$ pour tout $\noyau \in \bmu_m$. Une fonction $f$ sur $\tilde{G}$ est dite anti-spécifique si $f(\noyau \tilde{x}) = \noyau^{-1} f(\tilde{x})$ pour tout $\noyau \in \bmu_m$. On définit la notion des distributions spécifiques de façon usuelle, de sorte qu'elle englobe les caractères de représentations spécifiques. Les espaces des objets spécifiques (resp. anti-spécifiques) sont affectés d'un sous-script $-$ (resp. $\asp$), eg. $C_{c,\asp}(\tilde{G})$. Les mêmes conventions s'appliquent aux sous-ensembles de $\tilde{G}$ stables par $\bmu_m$.

Nos notations pour les systèmes de racines, sous-groupes paraboliques, $(G,M)$-familles etc., sont celles d'Arthur. Cf. \cite[\S 2.3, \S 4]{Li14a}. Par exemple, soit $M$ un sous-groupe de Lévi de $G$ contenant $M_0$, on a les $\R$-espaces $\mathfrak{a}_M$, $\mathfrak{a}^G_M$, etc. Rappelons la définition de $\mathfrak{a}_M$:
$$ \mathfrak{a}_M := \Hom_\Z(\Hom_{F-\text{grp}}(M, \Gm), \R). $$
On a un scindage canonique $\mathfrak{a}_M = \mathfrak{a}^G_M \oplus \mathfrak{a}_G$. D'autre part, pour $x \in M(\A)$, l'homomorphisme $\chi \mapsto \log|\chi(x)|$ de $\Hom_F(M, \Gm)$ sur $\R$ induit un homomorphisme $H_M: M(\A) \to \mathfrak{a}_M$, ce que l'on appelle l'application d'Harish-Chandra. Soit $P$ un sous-groupe parabolique de $G$ admettant $M$ comme une composante de Lévi. La décomposition d'Iwasawa $G(\A) = P(\A)K$ et l'application quotient $P \twoheadrightarrow M$ permettent de bien définir $H_P: G(\A) \to \mathfrak{a}_M$. En composant avec $\rev$, on obtient $H_{\tilde{P}}: \tilde{P} \to \mathfrak{a}_M$.

Les duaux linéaires de $\mathfrak{a}_M$, etc. sont affectés de l'exposant $*$ et les complexifiés sont affectés du sous-script $\C$. On définit les ensembles finis suivant.
\begin{align*}
  \mathcal{L}^G(M) & := \{\text{sous-groupes de Lévi de $G$ contenant $M$}\}, \\
  \mathcal{P}^G(M) & := \{\text{sous-groupes paraboliques ayant $M$ comme une composante de Lévi}\}, \\
  \mathcal{F}^G(M) & := \{\text{sous-groupes paraboliques contenant $M$}\}, \\
  W^G_0 & := \text{le groupe de Weyl de $G$ par rapport à $M_0$}, \\
  W^G(M) & := \{w \in W^G_0 : wMw^{-1} = M \}/W^M_0, \\
  W^G_\text{reg}(M) & := \{w \in W^G(M): \det(w-1|\mathfrak{a}^G_M) \neq 0\}.
\end{align*}
Lorsqu'il n'y a pas de confusion à craindre sur $G$, on supprime l'exposant $G$.

Pour tout $P \in \mathcal{F}(M_0)$, on désigne son opposé relativement à $M_0$ par $\bar{P}$. La décomposition de Lévi de $P$ est écrite comme $P=MU$. On note $\Sigma_P \subset \mathfrak{a}^*_M$ l'ensemble des racines positives associé; son sous-ensemble des racines réduites est noté $\Sigma^{\text{red}}_P$. Les coracines sont désignées par $\alpha^\vee$, etc. 

Il y a aussi une recette pour choisir les mesures de Haar: voir \cite[\S 2.5]{Li14a}.

\subsection{Fonctions test}\label{sec:fonctions-test}
On pose $\tilde{G}^1 := \Ker(H_{\tilde{G}}) = \rev^{-1}(G(\A)^1)$, où $H_{\tilde{G}} := H_G \circ \rev$ est l'application d'Harish-Chandra. Alors $\tilde{G}^1 \supset \tilde{K}$.

Plus généralement, étant donnés $V$ un ensemble de places de $F$, et un sous-ensemble $E \subset \mathfrak{a}_G$, on pose
\begin{gather}\label{eqn:G_V^E}
  \tilde{G}_V^E := \{ \tilde{x} \in \tilde{G}_V : H_{\tilde{G}}(\tilde{x}) \in E \}.
\end{gather}
Voici une ambiguïté potentielle: lorsque $E=\{0\}$, on retrouve $\tilde{G}_V^1$.

Supposons désormais que $V \supset V_\text{ram}$. On note $F_\infty := \prod_{v|\infty} F_v$. On a aussi défini un sous-groupe central fermé $A_{G,\infty} \subset G(F_\infty)$ tel que $G(\A)^1 \times A_{G,\infty} = G(\A)$, et $H_G$ induit un isomorphisme $A_{G,\infty} \rightiso \mathfrak{a}_G$ en tant que groupes de Lie. Donc $A_{G,\infty}$ se relève canoniquement à un sous-groupe de $\tilde{G}$, qui est aussi central. Par conséquent
\begin{align}
  \label{eqn:iden}
  \tilde{G} & = \tilde{G}^1 \times A_{G,\infty}, \\
  \tilde{G}_V & = \tilde{G}^1_V \times A_{G,\infty}.
\end{align}
Si $M$ est un sous-groupe de Lévi de $G$, alors $A_{G,\infty} \hookrightarrow A_{M,\infty}$ de façon compatible avec $\mathfrak{a}_G \hookrightarrow \mathfrak{a}_M$.

Inspiré par \cite{Ar02}, nous adoptons les conventions suivantes
\begin{itemize}
  \item une fonction sur $\tilde{G}$ est écrite comme $\mathring{f}$;
  \item une fonction sur $\tilde{G}^1$ ou sur $\tilde{G}/A_{G,\infty}$ est écrite comme $\mathring{f}^1$;
  \item une fonction sur $\tilde{G}_V$ est écrite comme $f$;
  \item une fonction sur $\tilde{G}^1_V$ ou sur $\tilde{G}_V/A_{G,\infty}$ est écrite comme $f^1$;
\end{itemize}

Soient $z \in \tilde{G}$, $f$ une fonction sur $\tilde{G}$. On note $f_z$ la fonction translatée $f_z: \tilde{x} \mapsto f(z\tilde{x})$ sur $\tilde{G}$. Cette convention s'applique également au cas local.

\paragraph{Le cas global}
Suivant Arthur, on prend les espaces de Hecke
\begin{align*}
  \mathcal{H}_{\asp}(\tilde{G}) & := \{ \mathring{f} \in C_{c,\asp}^\infty(\tilde{G}) : \text{$\tilde{K}$-fini par translation à gauche et à droite} \},\\
  \mathcal{H}_{\asp}(\tilde{G}^1) & := \{ \mathring{f}^1 \in C_{c,\asp}^\infty(\tilde{G}^1) : \text{$\tilde{K}$-fini par translation à gauche et à droite} \}.
\end{align*}
On a l'application de restriction $\mathcal{H}_{\asp}(\tilde{G}) \to \mathcal{H}_{\asp}(\tilde{G}^1)$ qui est surjective. Tel est le point de vue adopté dans \cite{Li14a,Li12b,Li13} où nous considérons effectivement des distributions sur $\tilde{G}^1$. Dans cet article, nous préférons de travailler avec l'espace
\begin{gather*}
  \mathcal{H}_{\asp}(\tilde{G}, A_{G,\infty}) := \{ \mathring{f}^1 \in C_{c,\asp}^\infty(\tilde{G}/A_{G,\infty}) : \text{$\tilde{K}$-fini par translation à gauche et à droite} \}.
\end{gather*}

On dispose d'une application surjective
\begin{align}\label{eqn:f-moyenne-global}
  \mathcal{H}_{\asp}(\tilde{G}) & \longrightarrow \mathcal{H}_{\asp}(\tilde{G},A_{G,\infty}), \\
  \mathring{f} & \longmapsto \left[\mathring{f}^1(\cdot) := \int_{A_{G,\infty}} \mathring{f}_z(\cdot) \dd z \right].
\end{align}

Remarquons que $\mathcal{H}_{\asp}(\tilde{G})$ est une algèbre et $\mathcal{H}_{\asp}(\tilde{G}, A_{G,\infty})$ est un $\mathcal{H}_{\asp}(\tilde{G})$-bimodule par rapport au produit convolution.

Adoptons la convention suivante dans cet article. Soit $J$ une distribution spécifique sur $\tilde{G}^1$, c'est-à-dire une forme linéaire $\mathcal{H}_{\asp}(\tilde{G}^1) \to \C$. Elle s'identifie à une forme linéaire $\mathcal{H}_{\asp}(\tilde{G}, A_{G,\infty}) \to \C$ par \eqref{eqn:iden}. On en déduit une forme linéaire de  $\mathcal{H}_{\asp}(\tilde{G})$, notée toujours $J$, caractérisée par la formule
\begin{gather}\label{eqn:dist-recette}
  J(\mathring{f}) = J(\mathring{f}^1), \quad \mathring{f} \in \mathcal{H}_{\asp}(\tilde{G})
\end{gather}
à travers de l'application $\mathring{f} \mapsto \mathring{f}^1$ ci-dessus.

\paragraph{Le cas local}
Supposons que $V \supset V_\infty$. On peut toujours définir les espaces $\mathcal{H}_{\asp}(\tilde{G}_V)$, $\mathcal{H}_{\asp}(\tilde{G}_V, A_{G,\infty})$. Comme dans le cas précédent, $\mathcal{H}_{\asp}(\tilde{G}_V)$ est une algèbre et $\mathcal{H}_{\asp}(\tilde{G}_V,A_{G,\infty})$ est un $\mathcal{H}_{\asp}(\tilde{G}_V)$-bimodule pour le produit convolution. On dispose d'une application $\mathcal{H}_{\asp}(\tilde{G}_V) \to \mathcal{H}_{\asp}(\tilde{G}_V, A_{G,\infty})$ définie par
\begin{gather}\label{eqn:f-f^1}
  f \longmapsto \left[ f^1(\cdot) := \int_{A_{G,\infty}} f_z(\cdot) \dd z \right].
\end{gather}
De même, la convention \eqref{eqn:dist-recette} concernant les distributions spécifiques s'y applique.

Soit $J: \mathcal{H}_{\asp}(\tilde{G}_V, A_{G,\infty}) \to \C$ une distribution spécifique. Par la même recette, on en déduit une distribution $J: \mathcal{H}_{\asp}(\tilde{G}_V) \to \C$ caractérisée par $J(f)=J(f^1)$.

On reviendra aux espaces de fonctions test dans un cadre plus général dans \S\ref{sec:PW}.

\subsection{Distributions géométriques et spectrales}\label{sec:dists}
On fixe un ensemble fini $V$ des places de $F$. Dans cette sous-section, on va fixer des ensembles qui serviront à indexer les termes dans la formule des traces invariante.

Dans ce qui suit, une distribution spécifique sur $\tilde{G}_V$ signifie une forme linéaire de $\mathcal{H}_{\asp}(\tilde{G}_V)$. Lorsque $V \supset V_\infty$, une distribution spécifique $A_{G,\infty}$-équivariante signifie une forme linéaire de $\mathcal{H}_{\asp}(\tilde{G}_V, A_{G,\infty})$. Les deux espaces ne sont pas stables sous conjugaison par $\tilde{G}_V$, donc la définition usuelle d'invariance ne marche pas. Or ils sont stables sous convolution par éléments de $\mathcal{H}_{\asp}(\tilde{G}_V)$, ce qui justifie la définition suivante.

\begin{definition-fr}\label{def:dist-inv}
  On dit qu'une distribution spécifique $D$ sur $\tilde{G}_V$ est invariante si
  $$ D(f*h)=D(h*f), \quad f,h \in \mathcal{H}_{\asp}(\tilde{G}_V). $$

  Idem pour les distributions $A_{G,\infty}$-équivariantes pourvu que $V \supset V_\infty$.
\end{definition-fr}

\paragraph{Base géométrique}
On note $\mathcal{D}(\tilde{G}_V)$ l'espace des distributions spécifiques $D$ sur $\tilde{G}_V$ vérifiant
\begin{itemize}
  \item $D$ est invariante;
  \item $\Supp(D)$ est une réunion finie de classes de conjugaison;
  \item si $f$ est nulle sur $\Supp(D)$, alors $D(f)=0$.
\end{itemize}

Rappelons maintenant la notion des bons éléments et des bonnes classes \cite[Définition 2.6.1]{Li14a}.
\begin{definition-fr}\label{def:bonte}
  On dit qu'un élément $\tilde{x} \in \tilde{G}_V$ est bon si $Z_{\tilde{G}_V}(\tilde{x}) = \rev^{-1}(Z_{G(F_V)}(x))$, où $x := \rev(\tilde{x})$. La bonté de $\tilde{x}$ ne dépend que de la classe de conjugaison de $x$ par $G(F_V)$; par conséquent, on parle aussi de la bonté d'une classe de conjugaison dans $G(F_V)$ ou dans $\tilde{G}_V$.

  Ici, l'ensemble $V$ peut être n'importe quel sous-ensemble des places de $F$. On dispose donc de la notion des bons éléments dans le revêtement adélique $\rev: \tilde{G} \to G(\A)$.
\end{definition-fr}

On note $\Gamma(\tilde{G}_V)$ l'ensemble des bonnes classes de conjugaison dans $\tilde{G}_V$. Pour définir les intégrales orbitales anti-spécifiques (qui sont des distributions spécifiques), il faut choisir des mesures invariantes sur ces classes. Cela conduit au $\R_{>0}$-torseur $\dot{\Gamma}(\tilde{G}_V) \to \Gamma(\tilde{G}_V)$. Dans cet article, on fixe un relèvement dans $\dot{\Gamma}(\tilde{G}_V)$ pour chaque élément de $\Gamma(\tilde{G}_V)$. Par conséquent, on regarde $\Gamma(\tilde{G}_V)$ comme une base de $\mathcal{D}(\tilde{G}_V)$.

Remarquons qu'à cause de la dernière condition ci-dessus, $\mathcal{D}(\cdots)$ correspond essentiellement à l'ensemble noté $\mathcal{D}_\text{orb}(\cdots)$ dans \cite[\S 1]{Ar02}. De même, notre $\Gamma(\cdots)$ correspond à $\Gamma_\text{orb}(\cdots)$ dans op.\! cit.

Rappelons aussi qu'il y a des sous-ensembles $\Gamma_\text{unip}(\tilde{G}_V)$ et $\Gamma_\text{ss}(\tilde{G}_V) \supset \Gamma_\text{reg}(\tilde{G}_V)$ des éléments unipotents, semi-simples et semi-simples réguliers, respectivement.

Varions maintenant l'ensemble $V$. On demande que pour toute décomposition $V=V_1 \sqcup V_2$, les produits tensoriels des éléments dans $\Gamma(\tilde{G}_{V_1})$ et $\Gamma(\tilde{G}_{V_2})$ donnent $\Gamma(\tilde{G}_V)$.

\paragraph{Base spectrale}
On note $\mathcal{F}(\tilde{G}_V)$ l'espace des distributions spécifiques invariantes sur $\tilde{G}_V$ engendré par les caractères des représentations spécifiques admissibles de longueur finie. On en fixe une base $\Pi_-(\tilde{G}_V)$, disons l'ensemble des classes d'équivalence des représentations spécifiques irréductibles admissibles de $\tilde{G}_V$.

Comme la base géométrique, $\Pi_-(\tilde{G}_V)$ admet une chaîne de sous-ensembles $\Pi_{\text{unit},-}(\tilde{G}_V) \supset \Pi_{\text{temp},-}(\tilde{G}_V)$ correspondant aux représentations unitaires et tempérées, respectivement.

C'est clair que pour toute décomposition $V=V_1 \sqcup V_2$, les produits tensoriels des éléments dans $\Pi_-(\tilde{G}_{V_1})$ et $\Pi_-(\tilde{G}_{V_2})$ donnent $\Pi_-(\tilde{G}_V)$.

\paragraph{$A_{G,\infty}$-équivariance}
Dans ce paragraphe, on suppose que $V \supset V_\infty$. Donc on a $\tilde{G}_V = \tilde{G}^1_V \times A_{G,\infty}$.

On définit les espaces $\mathcal{D}(\tilde{G}_V, A_{G,\infty})$ et $\mathcal{F}(\tilde{G}_V, A_{G,\infty})$ en considérant les formes linéaires sur $\mathcal{H}_{\asp}(\tilde{G}_V, A_{G,\infty})$. Pour $\mathcal{D}(\tilde{G}_V, A_{G,\infty})$, la condition sur $\Supp(D)$ devient qu'il est une réunion finie de classes de conjugaison modulo $A_{G,\infty}$.

Étant fixées les bases $\Gamma(\tilde{G}_V)$ et $\Pi_-(\tilde{G}_V)$, c'est facile de définir leurs versions équivariantes. Définissons d'abord la base $\Gamma(\tilde{G}_V, A_{G,\infty})$ de $\mathcal{D}(\tilde{G}_V, A_{G,\infty})$. On note $\Gamma(\tilde{G}_V^1)$ le sous-ensemble de $\Gamma(\tilde{G}_V)$ des classes à support dans $\tilde{G}_V^1$, alors $\Gamma(\tilde{G}_V^1)$ forme une base de $\mathcal{D}(\tilde{G}_V, A_{G,\infty})$: pour tous $D \in \Gamma(\tilde{G}^1_V)$ et $f^1 \in \mathcal{H}_{\asp}(\tilde{G}_V, A_{G,\infty})$, on pose $D(f^1) := D(f^{1,b})$ où $f^{1,b}(\cdot) = f^1(\cdot) b(H_{\tilde{G}}(\cdot))$ appartient à $\mathcal{H}_{\asp}(\tilde{G}_V)$, et $b \in C^\infty_c(\mathfrak{a}_G)$ satisfait à $b(0)=1$. C'est bien défini. D'autre part, $\Pi_-(\tilde{G}_V, A_{G,\infty})$ s'obtient en prenant les éléments de $\Pi_-(\tilde{G}_V)$ dont le caractère central sous $A_{G,\infty}$ est trivial.

\paragraph{Desiderata}
On construira dans \S\ref{sec:developpements} deux ensembles
\begin{align*}
  \Gamma(\tilde{M}^1, V) & \subset \Gamma(\tilde{M}_V, A_{M,\infty}), \\
  \Pi_{t,-}(\tilde{M}^1, V) & \subset \Pi_{\text{unit},-}(\tilde{M}_V, A_{M,\infty}) \quad (t \geq 0),
\end{align*}
pour tout $M \in \mathcal{L}(M_0)$, qui servent à indexer les termes intervenant dans la formule des traces invariante pour $\tilde{G}$. L'ensemble $\Pi_{t,-}(\tilde{M}^1, V)$ sera muni d'une mesure.

\section{Espaces de Paley-Wiener}\label{sec:PW}
Dans cette section, on fixe un revêtement $\rev: \tilde{G} \to G(\A)$, un ensemble fini $V$ de places de $F$, d'où se déduit le revêtement local $\tilde{G}_V \to G(F_V)$, et un sous-groupe compact maximal $\tilde{K}_V = \prod_{v \in V} K_v$ de $\tilde{G}_V$. On fixe de plus un sous-groupe de Lévi minimal $M_0$ de $G$ et on suppose que $M_0$ est en bonne position relativement à $K_v$, pour tout $v \in V$. On fixe aussi un ensemble fini $\Gamma$ de représentations irréductibles de $\tilde{K}_V$ (i.e. $\tilde{K}_V$-types).

\subsection{Fonctions test équivariantes}\label{sec:PW-fonctions}
La notion suivante est due à Arthur.

\begin{definition-fr}\label{def:adherence}
  On dit que $V$ satisfait à la condition d'adhérence si soit $V \cap V_\infty \neq \emptyset$, soit les $F_v$, $v \in V$ ont la même caractéristique résiduelle.
\end{definition-fr}
Supposons désormais que $V$ satisfait à la condition d'adhérence. On pose
\begin{align*}
  A_{G,\infty,V} & := \tilde{G}_V \cap A_{G,\infty}, \\
  \mathfrak{a}_{G,V} & := H_G(G(F_V)).
\end{align*}
La condition d'adhérence garantit que
$$ \mathfrak{a}_{G,V} = \begin{cases}
  \mathfrak{a}_G, & \text{si } V \cap V_\infty \neq \emptyset, \\
  \text{un réseau dans } \mathfrak{a}_G, & \text{sinon}.
\end{cases}$$
Lorsque $V \cap V_\infty \neq \emptyset$, la définition de $A_{G,\infty}$ dans \cite[\S 3.4]{Li14a} implique que $H_{\tilde{G}}: A_{G,\infty,V} \rightiso \mathfrak{a}_G$.

On note $i\mathfrak{a}^*_{G,V}$ le dual de Pontryagin de $\mathfrak{a}_{G,V}$ muni de la mesure de Haar duale. Il se réalise comme un quotient de $i\mathfrak{a}^*_G$. L'accouplement $i\mathfrak{a}^*_{G,V} \times \mathfrak{a}_{G,V} \to \C^\times$ peut donc s'écrire comme $e^{\angles{\cdot,\cdot}}$. Lorsque $V$ se réduit à $\{v\}$, on écrit également $\mathfrak{a}_{G,F_v}$ et $i\mathfrak{a}^*_{G,F_v}$.

\begin{definition-fr}
  Une donnée centrale pour $\tilde{G}_V$ est une paire $(A,\mathfrak{a})$ où $A \subset A_{G,\infty,V}$ est un sous-groupe de Lie connexe, $\mathfrak{a} \subset \mathfrak{a}_{G,V}$, tels que $H_G: A \rightiso \mathfrak{a}$. On suppose de plus qu'ils sont munis de mesures de Haar compatibles.
\end{definition-fr}
Notre définition entraîne que $A=\{1\}$ si $V \cap V_{\infty} = \emptyset$. On note
$$ \mathfrak{a}^\perp := \left\{ \lambda \in i\mathfrak{a}^*_{G,V} : e^{\angles{\lambda,\cdot}}=1 \text{ sur } \mathfrak{a} \right\}. $$
C'est muni de la mesure de Haar duale à celle de $\mathfrak{a}_{G,V}/\mathfrak{a}$. On le désignera souvent par $i\mathfrak{a}_{G,V}^* \cap \mathfrak{a}^\perp$ afin de sous-ligner la référence à $G$.

Soit $M$ un sous-groupe de Lévi de $G$, alors les mêmes notations s'y appliquent et il y a un homomorphisme surjectif canonique
\begin{gather}\label{eqn:h_G}
  h_G: \mathfrak{a}_{M,V} \to \mathfrak{a}_{G,V},
\end{gather}
d'où l'injection duale $i\mathfrak{a}_{G,V}^* \hookrightarrow i\mathfrak{a}_{M,V}^*$. Remarquons aussi que $(A,\mathfrak{a})$ sert comme une donnée centrale pour $\tilde{M}_V$.

Soit $N \in \Z_{>0}$. Définissons $\mathcal{H}_{\asp}(\tilde{G}_V, A)_{N,\Gamma}$ comme l'espace des $f \in C_{c,\asp}^\infty(\tilde{G}_V/A)$ telles que
$$ \Supp(f) \subset \{ \tilde{x} \in \tilde{G} : \log\|\tilde{x}\| \leq N \} \cdot A $$
(rappelons que $\|\cdot\|$ est une fonction hauteur choisie) et que les $\tilde{K}_V$-types de l'espace engendré par $f$ sous la translation bilatérale par $\tilde{K}_V$ sont contenus dans $\Gamma$. Cet espace est muni des semi-normes
$$ \|f\|_D = \sup_{\tilde{x}} |Df(\tilde{x})|$$
où $D$ parcourt l'espace des opérateurs différentiels sur $\tilde{G}_{V \cap V_\infty}/A$. On définit les espaces vectoriels topologiques $\mathcal{H}_{\asp}(\tilde{G}_V, A)_\Gamma$ (resp. $\mathcal{H}_{\asp}(\tilde{G}_V, A)$) en prenant la $\varinjlim_N$ (resp. $\varinjlim_{\Gamma,N}$) topologique. On revoie à \cite[p.126 et p.515]{Tr67} pour la définition précise de $\varinjlim$ des espaces localement convexes.

\begin{remarque}\label{rem:topologique-Hecke}
  Du point de vu de l'analyse fonctionnelle, $\mathcal{H}_{\asp}(\tilde{G}_V, A)$ est un espace raisonnable de fonctions test. Effectivement, on peut montrer que $\mathcal{H}_{\asp}(\tilde{G}_V, A)$ est un espace LF nucléaire et son dual fort est aussi complet et nucléaire. De plus, $\mathcal{H}_{\asp}(\tilde{G}_V, A)$ est réflexif. La démonstration est analogue au cas de $C^\infty_c(\tilde{G}_V/A)$, cf. \cite[Chap. 13 et 51]{Tr67} ou \cite[\S 4]{Br61}.
\end{remarque}

\begin{proposition-fr}\label{prop:f-moyenne}
  Soient $(A,\mathfrak{a})$ et $(B,\mathfrak{b})$ des données centrales telles que $A \subset B$, alors l'application linéaire
  \begin{align*}
    \mathcal{H}_{\asp}(\tilde{G}_V, A)_\Gamma & \longrightarrow \mathcal{H}_{\asp}(\tilde{G}_V, B)_\Gamma, \\
    f^A & \longmapsto \left[ f^B(\cdot) = \int_{B/A} f_z(\cdot) \dd z \right]
  \end{align*}
  est continue, surjective et ouverte.

  De plus, cette application est transitive au sens suivant: si $A \subset B \subset C$, alors la composée $f^A \mapsto f^B \mapsto f^C$ est égale à $f^A \mapsto f^C$.
\end{proposition-fr}
\begin{proof}
  C'est une variante de la poussée-en-avant des fonctions test par une submersion. La continuité est facile à vérifier sur $\mathcal{H}_{\asp}(\tilde{G}_V, A)_{N,\Gamma}$, pour chaque $N$. Pour montrer que $f^A \mapsto f^B$ est surjectif et ouvert, il suffit de construire une section continue $\mathcal{H}_{\asp}(\tilde{G}_V, B)_\Gamma \to \mathcal{H}_{\asp}(\tilde{G}_V, A)_\Gamma$. En effet, on peut fixer un homomorphisme continu $\text{pr}: \tilde{G}_V \to B$ tel que $\text{pr}|_B = \identity_B$. Alors une section est donnée par l'application $f^B \mapsto f^B \cdot (\alpha \circ \text{pr})$ avec $\alpha \in C^\infty_c(B/A)$ tel que $\int_{B/A} \alpha(z) \dd z = 1$. Enfin, la transitivité est claire.
\end{proof}

\begin{remarque}
  Si $V \supset V_\infty$ et $A=A_{G,\infty}=A_{G,\infty,V}$, on retrouve l'espace $\mathcal{H}_{\asp}(\tilde{G}_V, A_{G,\infty})$ défini dans \S\ref{sec:fonctions-test}. D'autre part, si $A=\{1\}$, on retrouve l'espace de Hecke usuel $\mathcal{H}_{\asp}(\tilde{G}_V)$. De plus, on retrouve l'application $f \mapsto f^1$ comme un cas spécial de la Proposition \ref{prop:f-moyenne}.
\end{remarque}

\begin{lemme}\label{prop:f-coinvariant}
  L'application $\mathcal{H}_{\asp}(\tilde{G}_V, A)_\Gamma \to \mathcal{H}_{\asp}(\tilde{G}_V, B)_\Gamma$ est l'application co-invariant pour l'action de $B/A$ sur $\mathcal{H}_{\asp}(\tilde{G}_V, A)_\Gamma$ définie par $f^A \mapsto [f^A_z(\cdot) := f^A(z \cdot)]$ où $z \in B/A$.
\end{lemme}
\begin{proof}
  On a vu dans la Proposition \ref{prop:f-moyenne} que $\mathcal{H}_{\asp}(\tilde{G}_V, A)_\Gamma \to \mathcal{H}_{\asp}(\tilde{G}_V, B)_\Gamma$ est une application quotient d'espaces topologiques vectoriels. Vu les propriétés topologiques dans la Remarque \ref{rem:topologique-Hecke}, il revient au même de montrer que l'injection
  $$ \mathcal{H}_{\asp}(\tilde{G}_V, B)'_\Gamma \hookrightarrow \mathcal{H}_{\asp}(\tilde{G}_V, A)'_\Gamma $$
  est l'inclusion des $B/A$-invariants dans $\mathcal{H}_{\asp}(\tilde{G}_V, A)'_\Gamma$, où $(\cdots)'$ signifie le dual muni de la topologie forte. Pour ce faire, on peut reprendre l'argument usuel pour \cite[IV.5, Exemple 1]{S66}.
\end{proof}

\subsection{Espaces de Paley-Wiener équivariants}\label{sec:PW-equiv}
Les constructions dans cette sous-section nécessitent les théorèmes de Paley-Wiener  \S\ref{sec:PW-hyp}. On fixe toujours $V$ satisfaisant à la condition d'adhérence (la Définition \ref{def:adherence}) et une donnée centrale $(A,\mathfrak{a})$. On note
\begin{align*}
  \Pi_{\text{unit},-}(\tilde{G}_V, A) & := \{ \pi \in \Pi_{\text{unit},-}(\tilde{G}_V) : \text{le caractère central de 
$\pi$ est trivial sur } A \}, \\
  \Pi_{\text{temp},-}(\tilde{G}_V, A) & := \Pi_{\text{unit},-}(\tilde{G}_V, A) \cap \Pi_{\text{temp},-}(\tilde{G}_V).
\end{align*}
Pour $f^A \in \mathcal{H}_{\asp}(\tilde{G}_V, A)_\Gamma$ et $\pi \in \Pi_{\text{unit},-}(\tilde{G}_V, A)$, on définit l'opérateur $\pi(f^A) := \int_{\tilde{G}_V/A} f^A(\tilde{x}) \pi(\tilde{x}) \dd\tilde{x}$. On associe à chaque $f^A$ une fonction $\phi: \Pi_{\text{temp},-}(\tilde{G}_V,A) \to \C$ donnée par
$$ \phi(\pi) = \Tr\pi(f^A). $$

On obtient ainsi une application linéaire
$$ \phi_{\tilde{G}}: \mathcal{H}_{\asp}(\tilde{G}_V, A)_\Gamma \to \{\text{fonctions } \Pi_{\text{temp},-}(\tilde{G}_V,A) \to \C \}. $$
Son image est notée $I\mathcal{H}_{\asp}(\tilde{G}_V, A)_\Gamma$. Il y a pourtant une autre interprétation de $\phi_{\tilde{G}}$ que nous préférons. Étant donné $\phi: \Pi_{\text{temp},-}(\tilde{G}_V,A) \to \C$ comme ci-dessus, on considère ses coefficients de Fourier
$$ \phi(\pi,Z) := \int_{\mathfrak{a}^\perp} \phi(\pi_\lambda) e^{-\angles{\lambda,Z}}\dd\lambda, \quad \pi \in \Pi_{\text{temp},-}(\tilde{G}_V,A),\; Z \in \mathfrak{a}_{G,V}/\mathfrak{a}. $$
Il vérifie $\phi(\pi_\lambda,Z)=\phi(\pi,Z) e^{\angles{\lambda,Z}}$. Réciproquement, étant donné $\phi: \Pi_{\text{temp},-}(\tilde{G}_V,A) \times \mathfrak{a}_{G,V}/\mathfrak{a} \to \C$ tel que $\phi(\pi_\lambda,Z)=\phi(\pi,Z) e^{\angles{\lambda,Z}}$, la transformation de Fourier inverse nous permet de retrouver $\phi: \Pi_{\text{temp},-}(\tilde{G}_V,A) \to \C$. L'hypothèse est évidemment que $\lambda \mapsto \phi(\pi_\lambda)$ soit à décroissance rapide de sorte que la transformée de Fourier est bien définie. Cela sera bien le cas pour $\phi \in I\mathcal{H}_{\asp}(\tilde{G}_V, A)_\Gamma$ et nous le regardons désormais comme une fonction sur $\Pi_{\text{temp},-}(\tilde{G}_V,A) \times \mathfrak{a}_{G,V}/\mathfrak{a}$. En fait, on a le résultat suivant.

\begin{proposition-fr}\label{prop:2-variables}
  Soit $f \in \mathcal{H}_{\asp}(\tilde{G}_V, A)_\Gamma$, alors
  $$ \phi_{\tilde{G}}(f^A)(\pi, Z) = \Tr\left( \pi(f^A|_{\tilde{G}_V^{Z+\mathfrak{a}}}) \right), $$
  où on a utilisé la notion \eqref{eqn:G_V^E} et
  $$ \pi(f^A|_{\tilde{G}_V^{Z+\mathfrak{a}}})) := \int_{\tilde{G}^{Z+\mathfrak{a}}_V/A} f^A(\tilde{x}) \dd\tilde{x}. $$
\end{proposition-fr}
C'est une conséquence simple de l'inversion de Fourier. On en démontrera une généralisation aux caractères pondérés dans la Proposition \ref{prop:Fourier-restriction}. 

On va considérer les groupes $\mathcal{M} = \prod_{v \in V} M_v$ où $M_v$ est un sous-groupe de Lévi de $G \times_F F_v$, pour tout $v$. On pose $\mathfrak{a}_{\mathcal{M},V} := \bigoplus_{v \in V} \mathfrak{a}_{M_v, F_v}$; son dual de Pontryagin est noté $i\mathfrak{a}^*_{\mathcal{M},V}$. L'image réciproque de $\prod_{v \in V} M_v(F_v)$ dans $\tilde{G}_V$ est notée $\tilde{\mathcal{M}}$. On a un prolongement canonique $\mathfrak{a} \hookrightarrow \mathfrak{a}_{\mathcal{M},V}$; en effet, il suffit de considérer le cas facile $V \cap V_\infty \neq \emptyset$, car sinon $\mathfrak{a}=\{0\}$. De même, $A$ se plonge canoniquement dans $\tilde{\mathcal{M}}$ comme un sous-groupe central. On dispose toujours de l'application $h_G: \mathfrak{a}_{\mathcal{M},V} \to \mathfrak{a}_{G,V}$ et son dual, et on définit l'ensemble $\Pi_{\text{temp},-}(\tilde{\mathcal{M}}, A)$ de façon évidente.

\begin{definition-fr}\label{def:PW-equiv}
  On définit $\text{PW}_{\asp}(\tilde{G}_V, A)_\Gamma$ comme l'espace des fonctions
  $$ \phi: \Pi_{\text{temp},-}(\tilde{G}_V,A) \times \mathfrak{a}_{G,V}/\mathfrak{a} \to \C $$
  telles que
  \begin{enumerate}
    \item $\phi(\pi_\lambda, Z) = \phi(\pi,Z) e^{\angles{\lambda,Z}}$ pour tout $(\pi,Z)$ et tout $\lambda \in i\mathfrak{a}^*_{G,V}$;
    \item $\phi(\pi,\cdot)=0$ si $\pi$ ne contient aucun $\tilde{K}_V$-type dans $\Gamma$;
    \item pour tout $\mathcal{M}$ et tout $(\sigma,X) \in \Pi_{\text{temp},-}(\tilde{\mathcal{M}}, A) \times \mathfrak{a}_{\mathcal{M},V}/\mathfrak{a}$, on pose
    \begin{gather}\label{eqn:phi-sigma}
      \phi(\sigma,X) := \int_{\frac{i\mathfrak{a}^*_{\mathcal{M},V} \cap \mathfrak{a}^\perp}{i\mathfrak{a}^*_{G,V} \cap \mathfrak{a}^\perp}} \phi(\sigma^G_\Lambda, h_G(X)) e^{-\angles{\Lambda,X}} \dd\Lambda;
    \end{gather}
    alors $\phi(\sigma,\cdot)$ est convergente et définit une fonction dans $C_c^\infty(\mathfrak{a}_{\mathcal{M},V}/\mathfrak{a})$, où $\sigma^G_\Lambda$ est l'induite parabolique normalisée à $\tilde{G}_V$, irréductible pour $\Lambda$ en position générale.
  \end{enumerate}
\end{definition-fr}
On laisse au lecteur le soin de formuler la variante où $\phi$ est regardée comme une fonction $\Pi_{\text{temp},-}(\tilde{G}_V,A) \to \C$.

On munit $\text{PW}_{\asp}(\tilde{G}_V, A)_\Gamma$ de la topologie induite par la topologie de fonctions test sur $C_c^\infty(\mathfrak{a}_{\mathcal{M},V}/\mathfrak{a})$ (celle utilisée par Schwartz lorsque $V \cap V_\infty \neq \emptyset$), à travers des applications $\phi(\sigma,\cdot)$ définies dans \eqref{eqn:phi-sigma}. En prenant $\varinjlim_\Gamma$, on obtient l'espace vectoriel topologique $\text{PW}_{\asp}(\tilde{G}_V, A)$.

\begin{proposition-fr}\label{prop:phi-moyenne}
  Soient $(A,\mathfrak{a})$ et $(B,\mathfrak{b})$ des données centrales telles que $A \subset B$, alors l'application linéaire
  \begin{align*}
    \mathrm{PW}_{\asp}(\tilde{G}_V, A)_\Gamma & \longrightarrow \mathrm{PW}_{\asp}(\tilde{G}_V, B)_\Gamma, \\
    \phi^A & \longmapsto \left[ \phi^B(\pi,Z) = \int_{\mathfrak{b}/\mathfrak{a}} \phi^B(\pi, Z+Y) \dd Y \right]
  \end{align*}
  où $(\pi,Z) \in \Pi_{\mathrm{temp},-}(\tilde{G}_V, B) \times \mathfrak{a}_{G,V}/\mathfrak{b}$, est continue, surjective et ouverte.

  De plus, cette application est transitive au sens suivant: si $A \subset B \subset C$, alors la composée $\phi^A \mapsto \phi^B \mapsto \phi^C$ est égale à $\phi^A \mapsto \phi^C$.
\end{proposition-fr}
\begin{proof}
  L'idée de la preuve est analogue à celle de la Proposition \ref{prop:f-moyenne}. Nous omettons les détails.
\end{proof}

\begin{lemme}\label{prop:phi-coinvariant}
  L'application $\mathrm{PW}_{\asp}(\tilde{G}_V, A)_\Gamma \to \mathrm{PW}_{\asp}(\tilde{G}_V, B)_\Gamma$ est l'application co-invariant pour l'action de $B/A$ sur $\mathrm{PW}_{\asp}(\tilde{G}_V, A)_\Gamma$ définie par
  $$ \phi \mapsto \left[ \phi_z : (\pi,Z) \mapsto \phi(\pi, Z+X) \omega_\pi(z)^{-1} \right], \quad (\pi,Z) \in \Pi_{\mathrm{temp},-}(\tilde{G}_V, A) \times \mathfrak{a}_{G,V}/\mathfrak{a}, $$
  où $z \in B/A$, $X := H_G(z) \in \mathfrak{b}/\mathfrak{a}$ et $\omega_\pi$ est le caractère central de $\pi$.
\end{lemme}
\begin{proof}
  Remarquons que les espaces $\text{PW}_{\asp}(\tilde{G}_V,\cdots)$ satisfont aussi aux propriétés topologiques dans la Remarque \ref{rem:topologique-Hecke}, par les mêmes arguments. La preuve de l'assertion est donc analogue à celle du Lemme \ref{prop:f-coinvariant}.
\end{proof}

\subsection{Hypothèse: théorème de Paley-Wiener invariant}\label{sec:PW-hyp}
Conservons les choix de $(A,\mathfrak{a})$, $\Gamma$, etc. En admettant la théorie de Harish-Chandra sur revêtements, c'est relativement facile de montrer que $\phi_{\tilde{G}}$ induit une application linéaire continue
$$ \phi_{\tilde{G}}: \mathcal{H}_{\asp}(\tilde{G}_V, A)_\Gamma \to \text{PW}_{\asp}(\tilde{G}_V, A)_\Gamma; $$
en particulier, $I\mathcal{H}_{\asp}(\tilde{G}_V, A)_\Gamma \subset \text{PW}_{\asp}(\tilde{G}_V, A)_\Gamma$. Pour le cas archimédien, la théorie de Harish-Chandra s'applique déjà aux revêtements et les arguments sont donnés dans \cite[2.1]{CD84}. Pour le cas non-archimédien, on renvoie à \cite[\S 2]{Li12b}.

\begin{lemme}\label{prop:PW-diagramme}
  Soient $(A,\mathfrak{a})$, $(B,\mathfrak{b})$ deux données centrales telles que $A \subset B$. Alors le diagramme suivant est commutatif
  $$\xymatrix{
    \mathcal{H}_{\asp}(\tilde{G}_V, A) \ar[d] \ar[r]^{\phi^A_{\tilde{G}}} & \mathrm{PW}_{\asp}(\tilde{G}_V, A) \ar[d] \\
    \mathcal{H}_{\asp}(\tilde{G}_V, B) \ar[r]_{\phi^B_{\tilde{G}}} & \mathrm{PW}_{\asp}(\tilde{G}_V, B)
  }$$
  où les flèches verticales sont $f^A \mapsto f^B$ et $\phi^A \to \phi^B$, respectivement. De plus, $\phi^A_{\tilde{G}}$ est équivariante pour les actions de $B/A$ définies dans le Lemme \ref{prop:f-coinvariant}, \ref{prop:phi-coinvariant}.
\end{lemme}
\begin{proof}
  Cela résulte aisément de la Proposition \ref{prop:2-variables}.
\end{proof}

L'hypothèse suivante est cruciale pour la formule des traces invariante. Remarquons que le cas des groupes réductifs connexes avec $A=\{1\}$ a été formulé par Arthur dans \cite[\S 11]{Ar89-IOR1}.
\begin{hypothese}\label{hyp:PW}
  Pour tout $M \in \mathcal{L}(M_0)$, on a $I\mathcal{H}_{\asp}(\tilde{M}_V, A) = \text{PW}_{\asp}(\tilde{M}_V, A)$, et l'application $\phi_{\tilde{M}}$ est une application linéaire surjective ouverte de $\mathcal{H}_{\asp}(\tilde{M}_V, A)_\Gamma$ sur $\text{PW}_{\asp}(\tilde{M}_V, A)_\Gamma$.
\end{hypothese}

Montrons que le cas général se déduit du cas $A=\{1\}$.
\begin{lemme}
  Soient $(A,\mathfrak{a})$, $(B,\mathfrak{b})$ deux données centrales telles que $A \subset B$. Si l'Hypothèse \ref{hyp:PW} est satisfaite pour $A$, alors elle l'est pour $B$.
\end{lemme}
\begin{proof}
  Il suffit de considérer le cas $M=G$. On regarde le diagramme du Lemme \ref{prop:PW-diagramme}. Les flèches verticales sont surjectives et ouvertes, donc $\phi_{\tilde{G}}^B$ est surjective et ouverte si $\phi_{\tilde{G}}^A$ l'est.
\end{proof}
En fait, on voit que $\phi_{\tilde{G}}^B$ est l'application déduite de $\phi_{\tilde{G}}^A$ par le foncteur de co-invariants. 

\subsection{Exemples}\label{sec:PW-exemples}
On a déjà dit qu'il suffit de vérifier l'Hypothèse \ref{hyp:PW} pour le cas $A=\{1\}$. On se ramène facilement au cas où $V=\{v\}$ est un singleton.

Notre hypothèse est vérifiée lorsque $F_v$ est non-archimédien. En effet, il suffit de reprendre les arguments de \cite{BDK86} puisque leurs techniques s'adaptent au cas de revêtements (voir les explications dans \cite[\S 2]{Li12b}).

Lorsque $F_v=\C$, tout revêtement est linéaire et on se ramène au cas traité par \cite{CD84,CD90}. L'hypothèse est donc vérifiée dans ce cas-là.

Le cas $F_v=\R$ est plus délicat car la démonstration du Théorème de Paley-Wiener invariant par Clozel et Delorme \cite{CD84} repose sur certaines propriétés des $K$-types minimaux \cite{Vo79}, qui ne sont connues que pour les groupes de Lie réductifs linéaires ou connexes. Néanmoins, cette difficulté est évitable pour certains revêtements importants. Donnons-en quelques exemples. Écrivons $G$ au lieu de $G \times_F F_v$ pour simplifier les notations.

\renewcommand{\labelenumi}{\Roman{enumi}.}\begin{enumerate}
  \item $G$ est anisotrope modulo le centre.

  Dans ce cas-là, $G$ n'a pas de sous-groupe de Lévi propre et on se ramène à l'analyse harmonique du groupe compact $\tilde{G}^1$ et de $\mathfrak{a}_G$.
  \item $G$ est un tore.

  Cela est un cas particulier du cas précédent.

  \item $\tilde{G} \to G(F_v)$ est un revêtement trivial.

  C'est le cas traité dans \cite{BDK86,CD84,CD90}.

  \item $\tilde{G}=\Mp(2n) \xrightarrow{\rev_n} \Sp(2n,\R)$ est le revêtement métaplectique de Weil \cite{Weil64}.

  Normalement cela signifie le revêtement non-trivial à deux feuillets de $\Sp(2n,\R)$. Ici, on l'augmente en un revêtement à huit feuillets comme dans \cite{Li11} via $\bmu_2 \hookrightarrow \bmu_8$, ce qui est évidemment loisible.  Tout sous-groupe de Lévi admet une décomposition $M = \prod_{i \in I} \GL(n_i) \times \Sp(2n^\flat)$. En fixant un caractère additif non-trivial $\psi: F \to \C^\times$, le revêtement $\rev$ se restreint en
  $$ \tilde{M} = \prod_{i \in I} \GL(n_i,\R) \times \Mp(2n^\flat) \xrightarrow{(\identity, \rev_{n^\flat})} M(\R) = \prod_{i \in I} \GL(n_i,\R) \times \Sp(2n^\flat) $$
  à l'aide du modèle de Schrödinger \cite[\S 5.4]{Li11}.

  Pour vérifier l'Hypothèse \ref{hyp:PW}, on applique le cas des groupes réductifs linéaires aux composantes $\GL(n_i,\R)$; pour la composante $\Mp(2n^\flat)$, on rappelle que c'est un groupe de Lie réductif connexe, donc le Théorème de Paley-Wiener invariant \cite{CD84} est encore valable. Plus généralement, nos hypothèses sont satisfaites pour les groupes dont tous les sous-groupes de Lévi sont un produit de $\R$-points de groupes algébriques avec des groupes de Lie connexes.

  \item $G$ est simplement connexe de rang relatif $1$ sur $\R$, eg. $G=\SL(2)$, et $\tilde{G} \to G(F_v)$ est quelconque.

  En effet, le seul sous-groupe de Lévi propre est un tore, tandis que $G(\R)$ est connexe et $\tilde{G}$ s'écrit comme le produit direct d'un groupe de Lie connexe et un sous-groupe de $\bmu_m$; en tout cas, le Théorème de Paley-Wiener invariant est valable.

  \item $G=U(p,q)$, le groupe unitaire réel de signature $(p,q)$, et $\tilde{G} \to G(F_v)$ est quelconque.

  Tout sous-groupe de Lévi de $G$ est de la forme
  $$ M = \prod_{i \in I} \GL_\C(n_i) \times \U(p',q'), $$
  d'où $M(\R)$ est connexe et on raisonne comme précédemment pour conclure.

  \item $G=\GL(n)$, $\tilde{G} \to G(F_v)$ quelconque.

  Nous avons remarqué que la preuve de Clozel et Delorme utilise des propriétés des $K$-types minimaux afin de décomposer l'induite parabolique normalisée de séries discrètes. Tout d'abord, spécialisons au cas où $\tilde{G}=G(\R)=\GL(n,R)$. La preuve du Théorème de Paley-Wiener invariant se simplifie énormément dans ce cas-là car l'induite parabolique de toute représentation unitaire d'un sous-groupe de Lévi est irréductible: voir par exemple la preuve de Vogan via induction cohomologique \cite[Theorem 17.6]{Vo86}. D'après \cite{Hu90}, la méthode de Vogan s'adapte aussi aux revêtements de $\GL(n,\R)$ donc l'irréductibilité d'induites paraboliques vaut pour $\tilde{G}$ et ses sous-groupes de Lévi (considérons induction par étapes). D'où l'Hypothèse \ref{hyp:PW}.

  \item $G=\GL(n,D)$ ou $\SL(n,D)$, où $D$ est l'algèbre de quaternions non-déployée sur $\R$, et $\tilde{G} \to G(F_v)$ est quelconque.

  Comme dans le cas de $\SL(2)$, il suffit de montrer que tout sous-groupe de Lévi de $G$ est connexe en tant qu'un groupe de Lie. La décomposition de Bruhat nous permet de se ramener au cas d'un sous-groupe de Lévi minimal. Traitons d'abord le cas de $\SL(n,D)$. Un sous-groupe de Lévi minimal $M$, confondu avec ses $\R$-points, est de la forme
  $$ M(\R) = \left\{ (x_i)_{i=1}^n \in  (D^\times)^n : \prod_{i=1}^n N(x_i)=1 \right\}$$
  où $N$ désigne la norme réduite de $D$. Or $D^1 := \Ker(N)$ est connexe et $D^\times = D^1 \times \R_{>0}$, donc $M(\R)$ est aussi connexe. Le cas de $\GL(n,D)$ est similaire.

\end{enumerate}\renewcommand{\labelenumi}{\arabic{enumi}.}

\subsection{Versions à support presque compact}\label{sec:ac}
Conservons les mêmes choix de $V$, $(A,\mathfrak{a})$ et $\Gamma$. On commence par observer que $C^\infty(\mathfrak{a}_{G,V}/\mathfrak{a})$ opère sur $C^\infty_{\asp}(\tilde{G}_V/A)$ par
$$ f \mapsto [f^b(\cdot) = f(\cdot) b(H_{\tilde{G}}(\cdot))], \quad b \in C^\infty(\mathfrak{a}_{G,V}/\mathfrak{a}), \; f \in C^\infty_{\asp}(\tilde{G}_V/A). $$

Cette action préserve $\mathcal{H}_{\asp}(\tilde{G}_V, A)_\Gamma$, ce qui permet de définir l'espace de Hecke à support presque compact (abrégé par ``ac'')
$$ \mathcal{H}_{\text{ac},\asp}(\tilde{G}_V, A)_\Gamma := \left\{f \in C^\infty_{\asp}(\tilde{G}_V/A) : \forall b \in C^\infty_c(\mathfrak{a}_{G,V}/\mathfrak{a}) \text{ on a } f^b \in \mathcal{H}_{\asp}(\tilde{G}_V, A)_\Gamma \right\}. $$
Il est muni de la topologie induite des semi-normes $f \mapsto \|f^b\|$, où $b \in C^\infty_c(\mathfrak{a}_{G,V}/\mathfrak{a})$ et $\|\cdot\|$ parcourt un ensemble de semi-normes définissant la topologie de $\mathcal{H}_{\asp}(\tilde{G}_V, A)_\Gamma$.

De même, $C^\infty(\mathfrak{a}_{G,V}/\mathfrak{a})$ opère sur l'ensemble des fonctions $\Pi_{\text{temp},-}(\tilde{G}_V, A) \times \mathfrak{a}_{G,V}/\mathfrak{a} \to \C$ par
$$ \phi \mapsto [\phi^b(\pi,Z) := \phi(\pi,Z)b(Z) ], \quad b \in C^\infty_{\asp}(\mathfrak{a}_{G,V}/\mathfrak{a}). $$
Comme dans le cas précédent, on définit
$$ \text{PW}_{\text{ac},\asp}(\tilde{G}_V, A)_\Gamma := \left\{ \begin{array}{l}
  \phi: \Pi_{\text{temp},-}(\tilde{G}_V, A) \times \mathfrak{a}_{G,V}/\mathfrak{a} \to \C, \\
  \forall b \in C^\infty_c(\mathfrak{a}_{G,V}/\mathfrak{a}) \text{ on a } \phi^b \in \text{PW}_{\asp}(\tilde{G}_V, A)_\Gamma
\end{array} \right\}.$$
On le munit de la topologie induite par les semi-normes $\phi \mapsto \|\phi^b\|$, où $b \in C^\infty_c(\mathfrak{a}_{G,V}/\mathfrak{a})$ et $\|\cdot\|$ parcourt un ensemble de semi-normes définissant la topologie de $\text{PW}_{\asp}(\tilde{G}_V, A)_\Gamma$.

On a $\mathcal{H}_{\text{ac},\asp}(\tilde{G}_V, A)_\Gamma \supset \mathcal{H}_{\asp}(\tilde{G}_V, A)_\Gamma$ et $\text{PW}_{\text{ac},\asp}(\tilde{G}_V, A)_\Gamma \supset \text{PW}_{\asp}(\tilde{G}_V, A)_\Gamma$. Passage à la $\varinjlim_\Gamma$ permet de définir leur avatar sans l'indice $\Gamma$.

Vu la Proposition \ref{prop:2-variables}, l'application $\phi_{\tilde{G}}$ se prolonge canoniquement en une application de $\mathcal{H}_{\text{ac},\asp}(\tilde{G}_V,A)$ sur l'ensemble des fonctions $\Pi_{\text{temp},-}(\tilde{G}_V, A) \times \mathfrak{a}_{G,V}/\mathfrak{a} \to \C$. On note son image par $I\mathcal{H}_{\text{ac},\asp}(\tilde{G}_V, A)$. Comme dans le cas des fonctions à support compact, $\phi_{\tilde{G}}$ donne toujours une application linéaire continue
$$ \phi_{\tilde{G}}: \mathcal{H}_{\text{ac},\asp}(\tilde{G}_V,A)_\Gamma \longrightarrow \text{PW}_{\text{ac},\asp}(\tilde{G}_V,A)_\Gamma . $$

\begin{proposition-fr}
  Si l'Hypothèse \ref{hyp:PW} est satisfaite, alors $I\mathcal{H}_{\mathrm{ac},\asp}(\tilde{G}_V,A)_\Gamma = \text{PW}_{\mathrm{ac},\asp}(\tilde{G}_V,A)_\Gamma$; de plus, $\phi_{\tilde{G}}$ est une application linéaire continue surjective et ouverte de $\mathcal{H}_{\mathrm{ac},\asp}(\tilde{G}_V,A)_\Gamma$ sur $\text{PW}_{\mathrm{ac},\asp}(\tilde{G}_V,A)_\Gamma$.
\end{proposition-fr}
\begin{proof}
  Cela résulte aisément de l'Hypothèse \ref{hyp:PW} à l'aide des multiplicateurs $b \in C_c^\infty(\mathfrak{a}_{G,V}/\mathfrak{a})$ et des définitions des espaces à support presque compact. Cf. \cite[p.329]{Ar88-TF1}.
\end{proof}

\begin{remarque}
  Si $V \cap V_\infty \neq \emptyset$ et $A=A_{G,V,\infty}$, alors $\mathfrak{a}_{G,V}/\mathfrak{a}=\{0\}$ et on a $\mathcal{H}_{\text{ac},\asp}(\tilde{G}_V,A)_\Gamma=\mathcal{H}_{\asp}(\tilde{G}_V,A)_\Gamma$.
\end{remarque}

Introduisons une opération importante sur ces espaces à support presque compact. Soient $(A,\mathfrak{a})$, $(B,\mathfrak{b})$ deux données centrales telles que $A \subset B$. On a l'inclusion naturelle
$$ \mathcal{H}_{\text{ac},\asp}(\tilde{G}_V, B)_\Gamma \hookrightarrow \mathcal{H}_{\text{ac},\asp}(\tilde{G}_V, A)_\Gamma. $$

D'autre part, on fixe une section de $\mathfrak{a}_{G,V}/\mathfrak{a} \twoheadrightarrow \mathfrak{a}_{G,V}/\mathfrak{b}$, ce qui existe: en effet, l'existence est tautologique sauf si $V \cap V_\infty \neq \emptyset$, mais dans ce cas-là tout espace en vue est un $\R$-espace vectoriel. Notons-la
$$ s: \mathfrak{a}_{G,V}/\mathfrak{b} \hookrightarrow \mathfrak{a}_{G,V}/\mathfrak{a}. $$

Cela permet de définir une application injective
\begin{gather}\label{eqn:res-inclusion}
  \Pi_{\text{temp},-}(\tilde{G}_V,B) \times \mathfrak{a}_{G,V}/\mathfrak{b} \longrightarrow \Pi_{\text{temp},-}(\tilde{G}_V,A) \times \mathfrak{a}_{G,V}/\mathfrak{a}.
\end{gather}

Soit $\phi: \Pi_{\text{temp},-}(\tilde{G}_V,A) \times \mathfrak{a}_{G,V}/\mathfrak{a} \to \C$ une fonction quelconque, on en déduit une fonction $\text{res}(\phi): \Pi_{\text{temp},-}(\tilde{G}_V,B) \times \mathfrak{a}_{G,V}/\mathfrak{b} \to \C$ grâce à \eqref{eqn:res-inclusion}. L'observation simple suivante nous sera utile.

\begin{lemme}\label{prop:G-M-V-observation}
  Soit $M \in \mathcal{L}(M_0)$, on a l'isomorphisme
  $$ \frac{i\mathfrak{a}^*_{M,V} \cap \mathfrak{a}^\perp}{i\mathfrak{a}^*_{G,V} \cap \mathfrak{a}^\perp} \rightiso \frac{i\mathfrak{a}_{M,V}^*}{i\mathfrak{a}^*_{G,V}} $$
  qui préserve les mesures de Haar.

  Idem si l'on remplace $M$ par $\mathcal{M} = \prod_{v \in V} M_v$ défini dans \S\ref{sec:PW-equiv}.
\end{lemme}
\begin{proof}
  On part du diagramme commutatif des groupes topologiques localement compacts commutatifs.
  $$\xymatrix{
    0 \ar[r] & \mathfrak{a} \ar[r] & \mathfrak{a}_{M,V} \ar[r] \ar[d]^{h_G} & \mathfrak{a}_{M,V}/\mathfrak{a} \ar[r] \ar[d] & 0 \\
    0 \ar[r] & \mathfrak{a} \ar[r] \ar@{=}[u] & \mathfrak{a}_{G,V} \ar[r] & \mathfrak{a}_{G,V}/\mathfrak{a} \ar[r] & 0
  }$$
  Les lignes sont des suites exactes courtes qui respectent les mesures de Haar. En dualisant, on obtient le diagramme avec mesures
  $$\xymatrix{
    0 \ar[r] & i\mathfrak{a}^*_{M,V} \cap \mathfrak{a}^\perp \ar[r] & i\mathfrak{a}^*_{M,V} \ar[r] & \mathfrak{a}^* \ar[r] & 0 \\
    0 \ar[r] & i\mathfrak{a}^*_{G,V} \cap \mathfrak{a}^\perp \ar[r] \ar@{^{(}->}[u] & i\mathfrak{a}^*_{G,V} \ar[r] \ar@{^{(}->}[u] & \mathfrak{a}^* \ar[r] \ar@{=}[u] & 0
  }$$
  où $\mathfrak{a}^*$ signifie le dual de Pontryagin de $\mathfrak{a}$. Pour conclure, il suffit d'appliquer le lemme du serpent et noter le rapport entre les mesures de Haar.

  Le cas pour $\mathcal{M} = \prod_{v \in V} M_v$ est analogue. On doit rappeler les définitions des plongements $\mathfrak{a} \to \mathfrak{a}_{\mathcal{M},V}$ et $\mathfrak{a} \to \mathfrak{a}_{G,V}$ pour montrer que les diagrammes sont commutatifs.
\end{proof}

\begin{remarque}\label{rem:G-M-V-observation}
  Plus généralement, il en résulte que pour $\mathfrak{a} \subset \mathfrak{b}$, on a un isomorphisme canonique
  $$ \frac{i\mathfrak{a}^*_{M,V} \cap \mathfrak{b}^\perp}{i\mathfrak{a}^*_{G,V} \cap \mathfrak{b}^\perp} \rightiso \frac{i\mathfrak{a}_{M,V}^* \cap \mathfrak{a}^\perp}{i\mathfrak{a}^*_{G,V} \cap \mathfrak{a}^\perp} $$
  qui préserve les mesures de Haar.

  On peut également choisir les sections $s'$, $s$ de $\mathfrak{a}_{M,V}/\mathfrak{a} \to \mathfrak{a}_{M,V}/\mathfrak{b}$ et $\mathfrak{a}_{G,V}/\mathfrak{a} \to \mathfrak{a}_{G,V}/\mathfrak{b}$, respectivement, qui rendent commutatif le diagramme suivant.
  $$\xymatrix{
    \mathfrak{a}_{M,V}/\mathfrak{b} \ar[d]_{h_G} \ar[r]^{s'} & \mathfrak{a}_{M,V}/\mathfrak{a} \ar[d]^{h_G} \\
    \mathfrak{a}_{G,V}/\mathfrak{b} \ar[r]_{s} & \mathfrak{a}_{G,V}/\mathfrak{a}
  }$$
  De façon duale:
  $$\xymatrix{
    i\mathfrak{a}^*_{M,V} \cap \mathfrak{b}^\perp & i\mathfrak{a}^*_{M,V} \cap \mathfrak{a}^\perp \ar[l]_{s'^*} \\
    i\mathfrak{a}^*_{G,V} \cap \mathfrak{b}^\perp \ar[u] & i\mathfrak{a}^*_{G,V} \cap \mathfrak{a}^\perp \ar[u] \ar[l]^{s^*}
  }.$$

  Alors $s'^*$ et $s^*$ induisent une application surjective
  $$ \frac{i\mathfrak{a}^*_{M,V} \cap \mathfrak{a}^\perp}{i\mathfrak{a}^*_{G,V} \cap \mathfrak{a}^\perp} \longrightarrow \frac{i\mathfrak{a}^*_{M,V} \cap \mathfrak{b}^\perp}{i\mathfrak{a}^*_{G,V} \cap \mathfrak{b}^\perp} $$
  qui est un isomorphisme préservant les mesures de Haar. En effet, c'est l'inverse de l'isomorphisme précédent.

  Idem si l'on remplace $M$ par $\mathcal{M}$ comme dans ledit lemme.
\end{remarque}

\begin{proposition-fr}\label{prop:res-diagramme}
  L'application $\mathrm{res}$ induit une application linéaire continue $\mathrm{PW}_{\mathrm{ac},\asp}(\tilde{G}_V,A)_\Gamma \to \mathrm{PW}_{\mathrm{ac},\asp}(\tilde{G}_V,B)_\Gamma$. Le diagramme suivant est commutatif.
  $$\xymatrix{
    \mathcal{H}_{\mathrm{ac},\asp}(\tilde{G}_V, A)_\Gamma \ar[r]^{\phi_{\tilde{G}}^A} & \mathrm{PW}_{\mathrm{ac},\asp}(\tilde{G}_V, A)_\Gamma \ar[d]^{\mathrm{res}} \\
    \mathcal{H}_{\mathrm{ac},\asp}(\tilde{G}_V, B)_\Gamma \ar[r]_{\phi_{\tilde{G}}^B} \ar@{^{(}->}[u] & \mathrm{PW}_{\mathrm{ac},\asp}(\tilde{G}_V, B)_\Gamma
  }$$
\end{proposition-fr}
\begin{proof}
  Les assertions sont des tautologies sauf si $V \cap V_\infty \neq \emptyset$, ce que l'on suppose désormais. Vérifions la première assertion. Soit $\phi \in \text{PW}_{\text{ac},\asp}(\tilde{G}_V, A)_\Gamma$. Pour montrer que $\text{res}(\phi) \in \text{PW}_{\text{ac},\asp}(\tilde{G}_V,B)_\Gamma$, on observe d'abord que
  $$ \text{res}(\phi)(\pi_\lambda, Z) = \phi(\pi_\lambda, s(Z)) = \phi(\pi, s(Z)) e^{\angles{\iota(\lambda), s(Z)}}, $$
  où $\pi \in \Pi_{\text{temp},-}(\tilde{G}_V, B)$, $\lambda \in i\mathfrak{a}^*_{G,V} \cap \mathfrak{b}^\perp$, et $\iota$ est duale à la projection $\mathfrak{a}_{G,V}/\mathfrak{a} \to \mathfrak{a}_{G,V}/\mathfrak{b}$. Or cela entraîne que $e^{\angles{\iota(\lambda), s(Z)}}=e^{\angles{s^* \iota(\lambda),Z}} = e^{\angles{\lambda,Z}}$ car $s$ est une section de  $\mathfrak{a}_{G,V}/\mathfrak{a} \to \mathfrak{a}_{G,V}/\mathfrak{b}$. Donc $\text{res}(\phi)$ satisfait à la première condition de la Définition \ref{def:PW-equiv}. La deuxième condition est évidemment satisfaite. Considérons la troisième condition dans ladite définition. Soit $(\mathcal{M},\sigma)$ comme dans \eqref{eqn:phi-sigma} relativement à la donnée centrale $(B,\mathfrak{b})$. On prolonge $s$ en une section $s'$ de $\mathfrak{a}_{\mathcal{M},V}/\mathfrak{a} \twoheadrightarrow \mathfrak{a}_{\mathcal{M},V}/\mathfrak{b}$ et on obtient les homomorphismes duaux $s'^*, s^*$ comme dans la Remarque \ref{rem:G-M-V-observation}.

  Si l'on note par $\iota'$ l'inclusion naturelle $i\mathfrak{a}^*_{\mathcal{M},V} \cap \mathfrak{b}^\perp \hookrightarrow i\mathfrak{a}^*_{\mathcal{M},V} \cap \mathfrak{a}^\perp$, alors $s'^* \circ \iota' = \identity$. On a
  \begin{align*}
    \text{res}(\phi)(\sigma,X) &= \int_{\frac{i\mathfrak{a}^*_{\mathcal{M},V} \cap \mathfrak{b}^\perp}{i\mathfrak{a}^*_{G,V} \cap \mathfrak{b}^\perp}} \phi(\sigma^G_\Lambda, s(h_G(X))) e^{-\angles{\Lambda, X}} \dd\Lambda \\
    & = \int_{\frac{i\mathfrak{a}^*_{\mathcal{M},V} \cap \mathfrak{a}^\perp}{i\mathfrak{a}^*_{G,V} \cap \mathfrak{a}^\perp}} \phi(\sigma^G_{s'^* \Lambda}, s(h_G(X))) e^{-\angles{s'^*\Lambda, X}} \dd\Lambda \\
    & = \int_{\frac{i\mathfrak{a}^*_{\mathcal{M},V} \cap \mathfrak{a}^\perp}{i\mathfrak{a}^*_{G,V} \cap \mathfrak{a}^\perp}} \phi(\sigma^G_{s'^* \Lambda}, h_G(s'(X))) e^{-\angles{\Lambda, s'(X)}} \dd\Lambda
  \end{align*}
  pour tout $\sigma \in \Pi_{\text{temp},-}(\tilde{\mathcal{M}},B)$ et tout $X \in \mathfrak{a}_{G,V}/\mathfrak{b}$.

  L'expression à intégrer ne dépend que de la classe de $\Lambda$ modulo $i\mathfrak{a}^*_{G,V} \cap \mathfrak{a}^\perp$. Or le Lemme \ref{prop:G-M-V-observation} appliqué à $\mathfrak{a}$ et $\mathfrak{b}$ entraîne que l'on peut toujours prendre un représentant dans $i\mathfrak{a}^*_{G,V} \cap \mathfrak{b}^\perp$ de la classe de $\Lambda$; c'est-à-dire $\Lambda$ appartient à l'image de $\iota'$. Donc
  $$ \text{res}(\phi)(\sigma,X) = \int_{\frac{i\mathfrak{a}^*_{\mathcal{M},V} \cap \mathfrak{a}^\perp}{i\mathfrak{a}^*_{G,V} \cap \mathfrak{a}^\perp}} \phi(\sigma^G_\Lambda, h_G(s'(X))) e^{-\angles{\Lambda, s'(X)}} \dd\Lambda = \phi(\sigma, s'(X)). $$
  Cela prouve $\text{res}(\phi) \in \text{PW}_{\text{ac},\asp}(\tilde{G}_V, B)$, ainsi que la continuité de $\text{res}$.

  La commutativité du diagramme résulte de la Proposition \ref{prop:2-variables} et du fait que $\tilde{G}_V^A/A \rightiso \tilde{G}_V^B/B$. On démontrera en détails une généralisation pour les caractères pondérés dans la Proposition \ref{prop:J-res}.
\end{proof}

\subsection{Propriétés de distributions}
On fixe toujours une donnée centrale $(A,\mathfrak{a})$ pour $\tilde{G}_V$. Donnons quelques notions utiles.

\begin{definition-fr}
  Rappelons que $\mathcal{H}_{\asp}(\tilde{G}_V)$ est un $\mathcal{H}_{\asp}(\tilde{G}_V)$-bi-module sous convolution. On dit qu'une forme linéaire $I: \mathcal{H}_{\asp}(\tilde{G}_V, A) \to \C$ est invariante si
  $$ D(h*f) = D(f*h) $$
  pour tous $h \in \mathcal{H}_{\asp}(\tilde{G}_V)$ et $f \in \mathcal{H}_{\asp}(\tilde{G}_V, A)$. Cela généralise la Définition \ref{def:dist-inv}.
\end{definition-fr}

\begin{definition-fr}
  Soient $W$ un $\C$-espace vectoriel et $Z \in \mathfrak{a}_{G,V}/\mathfrak{a}$. On dit qu'une application linéaire $J: \mathcal{H}_{\asp}(\tilde{G}_V, A) \to W$ est concentrée en $Z$ si pour tout $b \in C^\infty(\mathfrak{a}_{G,V}/\mathfrak{a})$ tel que $b(Z)=1$, on a
  $$ J(f^b) = J(f), \quad f \in \mathcal{H}_{\asp}(\tilde{G}_V, A). $$

  Idem pour les espaces $\mathcal{H}_{\text{ac},\asp}(\tilde{G}_V, A)$, $\text{PW}_{\asp}(\tilde{G}_V,A)$, $\text{PW}_{\text{ac},\asp}(\tilde{G}_V,A)$ ou leurs variantes avec l'indice $\Gamma$.
\end{definition-fr}

\begin{proposition-fr}\label{prop:concentration}
  Soit $J: \mathcal{H}_{\asp}(\tilde{G}_V, A) \to W$ une application linéaire concentrée en $Z$. Alors $J$ se prolonge canoniquement à $\mathcal{H}_{\mathrm{ac},\asp}(\tilde{G}_V,A)$ par $J(f) = J(f^b)$ en choisissant $b \in C^\infty_c(\mathfrak{a}_{G,V}/\mathfrak{a})$ tel que $b(Z)=1$. Idem si l'on remplace $\mathcal{H}_{\asp}(\tilde{G}_V,A)$ , $\mathcal{H}_{\mathrm{ac},\asp}(\tilde{G}_V,A)$ par $\mathrm{PW}_{\asp}(\tilde{G}_V,A)$, $\mathrm{PW}_{\mathrm{ac},\asp}(\tilde{G}_V,A)$ ou leurs variantes avec l'indice $\Gamma$, respectivement.

  De plus, si $W$ est un espace vectoriel topologique et $J$ est continue, alors l'application prolongée est aussi continue.
\end{proposition-fr}
\begin{proof}
  Pour montrer que le prolongement est bien défini, il suffit d'observer que $J(f^b)=J(f^{bb'})=J(f^{b'})$ pour tous $b,b' \in C^\infty_c(\mathfrak{a}_{G,V}/\mathfrak{a})$ qui valent $1$ en $Z$. L'assertion de continuité découle du fait que $f \mapsto f^b$ est continue.
\end{proof}

D'après la Proposition \ref{prop:2-variables}, l'application $f \mapsto \phi_{\tilde{G}}(\cdot,Z)$ est un exemple de telles applications linéaires. Les intégrales orbitales le sont aussi; on les discutera dans la section prochaine.

\begin{definition-fr}
  On fixe une application linéaire continue $\phi: W \to W'$ entre des espaces vectoriels topologiques. On dit qu'une application linéaire $J: W \to W_1$ est supportée par $W'$ si $J|_{\Ker(\phi)}=0$.

  Si $\phi$ est surjective, alors $J$ est supportée par $W'$ si et seulement si $J$ se factorise par une application linéaire $W' \to W_1$; si cette condition est satisfaite, on désigne encore par $J$ l'application factorisée $W' \to W_1$. De plus, si $\phi$ est surjective et ouverte alors l'application factorisée est continue si $J$ l'est.
\end{definition-fr}

On s'intéressera principalement aux cas où $\phi$ est $\phi_{\tilde{G}}$ ou l'un de ses variantes. Dans ces cas-là, $\phi$ sera toujours surjective et ouverte.

\begin{lemme}\label{prop:lemma-support}
  Soient $(A,\mathfrak{a})$ et $(B,\mathfrak{b})$ deux données centrales telles que $A \subset B$. Soit $I$ une forme linéaire continue de $\mathcal{H}_{\asp}(\tilde{G}_V, B)$. On désigne par $I_A$ la forme linéaire continue $f^A \mapsto I(f^B)$, où $f^A \in \mathcal{H}_{\asp}(\tilde{G}_V, A)$ et $f^A \mapsto f^B$ est l'application de la Proposition \ref{prop:f-moyenne}.

  Si $I_A$ est supportée par $\mathrm{PW}_{\asp}(\tilde{G}_V, A)$, alors $I$ est supportée par $\mathrm{PW}_{\asp}(\tilde{G}_V, B)$.
\end{lemme}
\begin{proof}
  On regarde $I_A$ comme une forme linéaire continue de $\text{PW}_{\asp}(\tilde{G}_V, A)$. Rappelons que $\mathcal{H}_{\asp}(\tilde{G}_V, A) \to \text{PW}_{\asp}(\tilde{G}_V, A)$ est $B/A$-équivariante. Donc $I_A$ est $B/A$-invariante. Vu le Lemme \ref{prop:phi-coinvariant}, $I_A$ se factorise par $\text{PW}_{\asp}(\tilde{G}_V, B)$. Or cela implique que $I$ est nulle sur le noyau de $\mathcal{H}_{\asp}(\tilde{G}_V, B) \to \text{PW}_{\asp}(\tilde{G}_V, B)$ car $f^A \mapsto f^B$ est surjective.
\end{proof}

\section{Distributions locales invariantes}\label{sec:dists-locales}
Dans cette section, on fixe un ensemble fini $V$ de places de $F$ qui satisfait à la condition d'adhérence dans la Définition \ref{def:adherence}. On étudie un revêtement fini $\rev: \tilde{G}_V \to G(F_V)$ et on fixe une donnée centrale $(A,\mathfrak{a})$ pour $\tilde{G}_V$. Les données auxiliaires $M_0$, $\tilde{K}_V$ et les groupes $\mathfrak{a}_{G,V}$, $i\mathfrak{a}_{G,V}^*$, $\mathfrak{a}^\perp$, etc. sont définies dans \S\ref{sec:PW}.

\subsection{Caractères pondérés unitaires}
Cette sous-section est basée sur \cite[\S 5.7]{Li12b} qui reprend \cite{Ar98}. Remarquons que les caractères pondérés non-unitaires seront indispensable si l'on envisage la stabilisation, mais leur définition est assez pénible à écrire. Nous espérons l'inclure dans un article ultérieur.

\paragraph{Les $(G,M)$-familles spectrales}
Soient $M \in \mathcal{L}(M_0)$, $\pi \in \Pi_{\text{unit},-}(\tilde{M}_V)$, $\lambda \in \mathfrak{a}^*_{M,\C}$. Rappelons que dans \cite[\S 2]{Li12b}, on a défini les opérateurs d'entrelacement
$$ J_{Q|P}(\pi_\lambda): \mathcal{I}_{\tilde{P}}(\pi_\lambda) \to \mathcal{I}_{\tilde{Q}}(\pi_\lambda), \quad P,Q \in \mathcal{P}(M), $$
et les fonctions $\mu$ de Harish-Chandra pour les revêtements, qui sont de la forme
$$ \mu(\pi_\lambda) = \prod_{\alpha \in \Sigma^{\text{red}}_M/\pm 1} \mu_\alpha(\pi_\lambda), \quad \pi \in \Pi_{\text{unit},-}(\tilde{M}_V), $$
où $\Sigma^{\text{red}}_M$ désigne l'ensemble des racines réduites restreintes à $\mathfrak{a}_M$. Ils sont méromorphes en $\lambda$ et se factorisent en des objets définis sur corps locaux selon $\pi = \bigotimes_{v \in V} \pi_v$. De plus, $J_{Q|P}(\pi_\lambda)$ est holomorphe pour $\Re(\lambda)$ assez positif relativement à la chambre associée à $P$. Comme d'habitude, l'induite parabolique normalisée $\mathcal{I}_{\tilde{P}}(\pi_\lambda)$ se réalise sur un espace vectoriel qui ne dépend pas de $\lambda$; l'action de $\tilde{K}_V$ est aussi indépendant de $\lambda$.

Pour $P,Q \in \mathcal{P}(M)$, on pose
\begin{gather}\label{eqn:mu_QP}
  \mu_{Q|P}(\pi_\lambda) := \prod_{\alpha \in \Sigma_P^\text{red} \cap \Sigma_{\bar{Q}}^\text{red}} \mu_\alpha(\pi_\lambda).
\end{gather}
Il ne dépend que de $\{\angles{\lambda,\alpha^\vee} : \alpha \in \Delta_P^\text{red} \cap \Sigma_{\bar{Q}}^\text{red}\}$.

Choisissons $P \in \mathcal{P}(M)$ et introduisons les $(G,M)$-familles suivantes, méromorphes en les variables $\lambda$ et $\Lambda$:
\begin{align*}
  \mu_Q(\Lambda, \pi_\lambda, \tilde{P}) & := \mu_{Q|P}(\pi_\lambda)^{-1} \mu_{Q|P}\left(\pi_{\lambda+\frac{\Lambda}{2}} \right), \\
  \mathcal{J}(\Lambda,\pi_\lambda,\tilde{P}) & := J_{Q|P}(\pi_\lambda)^{-1} J_{Q|P}(\pi_{\lambda+\Lambda}), \\
  \mathcal{M}(\Lambda,\pi_\lambda,\tilde{P}) & := \mu_Q(\Lambda,\pi_\lambda,\tilde{P}) \mathcal{J}(\Lambda,\pi_\lambda,\tilde{P}), \quad \lambda, \Lambda \in \mathfrak{a}^*_{M,\C}.
\end{align*}

La théorie générale des $(G,M)$-familles fournit les opérateurs $\mathcal{M}_M(\pi_\lambda, \tilde{P})$. Indiquons que $(\mu_Q(\Lambda, \pi_\lambda, \tilde{P}))_{Q \in \mathcal{P}(M)}$ est une $(G,M)$-famille dite radicielle, ce qui est étudiée dans \cite[\S 7]{Ar82-Eis2}.

\paragraph{Caractères pondérés}
\begin{definition-fr}
  Soient $f \in \mathcal{H}_{\asp}(\tilde{G}_V, A)$, $\pi \in \Pi_{\text{unit},-}(\tilde{M}_V, A)$. Le caractère pondéré canoniquement normalisé de $\pi$ est défini comme
  \begin{gather}
    J_{\tilde{M}}(\pi, f) := \Tr\left( \mathcal{M}_M(\pi, \tilde{P}) \mathcal{I}_{\tilde{P}}(\pi, f) \right).
  \end{gather}
  Ici, $\mathcal{I}_{\tilde{P}}(\pi, f)$ est défini comme $\int_{\tilde{G}_V/A} f(\tilde{x})\mathcal{I}_{\tilde{P}}(\pi,\tilde{x}) \dd\tilde{x}$.
\end{definition-fr}
La trace ci-dessus se calcule sur un espace vectoriel de dimension finie, car nous avons supposé que $f$ est $\tilde{K}_V$-finie par translation bilatérale. A priori, $J_{\tilde{M}}(\pi, f)$ est défini pour $\pi$ en position générale dans une $i\mathfrak{a}^*_{M,V}$-orbite, et le choix de $P$ y intervient. La définition est justifiée par le résultat suivant.

\begin{theoreme}\label{prop:caractere-pondere-prop}
  Les caractères pondérés ci-dessus sont bien définis indépendamment du choix de $P$. L'application $\lambda \mapsto J_{\tilde{M}}(\pi_\lambda, f)$ où $\lambda \in i\mathfrak{a}^*_{M,V}$ est analytique et à décroissance rapide.
\end{theoreme}
\begin{proof}
  L'analycité est démontrée dans \cite[Proposition 2.3]{Ar98}. L'assertion de décroissance résulte du fait qu'en tant que fonctions en $\lambda$, c'est connu que les coefficients matriciels de $\mathcal{I}_{\tilde{P}}(\pi_\lambda, f)$ sont à décroissance rapide, tandis que ceux $\mathcal{M}_{\tilde{M}}(\pi_\lambda,\tilde{P})$ sont à croissance modérée. En effet, cela est déjà démontré dans \cite[Lemma 2.1]{Ar94} si l'on utilise l'opérateur $\mathcal{R}_{\tilde{M}}(\pi_\lambda,\tilde{P})$ défini par des facteurs normalisants au lieu de $\mathcal{M}_{\tilde{M}}(\pi_\lambda,\tilde{P})$. Pour passer à $\mathcal{M}_{\tilde{M}}(\pi_\lambda,\tilde{P})$, on aura besoin des facteurs $r^G_M(\pi_\lambda)$ et leurs propriétés analytiques comme dans la preuve \cite[Lemma 3.1]{Ar94}; on en reproduira la preuve dans le Lemme \ref{prop:majoration-r^G_M}.
\end{proof}
Lorsque $M=G$, on retrouve les caractères usuels.

\begin{remarque}
  On se limite au cas où $\pi$ est unitaire. Le cas général est considérablement plus compliqué, cf. \cite[\S 7]{Ar89-IOR1}. Il faut aussi prendre garde qu'Arthur utilise les facteurs normalisants pour définir les caractères pondérés dans \cite{Ar89-IOR1}, tandis que nous utilisons la version ``canoniquement normalisée'' à l'aide de fonctions $\mu$. Cependant, on peut facilement passer de l'un à l'autre grâce à \cite[Lemma 2.1 et Corollary 2.4]{Ar98}.
\end{remarque}

On considère les coefficients de Fourier des caractères pondérés comme suit.
\begin{definition-fr}
  Soit $f \in \mathcal{H}_{\asp}(\tilde{G}_V, A)$. Vu la décroissance rapide de $\lambda \mapsto J_{\tilde{M}}(\pi_\lambda, f)$, on peut poser
  $$ J_{\tilde{M}}(\pi,X,f) := \int_{i\mathfrak{a}^*_{M,V} \cap \mathfrak{a}^\perp} J_{\tilde{M}}(\pi_\lambda, f) e^{-\angles{\lambda, X}} \dd\lambda, \quad X \in \mathfrak{a}_{M,V}/\mathfrak{a} . $$
\end{definition-fr}
On voit que $J_{\tilde{M}}(\pi_\lambda,X,\cdot) = J_{\tilde{M}}(\pi,X,\cdot) e^{\angles{\lambda,X}}$.

\begin{proposition-fr}\label{prop:Fourier-restriction}
  Soient $f^A \in \mathcal{H}_{\asp}(\tilde{G}_V, A)$, $\pi \in \Pi_{\mathrm{unit},-}(\tilde{M}_V, A)$ et $X \in \mathfrak{a}_{M,V}/\mathfrak{a}$. On a
  $$ J_{\tilde{M}}(\pi,X,f^A) = \int_{\frac{i\mathfrak{a}^*_{M,V} \cap \mathfrak{a}^\perp}{i\mathfrak{a}^*_{G,V} \cap \mathfrak{a}^\perp}} \Tr \left( \mathcal{M}_M(\tilde{\pi}_\lambda, \tilde{P}) \mathcal{I}_{\tilde{P}}(\pi_\lambda, f^A|_{\tilde{G}^{h_G(X)+\mathfrak{a}}_V/A}) \right) e^{-\angles{\lambda,X}} \dd\lambda, $$
  où $P \in \mathcal{P}(M)$ et
  $$ \mathcal{I}_{\tilde{P}}\left( \pi_\lambda, f^A|_{\tilde{G}^{h_G(X)+\mathfrak{a}}_V/A} \right) := \int_{\tilde{G}^{h_G(X)+\mathfrak{a}}_V/A} f^A(\tilde{x}) \mathcal{I}_{\tilde{P}}(\pi_\lambda, \tilde{x})\dd\tilde{x}. $$
\end{proposition-fr}
\begin{proof}
  En déroulant les définitions, $J_{\tilde{M}}(\pi, X, f^A)$ est égal à
  $$ \int_{i\mathfrak{a}^*_{G,V} \cap \mathfrak{a}^\perp} \; \Tr \int_{\tilde{G}_V/A} f(\tilde{x}) \mathcal{M}_M(\pi_\lambda, \tilde{P}) \mathcal{I}_{\tilde{P}}(\pi_\lambda, \tilde{x}) e^{-\angles{\lambda,X}} \dd\tilde{x}\dd\lambda. $$

  Lorsque $\lambda$ est remplacé par $\lambda+\mu$ où $\mu \in i\mathfrak{a}^*_{G,V} \cap \mathfrak{a}^\perp$, on a
  \begin{align*}
    \mathcal{M}_M(\pi_{\lambda+\mu}, \tilde{P}) & = \mathcal{M}_M(\pi_\lambda, \tilde{P}),\\
    \mathcal{I}_{\tilde{P}}(\pi_{\lambda+\mu}, \tilde{x}) & = \mathcal{I}_{\tilde{P}}(\pi_\lambda, \tilde{x}) e^{\angles{\mu, H_G(x)}}, \\
    e^{-\angles{\lambda+\mu, X}} &= e^{-\angles{\lambda,X}} e^{\angles{\mu,-X}} \\
    & = e^{-\angles{\lambda,X}} e^{\angles{\mu, -h_G(X)}}.
  \end{align*}
  L'inversion de Fourier sur $i\mathfrak{a}^*_{G,V} \cap \mathfrak{a}^\perp$ nous ramène à l'expression suivante
  $$ \int_{\frac{i\mathfrak{a}^*_{M,V} \cap \mathfrak{a}^\perp}{i\mathfrak{a}^*_{G,V} \cap \mathfrak{a}^\perp}} \Tr \left(\; \int_{\tilde{G}^{h_G(X)+\mathfrak{a}}_V/A} f(\tilde{x}) \mathcal{M}_M(\pi_\lambda, \tilde{P}) \mathcal{I}_{\tilde{P}}(\pi_\lambda, \tilde{x}) \dd\tilde{x} \right) e^{-\angles{\lambda,X}} \dd\lambda, $$
  ce qu'il fallait démontrer.
\end{proof}

\begin{corollaire}
  La forme linéaire $f \mapsto J_{\tilde{M}}(\pi,X,f)$ est concentrée en $h_G(X)$. Plus précisément, pour tout $b \in C^\infty(\mathfrak{a}_{G,V}/\mathfrak{a})$,
  \begin{gather}\label{eqn:J-concentration}
    J_{\tilde{M}}(\pi,X, f^b) = J_{\tilde{M}}(\pi,X,f) b(h_G(X)).
  \end{gather}
  Par conséquent, $J_{\tilde{M}}(\pi,X,\cdot)$ se prolonge canoniquement à $\mathcal{H}_{\mathrm{ac},\asp}(\tilde{G}_V, A)$.
\end{corollaire}
\begin{proof}
  C'est dû à la Proposition \ref{prop:concentration}.
\end{proof}

\begin{proposition-fr}\label{prop:J-moyenne}
  Soient $(A,\mathfrak{a})$ et $(B,\mathfrak{b})$ deux données centrales telles que $A \subset B$. Alors on a
  $$ J_{\tilde{M}}(\pi, X, f^B) = \int_{\mathfrak{b}/\mathfrak{a}} J_{\tilde{M}}(\pi, X+Y, f^A) \dd Y, \quad \pi \in \Pi_{\mathrm{unit},-}(\tilde{M}_V, B), X \in \mathfrak{a}_{M,V}/\mathfrak{b}, $$
  pour tout $f^A \in \mathcal{H}_{\asp}(\tilde{G}_V, A)$ et $f^A \mapsto f^B$ par l'application de la Proposition \ref{prop:f-moyenne}.
\end{proposition-fr}
\begin{proof}
  On vérifie que
  $$ \mathcal{I}_{\tilde{P}}(\pi_\lambda, f^B|_{\tilde{G}^{h_G(X)+\mathfrak{b}}_V/A}) = \int_{\mathfrak{b}/\mathfrak{a}} \mathcal{I}_{\tilde{P}}(\pi_\lambda, f^A|_{\tilde{G}^{X+Y+\mathfrak{a}}_V/A}) \dd Y $$
  en se rappelant la définition de $f^A \mapsto f^B$.

  D'autre part, d'après le Lemme \ref{prop:G-M-V-observation} on a
  $$ \frac{i\mathfrak{a}^*_{M,V} \cap \mathfrak{a}^\perp}{i\mathfrak{a}^*_{G,V} \cap \mathfrak{a}^\perp} \rightiso \frac{i\mathfrak{a}^*_{M,V}}{i\mathfrak{a}^*_{G,V}} \leftiso \frac{i\mathfrak{a}^*_{M,V} \cap \mathfrak{b}^\perp}{i\mathfrak{a}^*_{G,V} \cap \mathfrak{b}^\perp}. $$

  L'assertion résulte alors des formules données par la Proposition \ref{prop:Fourier-restriction} pour $f^A$ et $f^B$.
\end{proof}

\paragraph{Version à support presque compact}
On vient de voir que $J_{\tilde{M}}(\pi,X,\cdot)$ peut être défini sur $\mathcal{H}_{\text{ac},\asp}(\tilde{G}_V,A)$. Donnons une propriété de $J_{\tilde{M}}(\pi,X,\cdot)$ qui intervient lorsque l'on traite les fonctions à support presque compact.

Considérons la situation suivante. Soient $(A,\mathfrak{a})$ et $(B,\mathfrak{b})$ deux données centrales. On choisit des sections $s'$ et $s$ des projections $\mathfrak{a}_{M,V}/\mathfrak{a} \to \mathfrak{a}_{M,V}/\mathfrak{b}$ et $\mathfrak{a}_{G,V}/\mathfrak{a} \to \mathfrak{a}_{G,V}/\mathfrak{b}$ qui satisfont les conditions dans la Remarque \ref{rem:G-M-V-observation}, et on note $s'^*$, $s^*$ leurs duaux.

Soit $f_B \in \mathcal{H}_{\text{ac},\asp}(\tilde{G}_V,B)$, son image sous l'inclusion naturelle $\mathcal{H}_{\text{ac},\asp}(\tilde{G}_V, B) \to \mathcal{H}_{\text{ac},\asp}(\tilde{G}_V,A)$ est notée $f_A$. On l'écrit $f_B \mapsto f_A$.

\begin{proposition-fr}\label{prop:J-res}
  Soit $f_B \in \mathcal{H}_{\text{ac},\asp}(\tilde{G}_V,B)$. Alors
  $$ J_{\tilde{M}}(\pi, X, f_B) = J_{\tilde{M}}(\pi, s'(X), f_A), \quad \pi \in \Pi_{\mathrm{unit},-}(\tilde{G}_V, B), X \in \mathfrak{a}_{M,V}/\mathfrak{b}, $$
  si $f_B \mapsto f_A$ sous l'inclusion naturelle.
\end{proposition-fr}
\begin{proof}
  Cette démonstration est analogue à celle de la Proposition \ref{prop:res-diagramme}. Le Lemme \ref{prop:G-M-V-observation} entraîne que $s'^*$ et $s^*$ induisent un isomorphisme
  $$ \frac{i\mathfrak{a}^*_{M,V} \cap \mathfrak{a}^\perp}{i\mathfrak{a}^*_{G,V} \cap \mathfrak{a}^\perp} \rightiso \frac{i\mathfrak{a}^*_{M,V} \cap \mathfrak{b}^\perp}{i\mathfrak{a}^*_{G,V} \cap \mathfrak{b}^\perp}. $$

  En appliquant la Proposition \ref{prop:Fourier-restriction} à $f_B$ et en utilisant l'isomorphisme ci-dessus, on obtient
  $$ J_{\tilde{M}}(\pi, X, f_B) = \int_{\frac{i\mathfrak{a}^*_{M,V} \cap \mathfrak{a}^\perp}{i\mathfrak{a}^*_{G,V} \cap \mathfrak{a}^\perp}} \Tr \left( \mathcal{M}_M(\pi_{s'^*(\lambda)}, \tilde{P}) \mathcal{I}_{\tilde{P}}(\pi_{s'^*(\lambda)}, f_B|_{\tilde{G}^{h_G(X)+\mathfrak{b}}_V/B}) \right) e^{-\angles{s'^*(\lambda), X}} \dd\lambda . $$

  Toujours d'après le Lemme \ref{prop:G-M-V-observation}, on peut choisir un représentant dans $i\mathfrak{a}^*_{M,V} \cap \mathfrak{b}^\perp$ de la classe de $\lambda$. Alors $s'^*(\lambda)=\lambda$. D'autre part,
  \begin{align*}
    e^{-\angles{s'^*(\lambda),X}} &= e^{-\angles{\lambda, s'(X)}},\\
    \mathcal{I}_{\tilde{P}}\left( \pi_{s'^*(\lambda)}, f_B|_{\tilde{G}^{h_G(X)+\mathfrak{b}}_V/B} \right) &= \mathcal{I}_{\tilde{P}}\left( \pi_{s'^*(\lambda)}, f_A|_{\tilde{G}^{h_G(s'(X))+\mathfrak{a}}_V/A} \right),
  \end{align*}
  car $\tilde{G}^{h_G(s'(X))+\mathfrak{a}}_V/A = \tilde{G}^{s(h_G(X))+\mathfrak{a}}_V/A \rightiso \tilde{G}^{h_G(X)+\mathfrak{b}}_V/B$ par l'application quotient. On arrive ainsi à la formule de $J_{\tilde{M}}(\pi, s'(X), f_A)$ fournie par la Proposition \ref{prop:Fourier-restriction}.
\end{proof}

\begin{remarque}\label{rem:J-continuite}
  Étant donnés $(\pi,X)$, les formes linéaires $f \mapsto J_{\tilde{M}}(\pi,f)$ et $f \mapsto J_{\tilde{M}}(\pi,X,f)$ sont continues sur $\mathcal{H}_{\asp}(\tilde{G}_V, A)$. Fixons $\Gamma$, $N$ et supposons que $f \in \mathcal{H}_{\asp}(\tilde{G}_V, A)_{N,\Gamma}$, alors l'image de $\mathcal{I}_{\tilde{P}}(\pi_\lambda, f)$ appartient à un espace vectoriel de dimension finie indépendant de $\lambda$. La continuité est une conséquence des faits suivants
  \begin{itemize}
    \item les coefficients matriciels de $\mathcal{M}_M(\pi_\lambda, f)$ sont à croissance modérée en $\lambda \in i\mathfrak{a}^*_M$;
    \item pour tout $n \in \Z_{>0}$, les coefficients matriciels de $\mathcal{I}_{\tilde{P}}(\pi_\lambda, f)$ (avec $\lambda \in i\mathfrak{a}^*_M$) admettent une majoration par
      $$ p_n(f) (1+\|\lambda\|)^{-n} $$
      où $\|\cdot\|$ est une norme euclidienne quelconque sur $\mathfrak{a}^*_{M,\C}$ et $p_n(f)$ est une semi-norme continue de $\mathcal{H}_{\asp}(\tilde{G}_V, A)_{N,\Gamma}$.
  \end{itemize}
  Cf. \cite[(12.7)]{Ar89-IOR1}.
\end{remarque}

\subsection{Intégrales orbitales pondérées}
Rappelons brièvement la définition des intégrales orbitales pondérées anti-spécifiques dans \cite[\S 6.3]{Li14a}. Soient
\begin{itemize}
  \item $M \in \mathcal{L}(M_0)$,
  \item $\tilde{\gamma} \in \Gamma(\tilde{M}_V)$,
  \item $D^M: M(F_V) \to F_V$ le discriminant de Weyl \cite[Définition 5.6.1]{Li14a},
  \item $v_M: P(F_V) \backslash G(F_V)/K_V \to \R_{\geq 0}$, la fonction poids d'Arthur.
\end{itemize}
Selon notre convention dans \S\ref{sec:dists}, l'orbite $\tilde{\gamma}$ est déjà munie d'une mesure invariante. On prend un représentant dans la classe de $\tilde{\gamma}$, noté par le même symbole, et on désigne par $G_\gamma$ (resp. $M_\gamma$) le commutant connexe de $\gamma$ dans $G$ (resp. dans $M$). Alors $M_\gamma(F_V)$ est muni de la mesure de Haar correspondante à la mesure invariante choisie sur l'orbite.

Supposons d'abord que $M_\gamma=G_\gamma$, on pose
\begin{gather}\label{eqn:int-orb-ponderee}
  J_{\tilde{M}}(\tilde{\gamma}, f) := |D^M(\gamma)|^{\frac{1}{2}} \int_{G_\gamma(F_V) \backslash G(F_V)} f(x^{-1}\tilde{\gamma}x) v_M(x) \dd x
\end{gather}
pour tout $f \in C^\infty_{c,\asp}(\tilde{G}_V/A)$. Cette intégrale est bien définie car $A \subset G_\gamma(F_V)$ pour tout $\gamma$.

En général, il existe des facteurs $s^L_M(\gamma,a)$, où $L \in \mathcal{L}(M)$ et $a$ parcourt des éléments dans $A_M(F_V)$ en position générale de sorte que $M_{a\gamma}=G_{a\gamma}$, qui permettent de définir
\begin{gather}\label{eqn:int-orb-ponderee-gen}
  J_{\tilde{M}}(\tilde{\gamma}, f) := \lim_{\substack{a \to 1 \\ a \in A_M(F_V)}} \sum_{L \in \mathcal{L}(M)} s^L_M(\gamma,a) J_{\tilde{L}}(a\tilde{\gamma}, f).
\end{gather}
On renvoie à \cite[Théorème 6.3.2]{Li14a} pour les détails. On vérifie que
$$ J_{\tilde{M}}(z\tilde{\gamma},\cdot)=J_{\tilde{M}}(\tilde{\gamma},\cdot), \quad z \in A. $$

\begin{proposition-fr}
  La forme linéaire $f \mapsto J_{\tilde{M}}(\tilde{\gamma}, f)$ de $\mathcal{H}_{\asp}(\tilde{G}_V, A)$ est continue et concentrée en $H_G(\gamma)$. En particulier, elle se prolonge canoniquement en une forme linéaire continue sur $\mathcal{H}_{\mathrm{ac},\asp}(\tilde{G}_V, A)$.
\end{proposition-fr}
\begin{proof}
  C'est connu que $J_{\tilde{M}}(\tilde{\gamma},\cdot)$ définit une mesure absolument continue par rapport à la mesure invariante sur l'orbite induite $\tilde{\gamma}^G$, dont la construction sera rappelée dans la Définition \ref{def:induction-orbites}; cf. \cite[Corollary 6.2]{Ar88LB} ou \cite[Proposition 5.3.7]{Li14a}. Le fait crucial établi dans loc. cit. est que toute fonction dans $C^\infty_c(\tilde{G}_V/A)$ est intégrable pour cette mesure. Ladite mesure n'est pas positive en générale. 

  Montrons la concentration. Pour tout $\tilde{\delta}$ dans la classe $\tilde{\gamma}^G$, la définition d'induction entraîne que $H_{\tilde{G}}(\tilde{\delta}) = H_G(\gamma)$, donc $J_{\tilde{M}}(\tilde{\gamma}, \cdot)$ est concentrée en $H_G(\gamma)$.

  Montrons ensuite la continuité. Il faut que $J_{\tilde{M}}(\tilde{\gamma}, \cdot)$ soit continue sur $\mathcal{H}_{\asp}(\tilde{G}_V, A)_{N,\Gamma}$ pour tous $N$, $\Gamma$ (voir \S\ref{sec:PW-fonctions} pour ces notations; en fait, $\Gamma$ n'intervient pas ici). Fixons $N$, $\Gamma$ et prenons $\varphi \in C^\infty_c(G(F_V)/A)$ telle que $\varphi \geq 0$ et
  $$ \varphi = 1 \quad \text{ sur } \{x \in G(F_V) : \log\|x\| \leq N \} \cdot A. $$

  Soit $f \in \mathcal{H}_{\asp}(\tilde{G}_V, A)_{N,\Gamma}$, alors $f = (\varphi \circ \rev) f$. On vient de remarquer que $J_{\tilde{M}}(\tilde{\gamma},\cdot)$ définit une mesure sur $\tilde{\gamma}^G$, notée $\mu$. Rappelons que $f \in L^1(\mu)=L^1(|\mu|)$. Donc on a
  $$ |\mu(f)| \leq |\mu|(|f|) \leq \sup|f| \cdot |\mu|(\varphi). $$
  Or on a aussi $\varphi \in L^1(|\mu|)$. Cela démontre la continuité de $J_{\tilde{M}}(\tilde{\gamma}, \cdot)$.
\end{proof}

En fait, la forme prolongée à $\mathcal{H}_{\text{ac},\asp}(\tilde{G}_V, A)$ sont définies par les mêmes formules \eqref{eqn:int-orb-ponderee} et \eqref{eqn:int-orb-ponderee-gen}. Les résultats suivants sont des conséquences immédiates de lesdites formules.

Indiquons que la condition d'adhérence pour $V$ n'intervient pas dans la définition de $J_{\tilde{M}}(\tilde{\gamma},\cdot)$

\begin{proposition-fr}
  Soient $(A,\mathfrak{a})$, $(B,\mathfrak{b})$ deux données centrales telles que $A \subset B$. Pour $f^A \in \mathcal{H}_{\asp}(\tilde{G}_V, A)$, $f^A \mapsto f^B \in \mathcal{H}_{\asp}(\tilde{G}_V, B)$ par l'application de la Proposition \ref{prop:f-moyenne}, on a
  $$ J_{\tilde{M}}(\tilde{\gamma}, f^B) = \int_{B/A} J_{\tilde{M}}(z\tilde{\gamma}, f^A) \dd z, \quad \tilde{\gamma} \in \Gamma(\tilde{M}_V). $$
\end{proposition-fr}

\begin{proposition-fr}\label{prop:orbint-res}
  Soient $(A,\mathfrak{a})$, $(B,\mathfrak{b})$ comme ci-dessus. Soit $f_B \in \mathcal{H}_{\mathrm{ac},\asp}(\tilde{G}_V, B)$, son image dans $\mathcal{H}_{\mathrm{ac},\asp}(\tilde{G}_V, A)$ sous l'inclusion naturelle est notée $f_A$. Alors
  $$ J_{\tilde{M}}(\tilde{\gamma}, f_B) = J_{\tilde{M}}(\tilde{\gamma}, f_A), \quad \tilde{\gamma} \in \Gamma(\tilde{M}_V). $$
\end{proposition-fr}

\subsection{L'application $\phi_{\tilde{M}}$}
Conservons les notations dans la sous-section précédente. On définit une application linéaire
\begin{align*}
  \phi_{\tilde{M}}: \mathcal{H}_{\text{ac},\asp}(\tilde{G}_V, A) & \longrightarrow \left\{ \text{fonctions } \Pi_{\text{temp},-}(\tilde{M}_V, A) \times \frac{\mathfrak{a}_{M,V}}{\mathfrak{a}} \to \C \right\} \\
  f & \longmapsto \left[ (\pi,X) \mapsto J_{\tilde{M}}(\pi, X, f) \right].
\end{align*}

\begin{theoreme}\label{prop:phi-M}
  Si l'Hypothèse \ref{hyp:PW} est satisfaite pour toute donnée centrale, alors $\phi_{\tilde{M}}$ induit une application linéaire continue
  $$ \phi_{\tilde{M}}: \mathcal{H}_{\mathrm{ac},\asp}(\tilde{G}_V, A) \to I\mathcal{H}_{\mathrm{ac},\asp}(\tilde{M}_V, A)$$
  pour toute donnée centrale $(A,\mathfrak{a})$ de $\tilde{G}_V$.
\end{theoreme}
Signalons qu'il en résulte que $f \mapsto J_{\tilde{M}}(\pi,X,f)$ est continue pour tout $(\pi,X)$, car l'évaluation en $(\pi,X)$ est une forme linéaire continue de $I\mathcal{H}_{\text{ac},\asp}(\tilde{M}_V, A)$.

\begin{proof}
  Quand $A=\{1\}$, c'est le cas traité par Arthur dans \cite[Theorem 12.1]{Ar89-IOR1}. La démonstration se fait par une étude approfondie des résidus de caractères pondérés; nous ne la répétons pas ici. En général, soit $(A',\mathfrak{a'})$ une donnée centrale telle que $A' \subset A$ et on suppose que l'assertion est prouvée pour $A'$. On choisit des données auxiliaires comme dans la Remarque \ref{rem:G-M-V-observation} pour définir l'application $\text{res}$; on considère le diagramme
  $$\xymatrix{
    \mathcal{H}_{\text{ac},\asp}(\tilde{G}_V, A') \ar[r] & I\mathcal{H}_{\text{ac},\asp}(\tilde{G}_V, A') \ar[d]^{\text{res}} \\
    \mathcal{H}_{\text{ac},\asp}(\tilde{G}_V, A) \ar[r] \ar@{^{(}->}[u] & \left\{ \text{fonctions } \Pi_{\text{temp},-}(\tilde{M}_V, A) \times \frac{\mathfrak{a}_{M,V}}{\mathfrak{a}} \to \C \right\}
  }$$
  dont les flèches horizontales sont les $\phi_{\tilde{M}}$ pour $A'$ et $A$. D'après la Proposition \ref{prop:J-res}, les définitions de $\phi_{\tilde{M}}$ et celle de $\text{res}$, ce diagramme est commutatif. D'autre part, la Proposition \ref{prop:res-diagramme} affirme que le but de $\text{res}$ peut être remplacé par $I\mathcal{H}_{\text{ac},\asp}(\tilde{G}_V, A) = \text{PW}_{\text{ac},\asp}(\tilde{G}_V, A)$. Maintenant on voit que $\phi_{\tilde{M}}: \mathcal{H}_{\mathrm{ac},\asp}(\tilde{G}_V, A) \to I\mathcal{H}_{\text{ac},\asp}(\tilde{M}_V, A)$ est la composée de trois applications linéaires continues. L'assertion en découle en prenant $A'=\{1\}$.
\end{proof}

Observons que pour $f \in \mathcal{H}_{\asp}(\tilde{G}_V, A)$, son image $\phi_{\tilde{M}}(f)$ dans $I\mathcal{H}_{\text{ac},\asp}(\tilde{M}_V, A)$ vérifie
$$ \phi_{\tilde{M}}(f)(\cdot,X) = 0 \quad \text{ pour } \|h_G(X)\| \gg 0 $$
où $\|\cdot\|$ est une norme quelconque de $\mathfrak{a}_{G,V}/\mathfrak{a}$, car la Proposition \ref{prop:Fourier-restriction} entraîne que $\phi_{\tilde{M}}(f)(\cdot,X)=0$ si $h_G(X)$ n'appartient pas à la projection via $H_{\tilde{G}}$ de $\Supp(f)$ sur $\mathfrak{a}_{G,V}/\mathfrak{a}$, ce qui est compact. Notons $I\mathcal{H}_{\text{ac},\asp}(\tilde{M}_V, A)^{G-\text{cpt}}$ le sous-espace des élément qui vérifient ladite propriété.

Pour $\phi^A \in I\mathcal{H}_{\text{ac},\asp}(\tilde{M}_V, A)^{G-\text{cpt}}$ et une donnée centrale $(B,\mathfrak{b})$ de $\tilde{G}_V$, l'application $\phi^A \mapsto \phi^B$ de la Proposition \ref{prop:phi-moyenne} est toujours bien définie.

\begin{proposition-fr}\label{prop:phi-M-moyenne}
  Soient $(A,\mathfrak{a})$ et $(B,\mathfrak{b})$ deux données centrales de $\tilde{G}_V$ telles que $A \subset B$, alors
  $$ \phi_{\tilde{M}}(f^B) = (\phi_{\tilde{M}}(f^A))^B $$
  pour toute $f^A \in \mathcal{H}_{\asp}(\tilde{G}_V, A)$, et $f^A \mapsto f^B$ est l'application de la Proposition \ref{prop:f-moyenne}.
\end{proposition-fr}
\begin{proof}
  C'est une conséquence immédiate de la Proposition \ref{prop:J-moyenne}.
\end{proof}

\begin{proposition-fr}\label{prop:phi-M-res}
  Soient $(A',\mathfrak{a})$ et $(A,\mathfrak{a})$ comme ci-dessus. Soit $f_B \in \mathcal{H}_{\mathrm{ac},\asp}(\tilde{G}_V, B)$ et on désigne par $f_A$ son image dans $\mathcal{H}_{\mathrm{ac},\asp}(\tilde{G}_V, A)$ sous l'inclusion naturelle. Alors on a
  $$ \phi_{\tilde{M}}(f_B) = \mathrm{res}(\phi_{\tilde{M}}(f_A)). $$
\end{proposition-fr}
\begin{proof}
  Afin de définir l'application $\text{res}$, il faut fixer des données auxiliaires comme dans la Remarque \ref{rem:G-M-V-observation}. Ensuite, l'assertion résulte immédiatement de la Proposition \ref{prop:J-res}.
\end{proof}

\subsection{Non-invariance}
Dans cette sous-section, on enregistre un résultat standard concernant le comportement des distributions pondérées par rapport à la conjugaison, ou plus précisément par rapport à la convolution. C'est ce qui permettra de fabriquer les distributions invariantes.

Fixons une donnée centrale $(A,\mathfrak{a})$ de $\tilde{G}_V$. Soient $Q = M_Q U_Q \in \mathcal{F}(M)$, $y \in G(F_V)$ et $f \in \mathcal{H}_{\asp}(\tilde{G}_V,A)$. La fonction module de $Q$ est notée $\delta_Q: Q(F_S) \to \R_{>0}$. Posons
\begin{gather}\label{eqn:descente-y}
  f_{\tilde{Q},y}(\tilde{m}) := \delta_Q(m)^{\frac{1}{2}} \iint_{\tilde{K}_V \times U_Q(\A)} f(k^{-1}\tilde{m}uk) u'_Q(k,y) \dd u \dd\tilde{k}, \quad \tilde{m} \in (\widetilde{M_Q})_V,
\end{gather}
où $u'_Q(k,y)$ est déduite de la fonction
$$ u_Q(\lambda,k,y) := e^{-\angles{\lambda, H_Q(ky)}}, \lambda \in \mathfrak{a}^*_{M,\C};$$
voir \cite[\S 4.1]{Li14a} pour la construction détaillée.

Pour $h \in \mathcal{H}_{\asp}(\tilde{G}_V)$, on pose
\begin{align*}
  L_{\tilde{y}} f & := f(\tilde{y}^{-1} \cdot), \\
  R_{\tilde{y}} f & := f(\cdot \tilde{y}^{-1}), \quad \tilde{y} \in \tilde{G}_V, \\
  L_{Q,h} f & := \int_{\tilde{G}_V} h(\tilde{y}) (L_{\tilde{y}} f)_{\tilde{Q},y^{-1}} \dd\tilde{y}, \\
  R_{Q,h} f & := \int_{\tilde{G}_V} h(\tilde{y}) (R_{\tilde{y}} f)_{\tilde{Q}, y^{-1}} \dd\tilde{y},
\end{align*}
où $y := \rev(\tilde{y})$. Toutes les opérations ci-dessus commutent avec les translations par $A$. Ainsi, on vérifie que $f \mapsto f_{\tilde{Q},y}$, $f \mapsto L_{Q,h} f$ et $f \mapsto R_{Q,h} f$ définissent des applications linéaires $\mathcal{H}_{\asp}(\tilde{G}_V, A) \to \mathcal{H}_{\asp}((\widetilde{M_Q})_V, A)$.

\begin{theoreme}[Cf. {\cite[Lemma 6.2 et (7.2)]{Ar89-IOR1}}]\label{prop:non-invariance}
  Soit $J_{\tilde{M}}$ l'une des distributions $J_{\tilde{M}}(\pi,\cdot)$, $J_{\tilde{M}}(\pi,X,\cdot)$ ou $J_{\tilde{M}}(\tilde{\gamma},\cdot)$, et pour tout $L \in \mathcal{L}(M)$ on note $J^{\tilde{L}}_{\tilde{M}}$ son avatar où le rôle de  $G$ est remplacé par $L$. Alors
  \begin{align*}
    J_{\tilde{M}}(L_h f) & = \sum_{Q=M_Q U_Q \in \mathcal{F}(M)} J^{\widetilde{M_Q}}_{\tilde{M}}(R_{Q,h} f), \\
    J_{\tilde{M}}(R_h f) & = \sum_{Q=M_Q U_Q \in \mathcal{F}(M)} J^{\widetilde{M_Q}}_{\tilde{M}}(L_{Q,h} f)
  \end{align*}
  pour tous $h \in \mathcal{H}_{\asp}(\tilde{G}_V)$ et $f \in \mathcal{H}_{\asp}(\tilde{G}_V)$.
\end{theoreme}
\begin{proof}[Esquisse de la démonstration]
  Le cas de $J_{\tilde{M}}(\tilde{\gamma},\cdot)$ est prouvé dans \cite[Lemma 8.2]{Ar81}. Le cas $A=\{1\}$ et $J_{\tilde{M}} = J_{\tilde{M}}(\pi,\cdot)$ ou $J_{\tilde{M}}(\pi,X,\cdot)$ est traité dans \cite[Lemma 6.2 et (7.2)]{Ar89-IOR1}, mais les caractères pondérés dans ces références-là sont définis à l'aide de facteurs normalisants. On peut les adapter aux caractères pondérés canoniquement normalisés en suivant la recette de \cite[\S 3]{Ar98}.

  Pour passer au cas général, il y a deux choix: (1) on observe que la présence de $A$ n'affecte pas les arguments dans op. cit.; (2) on part du cas $A=\{1\}$, alors le cas général en résulte d'après la Proposition \ref{prop:J-moyenne} car $f \mapsto f_{\tilde{Q},y}$ commute avec translation par $A$.
\end{proof}

On laisse le soin au lecteur de formuler une version pour $f \in \mathcal{H}_{\text{ac},\asp}(\tilde{G}_V)$.

\subsection{Construction des distributions locales invariantes}\label{sec:dist-locale-preuve}
Dans cette sous-section, on suit les idées d'Arthur dans \cite{Ar81,Ar88-TF1} pour définir des distributions invariantes à partir des objets pondérés et de l'application $\phi_{\tilde{M}}$. On utilisera les formules de descente pour données dans \cite[\S\S 8-9]{Ar88-TF1} pour les distributions pondérées.

On raisonne par récurrence sur $\dim G$. On suppose vérifiée l'Hypothèse \ref{hyp:PW} pour toute donnée centrale.

\paragraph{Côté spectral}
Fixons $M \in \mathcal{L}(M_0)$. Mettons l'hypothèse de récurrence suivante.

\begin{hypothese}\label{hyp:recurrence-spec}
  Pour tout $L \in \mathcal{L}(M)$ tel que $L \neq G$, supposons définies des formes linéaires
  $$ f^A \mapsto I^{\tilde{L}}_{\tilde{M}}(\pi, X, f^A), \quad \pi \in \Pi_{\text{unit},-}(\tilde{M}_V, A), X \in \mathfrak{a}_{M,V}/\mathfrak{a} $$
  de $\mathcal{H}_{\asp}(\tilde{L}_V, A)$, pour toute donnée centrale $(A,\mathfrak{a})$ de $\tilde{L}_V$, telles que
  \begin{enumerate}
    \item $I^{\tilde{L}}_{\tilde{M}}(\pi, X, \cdot)$ est continue, invariante et concentrée en $h_L(X)$, en particulier elle se prolonge à $\mathcal{H}_{\mathrm{ac},\asp}(\tilde{L}_V, A)$;
    \item $I^{\tilde{L}}_{\tilde{M}}(\pi, X, \cdot)$ est supportée par $I\mathcal{H}_{\mathrm{ac},\asp}(\tilde{L}_V, A)$;
    \item pour toute donnée centrale $(B,\mathfrak{b})$ de $\tilde{L}_V$ telle que $A \subset B$, on a
      $$ I^{\tilde{L}}_{\tilde{M}}(\pi, X, f^B) = \int_{\mathfrak{b}/\mathfrak{a}} I^{\tilde{L}}_{\tilde{M}}(\pi, X+Y, f^A) \dd Y, \quad \pi \in \Pi_{\text{unit},-}(\tilde{M}_V, B), X \in \mathfrak{a}_{M,V}/\mathfrak{b}, $$
      où $f^A \mapsto f^B$ est l'application dans la Proposition \ref{prop:f-moyenne};
    \item pour $(B,\mathfrak{b})$ comme ci-dessus, on choisit les données $(s',s)$ comme dans la Remarque \ref{rem:G-M-V-observation}, alors
      $$ I^{\tilde{L}}_{\tilde{M}}(\pi,X, f_B) = I^{\tilde{L}}_{\tilde{M}}(\pi, s'(X), f_A), \quad \pi \in \Pi_{\text{unit},-}(\tilde{M}_V, B), X \in \mathfrak{a}_{M,V}/\mathfrak{b}, $$
      où $f_B \mapsto f_A$ est l'inclusion naturelle $\mathcal{H}_{\mathrm{ac},\asp}(\tilde{G}_V, B) \hookrightarrow \mathcal{H}_{\mathrm{ac},\asp}(\tilde{G}_V, A)$.
  \end{enumerate}
\end{hypothese}

\begin{remarque}\label{rem:recurrence-spec}
  Selon la condition de support, on peut aussi bien regarder $I^{\tilde{L}}_{\tilde{M}}(\pi,X,\cdot)$ comme une forme linéaire continue de $I\mathcal{H}_{\text{ac},\asp}(\tilde{L}_V, A)$. Les remarques suivantes nous seront utiles.
  \begin{enumerate}
    \item Soit $\phi^A \in I\mathcal{H}_{\asp}(\tilde{L}_V, A)$ et $\phi^A \mapsto \phi^B$ par la Proposition \ref{prop:phi-moyenne}, alors l'égalité
    $$ I^{\tilde{L}}_{\tilde{M}}(\pi, X, \phi^B) = \int_{\mathfrak{b}/\mathfrak{a}} I^{\tilde{L}}_{\tilde{M}}(\pi, X+Y, \phi^A) \dd Y $$
    demeure valable pour $\phi^A \in I\mathcal{H}_{\text{ac},\asp}(\tilde{L_V}, A)^{G-\text{cpt}}$ si $(A,\mathfrak{a})$ et $(B,\mathfrak{b})$ proviennent de données centrales de $\tilde{G}_V$.
    \item Vu la Proposition \ref{prop:res-diagramme}, la dernière condition implique
      \begin{gather}\label{eqn:I-res-phi}
        I^{\tilde{L}}_{\tilde{M}}(\pi,s'(X),\phi) = I^{\tilde{L}}_{\tilde{M}}(\pi,X,\text{res}(\phi)), \quad \phi \in I\mathcal{H}_{\text{ac},\asp}(\tilde{L}_V, A).
      \end{gather}
  \end{enumerate}
\end{remarque}

\begin{definition-fr}\label{def:I-spec}
  Supposons vérifiée l'Hypothèse \ref{hyp:recurrence-spec}. Soit $(A,\mathfrak{a})$ une donnée centrale de $\tilde{G}_V$, on pose
  $$ I_{\tilde{M}}(\pi, X, f^A) = I^{\tilde{G}}_{\tilde{M}}(\pi, X, f^A) := J_{\tilde{M}}(\pi, X, f^A) - \sum_{\substack{L \in \mathcal{L}(M) \\ L \neq G}} I^{\tilde{L}}_{\tilde{M}}(\pi, X, \phi_{\tilde{L}}(f^A)) $$
  où $f^A \in \mathcal{H}_{\asp}(\tilde{G}_V, A)$, $\pi \in \Pi_{\text{unit},-}(\tilde{M}_V, A)$, $X \in \mathfrak{a}_{M,V}/\mathfrak{a}$, et toutes les formes linéaires sont définies par rapport à $(A,\mathfrak{a})$.
\end{definition-fr}
Lorsque $G$ est anisotrope modulo le centre, l'Hypothèse \ref{hyp:recurrence-spec} est vide et on retrouve les caractères usuels.

\begin{lemme}\label{prop:I_M-def}
  Les formes linéaires $I_{\tilde{M}}(\pi, X, \cdot)$ vérifient toutes les propriétés dans l'Hypothèse \ref{hyp:recurrence-spec} avec $G$ au lieu de $L$.
\end{lemme}
\begin{proof}[Démonstration partielle]
  La continuité est évidente. L'invariance résulte du Théorème \ref{prop:non-invariance} de façon standard: voir les discussions pour \cite[(4.1)]{Ar81} pour les détails. La concentration en $h_G(X)$ provient de \eqref{eqn:J-concentration} et de sa conséquence que $\phi_{\tilde{L}}(f^b) = \phi_{\tilde{L}}(f)^{b \circ h_G}$.

  La formule pour $I_{\tilde{M}}(\pi,X,f^B)$ résulte de la Proposition \ref{prop:J-moyenne} et \ref{prop:phi-moyenne}; il faut utiliser le premier point de la Remarque \ref{rem:recurrence-spec}. La formule pour $I^{\tilde{L}}_{\tilde{M}}(\pi,X, f_B)$ résulte de la Proposition \ref{prop:J-res}, \ref{prop:phi-M-res} et \eqref{eqn:I-res-phi}.

  Il reste à montrer que $I_{\tilde{M}}(\pi,X,\cdot)$ est supportée par $I\mathcal{H}_{\text{ac},\asp}(\tilde{G}_V, A)$. On complétera cette part après une discussion du côté géométrique.
\end{proof}

\paragraph{Côté géométrique}
Le cas géométrique admet une structure analogue à celle précédente, mais on verra que les démonstrations sont beaucoup plus faciles. Fixons toujours $M \in \mathcal{L}(M_0)$ et mettons l'hypothèse de récurrence suivante.

\begin{hypothese}\label{hyp:recurrence-geom}
  Pour tout $L \in \mathcal{L}(M)$ tel que $L \neq G$, supposons définies des formes linéaires
  $$ f^A \mapsto I^{\tilde{L}}_{\tilde{M}}(\tilde{\gamma}, f^A), \quad \tilde{\gamma} \in \Gamma(\tilde{M}_V) $$
  de $\mathcal{H}_{\asp}(\tilde{L}_V, A)$, pour toute donnée centrale $(A,\mathfrak{a})$ de $\tilde{L}_V$, telles que
  \begin{enumerate}
    \item $I^{\tilde{L}}_{\tilde{M}}(\tilde{\gamma}, \cdot)$ est continue, invariante et concentrée en $H_L(\gamma)$, en particulier elle se prolonge à $\mathcal{H}_{\mathrm{ac},\asp}(\tilde{L}_V, A)$;
    \item $I^{\tilde{L}}_{\tilde{M}}(\tilde{\gamma}, \cdot)$ est supportée par $I\mathcal{H}_{\mathrm{ac},\asp}(\tilde{L}_V, A)$;
    \item pour toute donnée centrale $(B,\mathfrak{b})$ de $\tilde{L}_V$ telle que $A \subset B$, on a
      $$ I^{\tilde{L}}_{\tilde{M}}(\tilde{\gamma}, f^B) = \int_{B/A} I^{\tilde{L}}_{\tilde{M}}(z\tilde{\gamma},f^A) \dd z, \quad \tilde{\gamma} \in \Gamma(\tilde{M}_V), $$
      où $f^A \mapsto f^B$ est l'application dans la Proposition \ref{prop:f-moyenne};
    \item pour $(B,\mathfrak{b})$ comme ci-dessus, on choisit les données $(s',s)$ comme dans la Remarque \ref{rem:G-M-V-observation}, alors
      $$ I^{\tilde{L}}_{\tilde{M}}(\tilde{\gamma}, f_B) = I^{\tilde{L}}_{\tilde{M}}(\tilde{\gamma}, f_A), \quad \gamma \in \Gamma(\tilde{M}_V), $$
      où $f_B \mapsto f_A$ est l'inclusion naturelle $\mathcal{H}_{\mathrm{ac},\asp}(\tilde{G}_V, B) \hookrightarrow \mathcal{H}_{\mathrm{ac},\asp}(\tilde{G}_V, A)$;
    \item on a
      $$ I^{\tilde{L}}_{\tilde{M}}(\tilde{\gamma}, f^A) = \lim_{\substack{a \to 1 \\ a \in A_M(F_V)}} \sum_{M_1 \in \mathcal{L}^L(M)} s^{M_1}_M(\gamma, a) I^{\tilde{L}}_{\tilde{M}^1}(a\tilde{\gamma}, f^A), $$
      où $a$ est en position générale de sorte que $M_{a\gamma} = L_{a\gamma}$ et $s^{M_1}_M(\gamma,a)$ est le facteur dans \eqref{eqn:int-orb-ponderee-gen} définissant les intégrales orbitales pondérées.
  \end{enumerate}
\end{hypothese}
La Remarque \ref{rem:recurrence-spec} s'applique aussi à cette situation géométrique.

\begin{definition-fr}
  Supposons vérifiée l'Hypothèse \ref{hyp:recurrence-geom}. Soit $(A,\mathfrak{a})$ une donnée centrale de $\tilde{G}_V$, on pose
  $$ I_{\tilde{M}}(\tilde{\gamma}, f^A) = I^{\tilde{G}}_{\tilde{M}}(\tilde{\gamma}, f^A) := J_{\tilde{M}}(\tilde{\gamma}, f^A) - \sum_{\substack{L \in \mathcal{L}(M) \\ L \neq G}} I^{\tilde{L}}_{\tilde{M}}(\tilde{\gamma}, \phi_{\tilde{L}}(f^A)) $$
  où $f^A \in \mathcal{H}_{\asp}(\tilde{G}_V, A)$, $\tilde{\gamma} \in \Gamma(\tilde{M}_V)$, et toutes les formes linéaires sont définies par rapport à $(A,\mathfrak{a})$.
\end{definition-fr}
Lorsque $G$ est anisotrope modulo le centre, l'Hypothèse \ref{hyp:recurrence-geom} est vide et on retrouve les intégrales orbitales usuelles.

\begin{lemme}
  Les formes linéaires $I_{\tilde{M}}(\tilde{\gamma}, \cdot)$ vérifient toutes les propriétés dans l'Hypothèse \ref{hyp:recurrence-geom} avec $G$ au lieu de $L$.
\end{lemme}
\begin{proof}
  On peut reprendre les arguments pour le Lemme \ref{prop:I_M-def}, qui se simplifient beaucoup dans le cadre géométrique. Par ailleurs, on démontre la dernière propriété par une récurrence simple. Il reste donc à montrer que $I_{\tilde{M}}(\tilde{\gamma}, \cdot)$ est supportée par $I\mathcal{H}_{\text{ac},\asp}(\tilde{G}_V, A)$. Soit $(A',\mathfrak{a}')$ une donnée centrale telle que $A' \subset A$. Au vu de la propriété
  $$ I^{\tilde{L}}_{\tilde{M}}(\tilde{\gamma}, f^A) = \int_{A/A'} I^{\tilde{L}}_{\tilde{M}}(z\tilde{\gamma},f^{A'}) \dd z, \quad f^{A'} \mapsto f^A $$
  que l'on vient d'établir, le Lemme \ref{prop:lemma-support} nous ramène au cas de la donnée centrale $(A',\mathfrak{a}')$. Prenons $A'=\{1\}$ pour se ramener au cas traité par Arthur dans \cite{Ar88-TF1,Ar88-TF2}. Les formules de descente nous ramène ensuite au cas $V=\{v\}$ est un singleton.

  Soit $f \in \mathcal{H}_{\text{ac},\asp}(\tilde{G}_v)$ telle que $\phi_{\tilde{G}}(f)=0$; le but est de montrer que $I_{\tilde{M}}(\tilde{\gamma}, f)=0$ pour tout $\tilde{\gamma} \in \Gamma(\tilde{M}_V)$. On peut aussi supposer que $f \in \mathcal{H}_{\asp}(\tilde{G}_v)$. Imposons une seconde hypothèse de récurrence que cette propriété d'annulation soit satisfaite pour les $I_{\tilde{M}_1}(\cdots)$ où $M_1 \in \mathcal{L}(M)$, $M_1 \neq M$.

  L'étape suivante est de se ramener au cas où $\tilde{\gamma}$ est semi-simple régulière en tant qu'un élément dans $\tilde{G}_v$. Cela se fait par une étude du comportement local des intégrales orbitales pondérées, eg. \cite[Proposition 6.4.1]{Li14a}; voir \cite[pp.523-524]{Ar88-TF2} pour les détails.

  Dans \cite{Li12b}, on a aussi étudié les intégrales orbitales pondérées $J^{\text{SHC}}_{\tilde{M}}(\cdots)$ le long de telles classes, évaluées en les fonctions dans l'espace de Schwartz-Harish-Chandra $\mathcal{C}_{\asp}(\tilde{G}_v)$. Du côté spectral, les caractères pondérés se prolongent aux fonctions dans $\mathcal{C}_{\asp}(\tilde{G}_v)$ pourvu que l'on se limite aux représentations tempérées. Dans ce cadre, on a également les applications $\phi^{\text{SHC}}_{\tilde{L}}: \mathcal{C}_{\asp}(\tilde{G}_v) \to I\mathcal{C}_{\asp}(\tilde{G}_v)$ et les distributions invariantes $I^{\text{SHC}}_{\tilde{L}}(\tilde{\gamma},\cdot)$ supportées par $I\mathcal{C}_{\asp}(\tilde{G}_v)$. Ici, on regarde $I\mathcal{C}_{\asp}(\tilde{G}_v)$ comme un ensemble de fonctions $\Pi_{\text{temp},-}(\tilde{G}_v) \times \mathfrak{a}_{G,F_v} \to \C$; dans \cite[\S 5.6]{Li12b} il n'y a pas de variable en $\mathfrak{a}_{G,F_v}$, mais on peut l'introduire à l'aide d'une transformation de Fourier si l'on veut.

  En reprenant la recette de \S\ref{sec:ac}, on introduit les espaces des fonctions à support presque compact $\mathcal{C}_{\text{ac},\asp}(\tilde{G}_v)$, $I\mathcal{C}_{\text{ac},\asp}(\tilde{G}_v)$ à l'aide de fonctions de Schwartz-Bruhat $b$ sur $\mathfrak{a}_{G,F_v}$, qui servent comme des multiplicateurs. De façon familière, $\phi^{\text{SHC}}_{\tilde{L}}$ et $I^{\text{SHC}}_{\tilde{L}}(\tilde{\gamma},\cdot)$ se prolongent à ces nouveaux espaces. Le lien entre les divers espaces est récapitulé par le diagramme commutatif
  $$\xymatrix{
    \mathcal{C}_{\asp}(\tilde{G}_v) \ar@{^{(}->}[r] & \mathcal{C}_{\text{ac},\asp}(\tilde{G}_v) \ar[r]^{\phi_{\tilde{L}}^{\text{SHC}}} & I\mathcal{C}_{\text{ac},\asp}(\tilde{L}_v) \\
    \mathcal{H}_{\asp}(\tilde{G}_v) \ar@{^{(}->}[r] \ar@{^{(}->}[u] & \mathcal{H}_{\text{ac},\asp}(\tilde{G}_v) \ar[r]_{\phi_{\tilde{L}}} \ar@{^{(}->}[u] & I\mathcal{H}_{\text{ac},\asp}(\tilde{G}_v) \ar@{^{(}->}[u]
  }.$$
  
  En raisonnant par récurrence sur $\dim G$, on obtient
  $$ I_{\tilde{M}}(\tilde{\gamma}, \varphi) = I^{\text{SHC}}_{\tilde{M}}(\tilde{\gamma}, \varphi), \quad \varphi \in \mathcal{H}_{\text{ac},\asp}(\tilde{G}_v). $$

  Appliquons cette égalité au cas $\varphi=f$. Puisque $\phi_{\tilde{G}}(f)=0$, on a aussi $\phi^{\text{SHC}}_{\tilde{G}}(f)=0$. D'où $I_{\tilde{M}}(\tilde{\gamma}, f) = I^{\text{SHC}}_{\tilde{M}}(\tilde{\gamma}, f)=0$ d'après \cite[Corollaire 5.8.9]{Li12b}.
\end{proof}

\begin{proof}[Complétion de la démonstration du Lemme \ref{prop:I_M-def}]
  Montrons que $I_{\tilde{M}}(\pi,X,\cdot)$ est supportée par $I\mathcal{H}_{\text{ac},\asp}(\tilde{G}_V, A)$. Comme dans la démonstration pour $I_{\tilde{M}}(\tilde{\gamma},\cdot)$, on se ramène au cas $A=\{1\}$ par le Lemme \ref{prop:lemma-support}, puis au cas $V=\{v\}$ par les formules de descente. Soit $f \in \mathcal{H}_{\text{ac},\asp}(\tilde{G}_v)$ telle que $\phi_{\tilde{G}}(f)=0$. On vient de voir que $I_{\tilde{L}}(\tilde{\gamma},f)=0$ pour tout $L \in \mathcal{L}(M_0)$ et $\tilde{\gamma} \in \Gamma(\tilde{L}_V)$.

  Si $F_v$ est non-archimédien, il résulte de la densité des distributions invariantes $I_{\tilde{G}}(\tilde{\gamma},\cdot)$ \cite[Proposition 4.2.6]{Li12b} que $I_{\tilde{M}}(\pi,X,f)=0$. Si $F_v$ est archimédien, la preuve de \cite[Theorem 6.1]{Ar88-TF1} permet de ramener, de façon assez formelle et purement locale, le cas de $I_{\tilde{M}}(\pi,X,\cdot)$ à celui des $I_{\tilde{L}}(\tilde{\gamma},\cdot)$. Malheureusement, on ne peut pas la reproduire ici puisque cela nécessiterait beaucoup d'espaces et formes linéaires auxiliaires.
\end{proof}

\subsection{Variante: l'application $\phi^1_{\tilde{M}}$}\label{sec:phi^1}
Dans cette sous-section, on suppose que $V \supset V_\infty$.

\begin{lemme}\label{prop:phi^1}
  Soit $M \in \mathcal{L}(M_0)$. Pour $f^1 \in \mathcal{H}_{\asp}(\tilde{G}_V, A_{G,\infty})$, on pose
  \begin{align*}
    \phi^1_{\tilde{M}}(f^1): \Pi_{\mathrm{temp}}(\tilde{M}_V, A_{M,\infty}) & \longrightarrow \C \\
    \pi & \longmapsto \left[ \pi \mapsto J_{\tilde{M}}(\pi, 0, f^1) \right].
  \end{align*}
  Alors $\phi^1_{\tilde{M}}$ induit une application linéaire continue $\mathcal{H}_{\asp}(\tilde{G}_V, A_{G,\infty}) \to I\mathcal{H}_{\asp}(\tilde{M}, A_{M,\infty})$.
\end{lemme}
Puisque nous utilisons la donnée centrale $(A_{G,\infty},\mathfrak{a}_G)$ (resp. $(A_{M,\infty}, \mathfrak{a}_M)$) pour $\tilde{G}_V$ (resp. pour $\tilde{M}_V$), il n'y a plus besoin d'introduire les espaces de fonctions à support presque compact.

Voici une remarque en passant: il y a des identifications naturelles $\Pi_-(\tilde{G}_V, A_{G,\infty})=\Pi_-(\tilde{G}^1_V)$, $\Pi_-(\tilde{M}_V, A_{M,\infty})=\Pi_-(\tilde{M}^1_V)$. Idem si l'on se limite aux représentations unitaires, tempérées, etc.

\begin{proof}
  Selon la définition de $\phi^1_{\tilde{M}}$, on a le diagramme commutatif
  $$\xymatrix{
    \mathcal{H}_{\text{ac},\asp}(\tilde{G}_V) \ar[rr] & & I\mathcal{H}_{\text{ac},\asp}(\tilde{M}_V) \ar[d]^{\text{res}} \\
    \mathcal{H}_{\asp}(\tilde{G}_V, A_{G,\infty}) \ar@{^{(}->}[u] \ar[rr] & & \{\text{fonctions } \Pi_{\text{temp},-}(\tilde{M}_V, A_{M,\infty}) \to \C \}
  }$$
  où les flèches horizontales sont $\phi_{\tilde{M}}$ et $\phi^1_{\tilde{M}}$. D'après la partie à droite du diagramme dans la Proposition \ref{prop:res-diagramme}, le but de $\text{res}$ peut être remplacé par $I\mathcal{H}_{\asp}(\tilde{M}_V, A_{M,\infty})$, et $\text{res}$ devient continue. Par conséquent, $\phi^1_{\tilde{M}}$ est la composée de trois applications linéaires continues, d'où l'assertion.
\end{proof}

\begin{proposition-fr}\label{prop:I-phi^1}
  Soient $M \in \mathcal{L}(M_0)$ et $f^1 \in \mathcal{H}_{\asp}(\tilde{G}_V, A_{G,\infty})$, alors on a
  $$ I_{\tilde{M}}(\pi,0,f^1) = J_{\tilde{M}}(\pi,0,f^1) - \sum_{\substack{L \in \mathcal{L}(M) \\ L \neq G}} I^{\tilde{L}}_{\tilde{M}}(\pi, 0, \phi^1_{\tilde{L}}(f^1)) $$
  pour tout $\pi \in \Pi_{\mathrm{unit},-}(\tilde{M}_V, A_{M,\infty})$.

  De même, on a
  $$ I_{\tilde{M}}(\tilde{\gamma},f^1) = J_{\tilde{M}}(\tilde{\gamma},f^1) - \sum_{\substack{L \in \mathcal{L}(M) \\ L \neq G}} I^{\tilde{L}}_{\tilde{M}}(\tilde{\gamma}, \phi^1_{\tilde{L}}(f^1)) $$
  pour tout $\tilde{\gamma} \in \Gamma(\tilde{M}_V)$.
\end{proposition-fr}
On prend garde que $I_{\tilde{M}}(\cdots)$ et $J_{\tilde{M}}(\cdots)$ sont définies en utilisant la donnée centrale $(A_{G,\infty}, \mathfrak{a}_G)$ de $\tilde{G}_V$, tandis que $I^{\tilde{L}}_{\tilde{M}}(\cdots)$ sont définies en utilisant la donnée centrale $(A_{L,\infty}, \mathfrak{a}_L)$ de $\tilde{L}_V$, pour tout $L \in \mathcal{L}(M)$, $L \neq G$.

\begin{proof}
  La Définition \ref{def:I-spec} appliquée à la donnée centrale $(A_{G,\infty}, \mathfrak{a}_G)$ donne
  $$ I_{\tilde{M}}(\pi,0,f^1) = J_{\tilde{M}}(\pi,0,f^1) - \sum_{\substack{L \in \mathcal{L}(M) \\ L \neq G}} I^{\tilde{L}}_{\tilde{M}}(\pi, 0, \phi_{\tilde{L}}(f^1)) $$
  où $\phi_{\tilde{L}}(f^1) \in I\mathcal{H}_{\text{ac},\asp}(\tilde{L}_V, A_{G,\infty})$.

  Le deuxième point de la Remarque \ref{rem:recurrence-spec} avec $X=0$ implique
  $$ I^{\tilde{L}}_{\tilde{M}}(\pi,0,\phi_{\tilde{L}}(f^1)) = I^{\tilde{L}}_{\tilde{M}}(\pi,0,\text{res}(\phi_{\tilde{L}}(f^1))),$$
  où le côté à droite est défini par rapport à la donnée centrale $(A_{L,\infty}, \mathfrak{a}_L)$ qui majore $(A_{G,\infty}, \mathfrak{a}_G)$. Or $\text{res}(\phi_{\tilde{L}}(f^1))=\phi^1_{\tilde{L}}(f^1)$ selon la définition de $\phi^1_{\tilde{L}}$. Cela permet de conclure le cas de $I_{\tilde{M}}(\pi,0,f^1)$. Le cas de $I_{\tilde{M}}(\tilde{\gamma},f^1)$ est analogue et plus simple.
\end{proof}

\begin{remarque}
  La construction s'inspire de \cite[p.192]{Ar02}. Les égalités dans cette Proposition donnent une caractérisation des formes linéaires $f^1 \mapsto I_{\tilde{M}}(\pi,0,f^1)$ sans référence aux espaces avec l'indice ``ac''.
\end{remarque}

\section{Développements fins révisés}\label{sec:developpements}
Dans cette section, on étudie un revêtement adélique à $m$ feuillets
$$ 1 \to \bmu_m \to \tilde{G} \xrightarrow{\rev} G(\A) \to 1, $$
muni des données auxiliaires $V \supset V_\text{ram}$, $M_0$, $K = \prod_v K_v$ comme dans \S\ref{sec:rev}. Ces données déterminent la distribution
$$ J: \mathcal{H}_{\asp}(\tilde{G}_V, A_{G,\infty}) \to \C $$
qui fait l'objet de la formule des traces grossière \cite{Li14a}. Elle admet deux développements, l'un géométrique et l'autre spectral. Les deux premières sous-sections sont consacrées à l'étude du côté géométrique, et le reste est consacré au côté spectral.

Adoptons la convention que les représentations locales sont désignées par des symboles usuels, eg. $\pi$, et les représentations automorphes sont dotées d'un $\circ$, eg. $\mathring{\pi}$.

\subsection{Le développement géométrique fin}
On commence en rappelant les résultats dans \cite[\S 6.5]{Li14a}. Soit $S$ un ensemble fini de places de $F$ tel que $S \supset V_\text{ram}$; on introduira un sous-ensemble $V \subset S$ plus loin. D'après Arthur, deux éléments $\gamma_1$, $\gamma_2 \in G(F)$ avec décompositions de Jordan $\gamma_i = \sigma_i u_i$ ($i=1,2$) sont dits $(G,S)$-équivalents s'il existe $x \in G(F)$ et $y \in G_{\sigma_2}(F_S)$ tels que
\begin{align*}
  x^{-1} \sigma_1 x & = \sigma_2, \\
  (xy)^{-1} u_1 xy & = u_2.
\end{align*}
Cette notion est plus grossière que $G(F)$-conjugaison, mais plus fine que la $\mathcal{O}$-équivalence \cite[\S 5.1]{Li14a}. En se rappelant la notion des bons éléments (la Définition \ref{def:bonte}), on pose
$$
  \begin{array}{ll}
  (G(F))_{G,S} & := \{\text{classes de $(G,S)$-équivalence dans $G(F)$}\}, \\
  (G(F))_{G,S}^K & := \left\{
    \begin{array}{ll}
      c \in (G(F))_{G,S}: & \exists \gamma = \sigma u \in c \;\text{(décomposition de Jordan),  où} \\
      & \sigma \text{ est $S$-admissible},\\
      & \sigma^S \in K^S.
    \end{array}
  \right\},\\
  (G(F))_{\tilde{G},S}^{K,\text{bon}} & := \{c \in (G(F))_{G,S}^K : c \text{ est bon dans } M(F_S) \}.
  \end{array}
$$
Un élément $\gamma$ satisfaisant à la condition définissant $(G(F))_{G,S}^K$ est dit un représentant admissible de $c$. Lorsqu'il n'y a pas de confusion à craindre, on utilisera un représentant admissible $\gamma$ au lieu de $c$ pour désigner un élément de $(G(F))_{G,S}^K$. 

Voici une procédure pour extraire la composante locale d'un élément dans $G(F) \subset \tilde{G}$. Pour $\gamma \in (G(F))_{\tilde{G},S}^{K,\text{bon}}$, confondu avec un représentant admissible, et $\widetilde{\gamma_S} \in \tilde{G}_S$, on écrit
\begin{gather}\label{eqn:leadsto}
  \gamma \leadsto \widetilde{\gamma_S}
\end{gather}
si $\gamma=\sigma u$ est la décomposition de Jordan et $\widetilde{\gamma_S} = \widetilde{\sigma_S} u_S$ où
\begin{itemize}
  \item $\sigma^S \in K^S \hookrightarrow \tilde{G}^S$, $\sigma = \widetilde{\sigma_S} \times \sigma^S$ dans $\tilde{G}$;
  \item $u_S \in G_\text{unip}(F_S) \hookrightarrow \tilde{G}_S$.
\end{itemize}
On observe que $\widetilde{\gamma_S} \in \tilde{G}_S^1$ s'il existe $\gamma$ tel que $\gamma \leadsto \widetilde{\gamma_S}$.

Il faut prendre garde que $\gamma \leadsto \widetilde{\gamma_S}$ est seulement une correspondance; elle ne définit pas forcément une application. Or on a la majoration suivante.
\begin{lemme}\label{prop:leadsto-majoration}
  Soit $\gamma \in (G(F))_{\tilde{G},S}^{K,\mathrm{bon}}$, alors
  $$ \#\{ \widetilde{\gamma_S} \in \Gamma(\tilde{G}_S^1) : \gamma \leadsto \widetilde{\gamma_S} \} \leq m. $$
\end{lemme}
\begin{proof}
  Soient $\gamma_1, \gamma_2$ deux représentants admissibles ayant la même classe dans $(G(F))_{\tilde{G},S}^{K,\text{bon}}$, ce qui donnent deux éléments locaux $\widetilde{\gamma_{i,S}} = \widetilde{\sigma_{i,S}} u_{i,S}$ ($i=1,2$). La définition de $(G,S)$-équivalence implique qu'il existe $t \in G(F_S)$ tels que $t^{-1} \gamma_{1,S} t = \gamma_{2,S}$. Puisqu'il y a exactement $m$ classes de conjugaison dans $\tilde{G}_S$ au-dessus d'une bonne classe de conjugaison dans $G(F_S)$, la majoration en découle.
\end{proof}

Rappelons que pour définir les distributions, il convient de fixer une mesure invariante sur les bonnes classes de $\tilde{M}_S$-conjugaison, d'où le $\R_{>0}$-torseur $\dot{\Gamma}(\tilde{G}_S) \to \Gamma(\tilde{G}_S)$. On désigne par $\dot{\widetilde{\gamma_S}}$ une image réciproque quelconque de la classe de $\widetilde{\gamma_S}$ pour ce torseur.

Soit $M \in \mathcal{L}(M_0)$, les définitions ci-dessus s'appliquent également à $M$. Tout élément  $\dot{\widetilde{\gamma_S}} \in \dot{\Gamma}(\tilde{M}_S)$ admet une décomposition de Jordan \cite[\S 6.3]{Li14a}
$$ \dot{\widetilde{\gamma_S}} = \dot{\widetilde{\sigma}} \dot{u} $$
parallèle à la décomposition de Jordan $\widetilde{\gamma_S} = \tilde{\sigma} u$ dans $M(F_S)$, où $\dot{u} \in \dot{\Gamma}(G_{\sigma_S})$. Si la mesure de Haar de $G_\sigma(F_S)$ est choisie, ce que l'on suppose, alors cette décomposition peut être rendue canonique. Rappelons maintenant la définition des coefficients géométriques \cite[Définition 6.5.5]{Li14a}.

\begin{definition-fr}\label{def:coeff-geom-0}
  Soit $\gamma$ une classe dans $(M(F))^{K^M,\text{bon}}_{\tilde{M},S}$, confondue avec un représentant admissible avec la décomposition de Jordan $\gamma=\sigma u$, ce qui définit $\gamma \leadsto \widetilde{\gamma_S}$. On pose
  \begin{gather*}
    \epsilon^M(\sigma) := \begin{cases} 1, & \text{si $\sigma$ est $F$-elliptique dans $M$},\\ 0,& \text{sinon}; \end{cases},  \\
    \text{Stab}(\sigma,u) := \{ t \in M_\sigma(F) \backslash M^\sigma(F) : t u t^{-1} \text{ est conjugué à } u \text{ dans } M_\sigma(F_S) \}, \\
    a^{\tilde{M}}(S, \dot{\widetilde{\gamma_S}}) := \epsilon^M(\sigma) |\text{Stab}(\sigma,u)|^{-1} a^{M_\sigma, [\cdot,\sigma]}(S,\dot{u}).
  \end{gather*}
  Ici, l'ellipticité de $\sigma$ dans $M$ signifie que $\mathfrak{a}_M = \mathfrak{a}_{M\sigma}$. Le terme $a^{M_\sigma, [\cdot,\sigma]}(S,\dot{u})$ est le coefficient géométrique fourni par \cite[Théorème 5.5.1]{Li14a} pour la formule des traces pour $M_\sigma(\A)$ tordue par le caractère-commutateur $[\cdot, \sigma]$ (voir \cite[\S 2.6]{Li14a}). On observe que si la classe unipotente triviale $1$ est $[\cdot,\sigma]$-bonne dans $M_\sigma(\A)$, ce qui revient au même de dire que $[\cdot,\sigma]$ est trivial sur $M_\sigma(\A)$, alors 
  $$ a^{M_\sigma, [\cdot,\sigma]}(S,1) = \mes(M_\sigma(F) A_{M_\sigma, \infty} \backslash M_\sigma(\A)); $$
  sinon ce coefficient est nul. On renvoie à \cite[Corollary 4.4 et 8.5]{Ar85} pour les détails.
\end{definition-fr}

Rappelons qu'Arthur a défini la notion de sous-ensembles admissibles de $G(F_S)$ (voir \cite[\S 5.6]{Li14a}). On dit qu'un sous-ensemble de $\Delta \subset \tilde{G}_S$ est admissible si $\rev(\Delta)$ l'est. L'admissibilité de $\Delta$ ne dépend que de $\rev(\Delta)$ modulo $Z_G(F_S)$. On pose
\begin{gather}\label{eqn:H-adm}
  \mathcal{H}_{\text{adm},\asp}(\tilde{G}_S, A_{G,\infty}) := \{ f \in \mathcal{H}_{\asp}(\tilde{G}_S, A_{G,\infty}) : \Supp(f) \text{ est admissible } \}.
\end{gather}
Il y a aussi une notion de $S$-admissibilité pour un sous-ensemble de $G(\A)$, donc de $\tilde{G}$. Cf. loc. cit.

Le développement géométrique fin s'énonce comme suit.
\begin{theoreme}\label{prop:dev-geom-fin}
  Soit $f_S^1 \in \mathcal{H}_{\mathrm{adm},\asp}(\tilde{G}_S, A_{G,\infty})$. On pose
  $$ \mathring{f}^1 := f_S^1 \cdot f_{K^S} \in \mathcal{H}_{\asp}(\tilde{G}, A_{G,\infty}) $$
  où $f_{K^S}$ est l'unité de l'algèbre de Hecke sphérique anti-spécifique $H_{\asp}(\tilde{G}^S /\!/ K^S)$. Alors
  \begin{gather}\label{eqn:dev-geom-fin}
    J(\mathring{f}^1) = \sum_{M \in \mathcal{L}(M_0)} \frac{|W^M_0|}{|W^G_0|} \sum_{\substack{\gamma \in (M(F))_{\tilde{M},S}^{K^M,\mathrm{bon}} \\ \gamma \leadsto \widetilde{\gamma_S}}} a^{\tilde{M}}(S, \dot{\widetilde{\gamma_S}}) J_{\tilde{M}}(\dot{\widetilde{\gamma_S}}, f_S^1),
  \end{gather}
  où le $\widetilde{\gamma_S}$ est un élément quelconque dans $\Gamma(\tilde{M}_S)$ tel que $\gamma \leadsto \widetilde{\gamma_S}$. Le terme $a^{\tilde{M}}(S, \dot{\widetilde{\gamma_S}}) J_{\tilde{M}}(\dot{\widetilde{\gamma_S}}, \cdot)$ est indépendant du choix de $\widetilde{\gamma_S}$ et $\dot{\widetilde{\gamma_S}}$.
\end{theoreme}
\begin{proof}
  C'est essentiellement \cite[Théorème 6.5.9]{Li14a}. Là, on demande que $S$ est suffisamment grand par rapport au support de $f_S^1$; le sens précis est contenu dans \cite[Théorème 6.5.8]{Li14a} et les remarques qui le suivent. On vérifie qu'il suffit de prendre $S \supset V_\text{ram}$ tel que $f_S^1 \in \mathcal{H}_{\mathrm{adm},\asp}(\tilde{G}_S, A_{G,\infty})$.
\end{proof}
Il sera utile d'isoler les termes dans \eqref{eqn:dev-geom-fin} avec $M=G$.
\begin{gather}\label{eqn:I_ell}
  I_\text{ell}(f^1_S) := \sum_{\substack{\gamma \in (G(F))_{\tilde{G},S}^{K,\text{bon}} \\ \gamma \leadsto \widetilde{\gamma_S}}} a^{\tilde{G}}(S, \dot{\widetilde{\gamma_S}}) J_{\tilde{G}}(\dot{\widetilde{\gamma_S}}, f_S^1), \quad f^1_S \in \mathcal{H}_{\text{adm},\asp}(\tilde{G}_S, A_{G,\infty}).
\end{gather}

L'étape suivante est d'indexer ce développement selon les classes dans $\Gamma(\tilde{M}_V)$, $M \in \mathcal{L}(M_0)$.

\subsection{Compression des coefficients géométriques}
Désormais, on écrit $\tilde{\gamma}$ au lieu de $\widetilde{\gamma_S}$ pour désigner un élément ou une classe de conjugaison dans $\tilde{G}_S$, afin d'alléger les notations. Rappelons que dans \S\ref{sec:dists}, on a identifié $\Gamma(\tilde{G}_S^1)$ à une base de $\mathcal{D}(\tilde{G}_S, A_{G,\infty})$. Autrement dit, on choisit une image réciproque dans $\dot{\Gamma}(\tilde{G}_S^1)$ pour chaque élément de $\Gamma(\tilde{G}_S^1)$. La notation $\dot{\gamma}$ n'est plus utilisée.

Soit $\tilde{\gamma} \in \Gamma(\tilde{G}^1_S)$. Vu la convention ci-dessus et le Lemme \ref{prop:leadsto-majoration}, on peut définir les coefficients elliptiques
\begin{gather}\label{eqn:coeff-ell}
  a^{\tilde{G}}_\text{ell}(\tilde{\gamma}) := \frac{1}{m} \sum_{\noyau \in \bmu_m} \; \sum_{\substack{\gamma \in (G(F))_{\tilde{G},S}^{K,\mathrm{bon}} \\ \gamma \leadsto \noyau\tilde{\gamma}}} \frac{\noyau^{-1} \cdot a^{\tilde{G}}(S, \noyau\tilde{\gamma})}{\#\{ \tilde{\eta} \in \Gamma(\tilde{G}_S) : \gamma \leadsto \tilde{\eta} \}}.
\end{gather}

Ce coefficient ne dépend que de $\tilde{\gamma}$. Il est défini de sorte que les assertions suivantes soient satisfaites.
\begin{lemme}\label{prop:a_ell}
  On a la caractérisation suivante de $a^{\tilde{G}}_{\mathrm{ell}}(\cdot)$.
  \begin{gather*}
    a^{\tilde{G}}_{\mathrm{ell}}(\noyau\tilde{\gamma}) = \noyau a^{\tilde{G}}_{\mathrm{ell}}(\tilde{\gamma}), \quad \noyau \in \bmu_m, \; \tilde{\gamma} \in \Gamma(\tilde{M}^1_S); \\
    \sum_{\tilde{\gamma} \in \Gamma(\tilde{M}^1_S)} a^{\tilde{G}}_{\mathrm{ell}}(\tilde{\gamma}) J_{\tilde{G}}(\tilde{\gamma}, f^1_S) = I_{\mathrm{ell}}(f^1_S), \quad f^1_S \in \mathcal{H}_{\mathrm{adm},\asp}(\tilde{G}_S, A_{G,\infty}).
  \end{gather*}
\end{lemme}

On définit $\Gamma_\text{ell}(\tilde{G}^1, S)$ comme l'ensemble de $\tilde{\gamma} \in \Gamma(\tilde{G}_S^1)$ tel qu'il existe $\mathring{\gamma} \in G(F)$ satisfaisant à
\begin{itemize}
  \item la partie semi-simple de $\mathring{\gamma}$ est $F$-elliptique dans $G$;
  \item $\mathring{\gamma}$ est un représentant admissible d'une classe dans $(G(F))_{\tilde{G},S}^{K,\text{bon}}$;
  \item il existe $\noyau \in \bmu_m$ tel que $\mathring{\gamma} \leadsto \noyau\tilde{\gamma}$.
\end{itemize}
On peut vérifier que $\Gamma_\text{ell}(\tilde{G}^1, S)$ est un sous-ensemble discret de l'espace topologique  $\Gamma(\tilde{G}_S^1)$. Le coefficient elliptique $a^{\tilde{G}}_\text{ell}(\cdot)$ est à support dans $\Gamma_\text{ell}(\tilde{G}^1, S)$.

On en déduira les autres coefficients du côté spectral et leurs domaines. Pour ce faire, il convient d'introduire une notion d'induction pour les classes de conjugaison.

\begin{definition-fr}\label{def:induction-orbites}
  Soit $M \in \mathcal{L}(M_0)$, on définit une application
  \begin{align*}
    \Gamma(\tilde{M}_S^1) & \longrightarrow \mathcal{D}(\tilde{G}_S, A_{G,\infty}) \\
    \tilde{\gamma} & \longmapsto \tilde{\gamma}^G,
  \end{align*}
  de la manière suivante. Rappelons que l'orbite de $\tilde{\gamma}$ est supposée munie d'une mesure invariante. Soit $\tilde{\gamma} = \dot{\tilde{\sigma}} \dot{u}$ une décomposition de Jordan munie de mesures invariantes. On a le caractère continu $[\cdot,\sigma]$ de $G_\sigma(F_S)$. On définit ensuite
  $$ \tilde{\gamma}^G := \dot{\tilde{\sigma}} \dot{u}^G, $$
  où $\dot{u} \mapsto \dot{u}^G$ est l'induction de Spaltenstein-Lusztig envoyant un élément de $\dot{\Gamma}_\text{unip}(M_\sigma(F_S))^{[\cdot,\sigma]}$ sur une combinaison linéaire d'éléments de $\dot{\Gamma}_\text{unip}(G_\sigma(F_S))^{[\cdot,\sigma]}$: voir \cite[Lemme 5.3.4]{Li14a}. L'expression $\dot{\tilde{\sigma}} \dot{u}^G$ définit alors un élément de $\mathcal{D}(\tilde{G}_S, A_{G,\infty})$ via la décomposition de Jordan munie de mesures invariantes.

  Soient $\tilde{\mu} \in \Gamma(\tilde{M}_S^1)$ et $\tilde{\gamma} \in \Gamma(\tilde{G}_S^1)$. On pose
  \begin{gather}\label{eqn:mult-geom}
    (\tilde{\mu}^G : \tilde{\gamma}) := \text{le coefficient de } \tilde{\gamma} \text{ dans } \tilde{\mu}^G .
  \end{gather}
\end{definition-fr}

On définit maintenant
\begin{gather}
  \Gamma(\tilde{G}^1,S) := \{\tilde{\gamma} \in \Gamma(\tilde{G}_S^1) : \exists M \in \mathcal{L}(M_0), \; \tilde{\mu} \in \Gamma_\text{ell}(\tilde{M}^1, S) \; \text{ tels que } (\tilde{\mu}^G : \tilde{\gamma}) \neq 0 \}.
\end{gather}
C'est toujours un sous-ensemble discret dans $\Gamma(\tilde{G}^1_S)$.

Pour définir les coefficients généraux, on change légèrement les notations: on considère deux ensembles finis $S$, $V$ des places de $F$ tels que
$$
  S \supset V \supset V_\text{ram}.
$$
Dans ce qui suit, $V$ sera fixé et on variera $S$.

\begin{definition-fr}\label{def:coeff-geom}
  Soit $\tilde{\gamma} \in \Gamma(\tilde{G}_V^1)$. On prend $S$ tel que $\{\gamma\} \times K^V$ est un ensemble $S$-admissible dans $G(\A)$. On pose
  $$ a^{\tilde{G}}(\tilde{\gamma}) :=  \sum_{M \in \mathcal{L}(M_0)} \frac{|W^M_0|}{|W^G_0|} \sum_{k \in \mathcal{K}_\text{ell}^V(\tilde{M},S)} \sum_{\tilde{\mu} \in \Gamma_\text{ell}(\tilde{M}^1, V)} (\tilde{\mu}^G : \tilde{\gamma}) a^{\tilde{M}}_\text{ell}(\tilde{\mu} \times k) r^{\tilde{G}}_{\tilde{M}}(k), $$
  où
  \begin{itemize}
    \item $\mathcal{K}(\tilde{M}^V_S)$ est l'ensemble des classes de conjugaison dans $\tilde{M}^V_S$ qui rencontrent $\tilde{M}^V_S \cap \tilde{K}^V_S$, de telles classes sont munies de mesures invariantes selon nos conventions dans \S\ref{sec:dists};
    \item pour tout $k \in \mathcal{K}(\tilde{M}^V_S)$, on définit l'intégrale orbitale pondérée anti-spécifique non ramifiée
    $$ r^{\tilde{G}}_{\tilde{M}}(k) := J_{\tilde{M}}(k, f_{K^V_S}) $$
    comme dans \cite[Définition 6.3.3]{Li14a}, où $f_{K^V_S}$ est l'unité de l'algèbre de Hecke sphérique anti-spécifique $\mathcal{H}_{\asp}(\tilde{G}^V_S /\!/ K^V_S)$;
    \item $\mathcal{K}_\text{ell}^V(\tilde{M},S)$ est le sous-ensemble de $\mathcal{K}(\tilde{M}^V_S)$ formé des éléments $k$ tels qu'il existe $\tilde{\mu} \in \Gamma(\tilde{M}_V^1)$ vérifiant
    $$ \tilde{\mu} \times k \in \Gamma_\text{ell}(\tilde{M}^1,S); $$
    dans ce cas-là, on a également $\tilde{\mu} \in \Gamma_\text{ell}(\tilde{M}^1, V)$.
  \end{itemize}
\end{definition-fr}
On voit que cette somme est finie pour tout $\tilde{\gamma}$. Ceci résulte, pour l'essentiel, du fait que $a^{\tilde{M}}_\text{ell}(\cdot): \Gamma(\tilde{M}^1_S) \to \C$ est à support dans le sous-ensemble discret $\Gamma_{\text{ell}}(\tilde{M}^1, S)$. Cf. \cite[p.195]{Ar02}.

\begin{remarque}
  Il en résulte immédiatement que $a^{\tilde{G}}(\cdot)$ est à support dans $\Gamma(\tilde{G}^1,V)$. D'ailleurs, il résultera de la Proposition \ref{prop:coeff-geom-indep} que $a^{\tilde{G}}(\tilde{\gamma})$ ne dépend pas du choix de $S$, ce qui justifie la notation.
\end{remarque}

\begin{theoreme}\label{prop:dev-geom-nouveau}
  Soient $V \supset V_\mathrm{ram}$ un ensemble fini de places de $F$ et $f^1 \in \mathcal{H}_{\asp}(\tilde{G}_V, A_{G,\infty})$. On pose $\mathring{f}^1 := f^1 f_{K^V} \in \mathcal{H}_{\asp}(\tilde{G}, A_{G,\infty})$ et $J(f^1) := J(\mathring{f}^1)$. Alors on a
  $$ J(f^1) = \sum_{L \in \mathcal{L}(M_0)} \frac{|W^L_0|}{|W^G_0|} \sum_{\tilde{\gamma} \in \Gamma(\tilde{L}^1,V)} a^{\tilde{L}}(\tilde{\gamma}) J_{\tilde{L}}(\tilde{\gamma}, f^1). $$
\end{theoreme}
\begin{proof}
  Prenons $S \supset V$ suffisamment grand de sorte que $f_S^1 := f^1 f_{K_S^V}$ appartient à $\mathcal{H}_{\text{adm},\asp}(\tilde{G}_S, A_{G,\infty})$, ce qui est toujours possible. Comme dans le Lemme \ref{prop:a_ell}, il résulte du Théorème \ref{prop:dev-geom-fin} que
  $$ J(\mathring{f}^1) = \sum_{M \in \mathcal{L}(M_0)} \frac{|W^M_0|}{|W^G_0|} \sum_{\tilde{\gamma}_S \in \Gamma_\text{ell}(\tilde{M}^1,S)} a^{\tilde{M}}_\text{ell}(\tilde{\gamma}_S) J_{\tilde{M}}(\tilde{\gamma}_S, f_S^1) $$
  d'après \eqref{eqn:coeff-ell}.

  On écrit $\tilde{\gamma}_S = \tilde{\mu} \times \tilde{\gamma}^V_S$ où $\tilde{\mu} \in \Gamma(\tilde{M}_V)$. Alors
  \begin{align*}
    J_{\tilde{M}}(\tilde{\gamma}_S, f_S^1) & = \sum_{L,L_1 \in \mathcal{L}(M)} d^G_M(L,L_1) J^{\tilde{L}_1}_{\tilde{M}}(\tilde{\mu}, (f^1)_{\tilde{Q}_1}) J^{\tilde{L}}_{\tilde{M}}(\tilde{\gamma}^V_S, (f_{K^V_S})_{\tilde{Q}}) \\
    & = \sum_{L \in \mathcal{L}(M)} \left( \sum_{L_1 \in \mathcal{L}(M)} d^G_M(L,L_1) J^{\tilde{L}_1}_{\tilde{M}}(\tilde{\mu}, (f^1)_{\tilde{Q}_1}) \right) J^{\tilde{L}}_{\tilde{M}}(\tilde{\gamma}^V_S, (f_{K^V_S})_{\tilde{Q}}) \\
    & = \sum_{L \in \mathcal{L}(M)} J_{\tilde{L}}(\tilde{\mu}^L, f^1) J^{\tilde{L}}_{\tilde{M}}(\tilde{\gamma}^V_S, f_{K^V_S \cap \tilde{L}}),
  \end{align*}
  où on a utilisé les formules de descente \cite[Proposition 6.4.3 et 6.4.2]{Li14a}, y compris une application $(L,L_1) \mapsto (Q \in \mathcal{P}(L), Q_1 \in \mathcal{P}(L_1))$ dépendant de choix auxiliaires que nous ne précisons pas, ainsi qu'un calcul simple qui montre que la descente parabolique $(f_{K^V_S})_{\tilde{Q}}$ est égale à $f_{K^V_S \cap \tilde{L}}$. Voir \cite[\S 6.4]{Li14a} pour la définition de descente parabolique $f^1 \mapsto (f^1)_{\tilde{Q}_1}$ et $f_{K^V_S} \mapsto (f_{K^V_S})_{\tilde{Q}}$.

  L'intégrale orbitale pondérée anti-spécifique $J^{\tilde{L}}_{\tilde{M}}(\tilde{\gamma}^V_S, f_{K^V_S \cap \tilde{L}})$ est nulle sauf si $\tilde{\gamma} = k \in \mathcal{K}(\tilde{M}^V_S)$; dans ce cas-là c'est égal à $r^{\tilde{L}}_{\tilde{M}}(k)$. Donc $J(\mathring{f}^1)$ s'écrit comme
  $$ \sum_{L \in \mathcal{L}(M_0)} \sum_{M \in \mathcal{L}^L(M_0)} \frac{|W^M_0|}{|W^G_0|} \sum_{\tilde{\mu},k} a^{\tilde{M}}_\text{ell}(\tilde{\mu} \times k) r^{\tilde{L}}_{\tilde{M}}(k) J_{\tilde{L}}(\tilde{\mu}^L, f^1) $$
  où $(\mu,k)$ parcourt $\Gamma_\text{ell}(\tilde{M}^1,V) \times \mathcal{K}_\text{ell}^V(\tilde{M},S)$. En décomposant $\tilde{\mu}^L$ selon $\{\tilde{\gamma} \in \Gamma(\tilde{L}^1,V)\}$ (rappelons \eqref{eqn:mult-geom}) et en regroupant cette somme, on arrive à
  $$ \sum_L \frac{|W^L_0|}{|W^G_0|} \sum_{\tilde{\gamma}} \left( \sum_{M \in \mathcal{L}^L(M_0)} \frac{|W^M_0|}{|W^L_0|} \sum_{k, \tilde{\mu}} (\tilde{\mu}^L : \tilde{\gamma}) a^{\tilde{M}}_\text{ell}(\tilde{\mu} \times k) r^{\tilde{L}}_{\tilde{M}}(k) \right) J_{\tilde{L}}(\tilde{\gamma}, f^1). $$

  En se rappelant la définition d'admissibilité, on voit que les $\tilde{\gamma}$ pour lesquels il existe $\tilde{\mu} \in \Gamma_\text{ell}(\tilde{M}^1,V)$ avec $(\tilde{\mu}^L : \tilde{\gamma}) \neq 0$ satisfont à la condition de la Définition \ref{def:coeff-geom}, donc $a^{\tilde{L}}(\tilde{\gamma})$ est défini et égal à l'expression dans la parenthèse. Cela permet de conclure.
\end{proof}

\begin{remarque}\label{rem:V-S-convention}
  La possibilité $V=S$ n'est pas exclue. Dans ce cas-là, on pose $\mathcal{K}(\tilde{M}^V_S)=\mathcal{K}_\text{ell}^V(\tilde{M},S)=\{*\}$, $\tilde{\mu} \times k = \tilde{\mu}$ et $r^G_M(*)=1$ si $G=M$, sinon $r^G_M(*)=0$. Les arguments ci-dessus demeurent valables avec cette convention.
\end{remarque}

\begin{proposition-fr}\label{prop:coeff-geom-indep}
  Les coefficients $a^{\tilde{G}}(\tilde{\gamma})$ sont indépendants du choix de $S$.
\end{proposition-fr}
\begin{proof}
  Observons tout d'abord que
  $$ J_{\tilde{G}}(\noyau\tilde{\gamma}, f) = \noyau^{-1} J_{\tilde{G}}(\tilde{\gamma}, f), \quad \noyau \in \bmu_m $$
  pour toute $f$ anti-spécifique, tandis que les coefficients ont l'équivariance opposée
  $$ a^{\tilde{G}}(\noyau\tilde{\gamma}) = \noyau a^{\tilde{G}}(\tilde{\gamma}), \quad \noyau \in \bmu_m $$
  d'après le Lemme \ref{prop:a_ell} et la Définition \ref{def:coeff-geom}. L'invariance des distributions $J_{\tilde{G}}(\tilde{\gamma}, \cdot)$ et lesdites propriétés d'équivariance permettent de montrer que lorsque $G/Z_G$ est anisotrope, les $a^{\tilde{G}}(\cdot)$ sont déterminés par la distribution spécifique $f^1 \mapsto J(f^1)$, qui est bien sûr indépendante de $S$.

  On raisonne par récurrence et on suppose que $a^{\tilde{L}}(\cdot)$ est indépendant de $S$, pour tout $L \in \mathcal{L}(M_0)$, $L \neq G$ (on vient d'établir le cas $L=M_0$). Les distributions $J_{\tilde{L}}(\tilde{\gamma}, \cdot)$ sont indépendantes de $S$. Vu le Théorème \ref{prop:dev-geom-nouveau}, la distribution
  $$ \sum_{\tilde{\gamma}} a^{\tilde{G}}(\tilde{\gamma}) J_{\tilde{G}}(\tilde{\gamma},\cdot) $$
  l'est aussi. Or $J_{\tilde{G}}(\tilde{\gamma},\cdot)$ n'est que l'intégrale orbitale invariante anti-spécifique. On raisonne comme dans le cas $G/Z_G$ anisotrope pour montrer que $a^{\tilde{G}}(\tilde{\gamma})$ est déterminé par $J$, d'où l'indépendance cherchée.
\end{proof}


\subsection{Le développement spectral fin}
Tout d'abord, on rappelle le développement spectral fin obtenu dans \cite{Li13}. Soit $\mathring{f}^1 \in \mathcal{H}_{\asp}(\tilde{G}, A_{G,\infty})$. On décompose
$$ J(\mathring{f}^1) = \sum_{t \geq 0} J_t(\mathring{f}^1) $$
où
$$ J_t(\mathring{f}^1) := \sum_{\substack{\chi \in \mathfrak{X} \\ \|\Im(\nu_\chi)=t\|}} J_\chi(\mathring{f}^1)  \quad \text{(somme finie)}. $$
Ici $\mathfrak{X} = \mathfrak{X}^{\tilde{G}}$ est l'ensemble des données cuspidales automorphes et $\nu_\chi$ est essentiellement le caractère infinitesimal associé à $\chi$. On renvoie à \cite[\S 6]{Li14a} et \cite{Li13} pour les définitions précises. Le paramètre $t$ est introduit pour des raisons de convergence absolue; on le discutera dans la Remarque \ref{rem:convergence-spec}.

Étant donné $t$, on définit des ensembles de représentations unitaires spécifiques de $\tilde{G}^1$:
\begin{align*}
  \Pi_{\text{disc},t,-}(\tilde{G}^1) & := \Pi_{\text{disc},t}(\tilde{G}^1) \cap \Pi_-(\tilde{G}^1),\\
  \Pi_{t,-}(\tilde{G}^1) & := \Pi_t(\tilde{G}^1) \cap \Pi_-(\tilde{G}^1);
\end{align*}
où $\Pi_{\text{disc},t}(\tilde{G}^1)$ et $\Pi_t(\tilde{G}^1)$ sont définis dans \cite[\S 7]{Li13}. On a $\Pi_{\text{disc},t,-}(\tilde{G}^1) \subset \Pi_{t,-}(\tilde{G}^1)$. Comme dans \S\ref{sec:dists}, on préfère de regarder les représentations de $\tilde{G}^1$ comme des représentations de $\tilde{G}$ telles que $A_{G,\infty}$ opère trivialement.

Soit $P=MU \in \mathcal{F}(M_0)$. On note $L_\text{disc}^2(A_{M,\infty}M(F) \backslash \tilde{M})$ la partie discrète de la représentation $R_M$ de $\tilde{M}$ sur $L^2(A_{M,\infty}M(F) \backslash \tilde{M})$. Plus généralement, on considère aussi sa variante tordue par $\lambda \in \mathfrak{a}^*_{M,\C}$.
\begin{gather}\label{eqn:L2-M}
  L_\text{disc}^2(A_{M,\infty}M(F) \backslash \tilde{M}) \otimes e^{\angles{\lambda, H_{\tilde{M}}(\cdot)}}.
\end{gather}

Soit $\lambda \in i(\mathfrak{a}^G_M)^*$. On désigne par $\mathcal{I}_{\tilde{P},\text{disc},t}(\lambda)$ la représentation intersection de la $t$-part de $L^2(A_{G,\infty}G(F) \backslash \tilde{G})$ et l'induite parabolique normalisée de \eqref{eqn:L2-M}. Dans \cite[\S 7]{Li13}, on a défini les coefficients discrets
$$ a^{\tilde{G}}_\text{disc}(\mathring{\pi}), \quad \mathring{\pi} \in \Pi_{\text{disc},t,-}(\tilde{G}^1). $$
Il sera commode de poser
\begin{gather}\label{eqn:coeff-spec-translation}
  a^{\tilde{G}}_\text{disc}(\mathring{\pi}_\lambda) := a^{\tilde{G}}_\text{disc}(\mathring{\pi}), \quad \mathring{\pi} \in \Pi_{\text{disc},t,-}(\tilde{G}^1), \lambda \in i\mathfrak{a}_G^*.
\end{gather}

Les coefficients discrets sont caractérisés par l'égalité
\begin{multline}\label{eqn:coeff-spec-caracterisation}
  \sum_{M \in \mathcal{L}(M_0)} \frac{|W^M_0|}{|W^G_0|} \sum_{s \in W^G(M)_\text{reg}} |\det(1-s|\mathfrak{a}^G_M)|^{-1} \Tr\left(M_{P|P}(s,0) \mathcal{I}_{\tilde{P},\text{disc},t}(0, \mathring{f}^1) \right) \\
  = \sum_{\mathring{\pi} \in \Pi_{\text{disc},t,-}(\tilde{G}^1)} a^{\tilde{G}}_\text{disc}(\mathring{\pi}) \Tr\mathring{\pi}(\mathring{f}^1)
\end{multline}
où
\begin{itemize}
  \item $P \in \mathcal{P}(M)$ est arbitraire;
  \item $M_{P|P}(s,0)$ est l'un des opérateurs d'entrelacement fournis par la théorie des série d'Eisenstein \cite[\S 5]{Li13}.
\end{itemize}
C'est une somme finie si l'on fixe un ensemble fini $\Gamma$ des $\tilde{K}$-types et $\mathring{f}^1 \in \mathcal{H}_{\asp}(\tilde{G}^1)_\Gamma$.


Fixons $P=MU$. Les opérateurs suivants forment l'ingrédient crucial du développement spectral fin
\begin{gather}
  \mathcal{J}_Q(\tilde{P}, \lambda, \Lambda) := M_{Q|P}(\lambda)^{-1} M_{Q|P}(\lambda+\Lambda), \quad Q \in \mathcal{P}(M),
\end{gather}
où $\lambda, \Lambda \in i\mathfrak{a}_M^*$ et $M_{Q|P}$ est encore l'un des opérateurs d'entrelacement fournis par la théorie des série d'Eisenstein (dans \cite{Li13}, on les note $\mathcal{M}_Q(\tilde{P}, \lambda, \Lambda)$). Ils forment une $(G,M)$-famille à valeurs dans les opérateurs d'auto-entrelacement de l'induite parabolique normalisée de \eqref{eqn:L2-M}.

Soit $\mathring{\pi} \in \Pi_{\text{disc},t,-}(\tilde{M}^1)$, identifiée à une représentation de $\tilde{M}$ comme précédemment, on note
\begin{gather}\label{eqn:J_Q-global}
  \mathcal{J}_Q(\tilde{P}, \mathring{\pi}_\lambda, \Lambda) := \mathcal{J}_Q(\tilde{P}, \lambda, \Lambda) \text{ restreinte à } \mathcal{I}_{\tilde{P}}(\mathring{\pi}_\lambda) .
\end{gather}
D'après la théorie de $(G,M)$-familles, on en déduit l'opérateur d'entrelacement $\mathcal{J}_M(\mathring{\pi}_\lambda, \tilde{P})$.

\begin{definition-fr}
  Soient $P=MU$, $\mathring{\pi}$, $\lambda$ comme ci-dessus. On suppose de plus que $\lambda \in i(\mathfrak{a}^G_M)^*$. Pour $\mathring{f}^1 \in \mathcal{H}_{\asp}(\tilde{G}, A_{G,\infty})$, on pose
  \begin{align*}
    \mathcal{I}_{\tilde{P}}(\mathring{\pi}_\lambda, \mathring{f}^1) & := \int_{\tilde{G}/A_{G,\infty}} \mathring{f}^1(\tilde{x}) \mathcal{I}_{\tilde{P}}(\mathring{\pi}_\lambda, \tilde{x}) \dd\tilde{x}, \\
    J_{\tilde{M}}(\mathring{\pi}_\lambda, \mathring{f}^1) & := \Tr\left( \mathcal{J}_M(\mathring{\pi}_\lambda, \tilde{P}) \mathcal{I}_{\tilde{P}}(\mathring{\pi}_\lambda, \mathring{f}^1) \right),
  \end{align*}
  qui sont bien définis. En fait, la trace se calcule sur un espace vectoriel de dimension finie et $J_{\tilde{M}}(\mathring{\pi}_\lambda, \mathring{f}^1)$ est indépendant du choix de $P \in \mathcal{P}(M)$: il suffit de reprendre la preuve du cas local \cite[Lemme 5.7.1]{Li12b} et utiliser l'équation fonctionnelle pour les opérateurs d'entrelacement globaux $M_{Q|P}(\cdot)$.
\end{definition-fr}

Le développement spectral fin s'énonce comme suit.
\begin{theoreme}\label{prop:dev-spec-fin}
  Soient $t \geq 0$, $\mathring{f}^1 \in \mathcal{H}_{\asp}(\tilde{G}, A_{G,\infty})$. Alors
  $$ J_t(\mathring{f}^1) = \sum_{L \in \mathcal{L}(M_0)} \frac{|W^L_0|}{|W^G_0|} \sum_{\mathring{\pi} \in \Pi_{\mathrm{disc},t,-}(\tilde{L}^1)} \; \int_{i(\mathfrak{a}^G_L)^*} a^{\tilde{L}}_{\mathrm{disc}}(\mathring{\pi}_\lambda) J_{\tilde{L}}(\mathring{\pi}_\lambda, \mathring{f}^1) \dd\lambda $$
  comme une intégrale absolument convergente.
\end{theoreme}
\begin{proof}
  Le point de départ est \cite[Théorème 6.9]{Li13} qui donne
  \begin{multline*}
    J_t(\mathring{f}^1) = \sum_{L \in \mathcal{L}(M_0)} \frac{|W^L_0|}{|W^G_0|} \sum_{M \in \mathcal{L}^L(M_0)} \frac{|W^M_0|}{|W^L_0|} \sum_{s \in W^L(M)_\text{reg}} |\det(1-s|\mathfrak{a}^L_M)|^{-1} \\
    \sum_{\mathring{\sigma} \in \Pi_{\text{unit},-}(\tilde{M}^1)} \; \int_{i(\mathfrak{a}^G_L)^*} \Tr\left( \mathcal{J}_L(\tilde{P},\lambda) M_{P|P}(s,0) \mathcal{I}_{\tilde{P},\text{disc},t}(\lambda, \mathring{f}^1)_{\mathring{\sigma}} \right) \dd\lambda
  \end{multline*}
  comme une intégrale absolument convergente, où $P \in \mathcal{P}(M)$ est quelconque. La représentation $\mathcal{I}_{\tilde{P},\text{disc},t}(\lambda)_{\mathring{\sigma}}$ est la $t$-part de l'induite de la composante $\mathring{\sigma}$-isotypique de \eqref{eqn:L2-M}. L'expression $\Tr(\cdots)$ est indépendante de $P$ d'après \cite[Corollaire 6.10]{Li13}. L'opérateur $\mathcal{J}_L(\tilde{P},\lambda)$ (avec $\lambda \in i(\mathfrak{a}^G_L)^*$) se déduit de la $(G,L)$-famille
  \begin{gather}\label{eqn:G-L}
    \forall Q \in \mathcal{P}(L), \; \mathcal{J}_Q(\tilde{P}, \lambda, \Lambda) := \mathcal{J}_R(\tilde{P}, \lambda, \Lambda), \;  \text{ où } R \in \mathcal{P}(M), R \subset Q \text{ est quelconque}.
  \end{gather}
  Cf. \cite[\S 4.2]{Li14a}.

  On fixe $M \subset L$ et un sous-groupe parabolique $R^L \in \mathcal{P}^L(M)$. Rappelons qu'avec ces choix, il y a une application $\mathcal{P}(L) \to \mathcal{P}(M)$ envoyant $Q \in \mathcal{P}(L)$ sur l'unique élément $R = Q(R^L) \in \mathcal{P}(M)$ vérifiant $R \subset Q$ et $R \cap L = R^L$. Plus explicitement, $Q(R^L)$ est un produit semi-direct
  $$ Q(R^L) = R^L U_Q .$$
  Dans ce qui suit, on fixe $Q_0 \in \mathcal{P}(L)$ et on suppose que $P = Q_0(R^L)$.

  Soient $s \in W^L(M)_\text{reg}$, $\lambda \in i(\mathfrak{a}^G_L)^*$, considérons l'opérateur
  \begin{gather}\label{eqn:op-trace}
    \mathcal{J}_L(\tilde{P},\lambda) M_{P|P}(s,0) \mathcal{I}_{\tilde{P},t}(\lambda, \mathring{f}^1).
  \end{gather}

  En induisant par étapes, $M_{P|P}(s,0) \mathcal{I}_{\tilde{P},t}(\lambda, \mathring{f}^1)$ s'identifie à $\mathcal{I}_{\tilde{Q}_0}(M^L_{R^L|R^L}(s,0))$ composé avec $\mathcal{I}_{\tilde{Q}_0,t}(\lambda, \mathring{f}^1)$. D'autre part, pour calculer la trace de \eqref{eqn:op-trace}, il suffit de regarder l'action de $\mathcal{J}_L(\tilde{P},\lambda)$ sur les espaces de la forme $\mathcal{I}_{\tilde{Q}_0}(\mathring{\pi})$ où $\mathring{\pi} \in \Pi_{\text{disc},t,-}(\tilde{L}^1)$, cf. la définition de $\Pi_{\text{disc},t,-}(\tilde{L}^1)$. Dans la description \eqref{eqn:G-L} de la $(G,L)$-famille $\{\mathcal{J}_Q(\tilde{P},\lambda,\Lambda) : Q \in \mathcal{P}(L)\}$, on peut prendre $R \in \mathcal{P}(M)$ de la forme $R = Q(R^L)$. Montrons le fait suivant (cf. \cite[pp.520-521]{Ar88-TF2})
  $$ M_{R|P}(\lambda+\Lambda) = M_{Q(R^L)|Q_0(R^L)}(\lambda+\Lambda) = M_{Q|Q_0}(\lambda+\Lambda) \quad \text{ sur } \mathcal{I}_{\tilde{Q}_0}(\mathring{\pi}), $$
  où $\Lambda \in i\mathfrak{a}_L^*$.

  En effet, la description de $Q \mapsto Q(R^L)$ montre que $\Sigma_{Q(R^L)}$ est l'union disjointe de $\Sigma^L_{R^L}$ et $\Sigma_Q$; idem pour $Q_0(R^L)$. Donc
  $$ \Sigma_{Q_0(R^L)} \cap (-\Sigma_{Q(R^L)}) = \Sigma_{Q_0} \cap (-\Sigma_Q). $$
  Soit $\mu \in \mathfrak{a}^*_{L,\C}$ tel que
  $$ \angles{\Re(\mu), \alpha^\vee} > 0, \quad \alpha \in \Sigma^{\text{red}}_{Q_0} \cap (-\Sigma^{\text{red}}_Q).$$
  Notons $\mathcal{U}$ le produit direct des sous-groupes radiciels $\{U_\alpha : \alpha \in \Sigma_{Q_0} \cap (-\Sigma_Q)\}$. D'après l'identification explicite de l'induction par étapes, $M_{R|P}(\mu)$ restreint à $\mathcal{I}_{\tilde{Q}_0}(\mathring{\pi})$ est égal à l'opérateur
  $$ \phi(\tilde{x}) \longmapsto \int_{\mathcal{U}(\A)} \phi(u\tilde{x}) e^{\angles{\mu, H_{Q_0}(ux) - H_Q(x)}} \dd u. $$
  Or cela s'identifie à $M_{Q|Q_0}(\mu)$. L'identification en $\mu=\lambda+\Lambda$ en résulte par prolongement méromorphe.

  Maintenant, en réfléchissant sur \eqref{eqn:coeff-spec-caracterisation}, on voit que la somme correspondant à $L$ dans $J_t(\mathring{f}^1)$ vaut
  $$ \sum_{\mathring{\pi} \in \Pi_{\mathrm{disc},t,-}(\tilde{L}^1)} \; \int_{i(\mathfrak{a}^G_L)^*} a^{\tilde{L}}_{\mathrm{disc}}(\mathring{\pi}_\lambda) J_{\tilde{L}}(\mathring{\pi}_\lambda, \mathring{f}^1) \dd\lambda, $$
  ce qui achève la démonstration.
\end{proof}

\begin{remarque}\label{rem:convergence-spec}
  Le Théorème \ref{prop:dev-spec-fin} donne une expression pour $J(\mathring{f}^1) = \sum_{t \geq 0} J_t(\mathring{f}^1)$ comme une intégrale itérée. C'est déjà observé dans \cite{Li13} que, vu les arguments de \cite{FLM11}, cette intégrale est probablement absolument convergente, donc les indices $t$ sont probablement évitables. Cependant nous ne poursuivons pas cette question dans cet article. 

  Vu cette simplification potentielle, on ne parlera pas de la majoration de convergence par multiplicateurs d'Arthur \cite[pp.199-200]{Ar02} et \cite[\S 6]{Ar88-TF2}. Le lecteur peut reproduire ces majorations en reprenant ses arguments.
\end{remarque}

\subsection{Paramètres spectraux non ramifiés}
On fixe $t \geq 0$, un ensemble fini de places $V \supset V_\text{ram}$ de $F$, et on définit
\begin{align*}
  \mathcal{C}_-(\tilde{G}^V) & := \{c \in \Pi_-(\tilde{G}^V) : c \text{ est non ramifiée } \}, \\
  \Pi_{\text{disc},t,-}(\tilde{G}^1, V) & := \{\pi \in \Pi_-(\tilde{G}^1_V) : \exists \mathring{\pi} \in \Pi_{\text{disc},t,-}(\tilde{G}^1) \text{ tel que } \mathring{\pi}_V = \pi \}, \\
  \mathcal{C}^V_{\text{disc},-}(\tilde{G}) & := \{ c \in \mathcal{C}_-(\tilde{G}^V) : \exists \mathring{\pi} \in \Pi_{\text{disc},t,-}(\tilde{G}^1), \lambda \in i\mathfrak{a}_G^* \text{ tels que } (\mathring{\pi}_\lambda)^V = c \}.
\end{align*}
On prendra garde que les définitions de $\Pi_{\text{disc},t,-}(\tilde{G}^1, V)$ et $\mathcal{C}^V_{\text{disc},-}(\tilde{G})$ ne sont pas symétriques: $\mathcal{C}_-(\tilde{G}^V)$ et $\mathcal{C}^V_{\text{disc},-}(\tilde{G})$ admettent une $i\mathfrak{a}_G^*$-action $c \mapsto c_\lambda$, mais $\Pi_{\text{disc},t,-}(\tilde{G}^1, V)$ ne l'admet pas. De telles définitions s'appliquent également aux sous-groupes de Lévi de $\tilde{G}$.

Soient $M \in \mathcal{L}(M_0)$, $c \in \mathcal{C}_-(\tilde{M}^V)$ et $\lambda \in \mathfrak{a}_{M,\C}^*$. On peut définir $c_\lambda \in \mathcal{C}_-(\tilde{M}^V)$. Rappelons que dans \cite[\S 3.4]{Li12b}, on a construit les facteurs normalisants faibles non ramifiés en une place $v \notin V$
$$ r_{Q|P}(c_{v,\lambda}), \quad P,Q \in \mathcal{P}(M). $$

\begin{lemme}
  Soit $c \in \mathcal{C}^V_{\mathrm{disc},-}(\tilde{M})$. Pour tous $P, Q \in \mathcal{P}(M)$, le produit infini
  $$ r_{Q|P}(c_\lambda) := \prod_{v \notin V} r_{P|Q}(c_{v,\lambda}) $$
  est absolument convergent si $\angles{\Re(\lambda), \alpha^\vee} \gg 0$ pour tout $\alpha \in \Sigma^{\mathrm{red}}_P$, et définit une fonction méromorphe en $\lambda$. De plus, $r_{Q|P}(c_\lambda)$ ne dépend que de la projection de $\lambda$ sur $(\mathfrak{a}^G_{M,\C})^*$.
\end{lemme}
\begin{proof}
  C'est loisible de supposer qu'il existe $\mathring{\sigma} \in \Pi_{\text{disc},t,-}(\tilde{M}^1)$ ayant la décomposition $\mathring{\sigma} = \sigma \otimes c$. On a démontré une assertion similaire dans \cite[Lemme 6.1]{Li13} pour le facteur
  \begin{gather}\label{eqn:r-adelique}
    r_{Q|P}(\mathring{\sigma}_\lambda) := \prod_{v \in V} r_{Q|P}(\sigma_{v,\lambda}) \cdot r_{Q|P}(c_\lambda)
  \end{gather}
  en choisissant des facteurs normalisants $r_{Q|P}(\sigma_{v,\lambda})$ pour $v \in V$. Fixons $v \in V$. Selon \cite[Définition 3.1.1]{Li12b}, la fonction $\lambda \mapsto r_{Q|P}(\sigma_{v,\lambda})$ est méromorphe et $\mathfrak{a}^*_{G,\C}$-invariante. Les propriétés (\textbf{R4}) et (\textbf{R6}) dans op. cit. entraînent qu'elle est inversible pourvu que $\angles{\Re(\lambda), \alpha^\vee} \gg 0$ pour tout $\alpha \in \Sigma^{\mathrm{red}}_P$. L'assertion en résulte.
\end{proof}
\begin{remarque}
  Pour les groupes réductifs, les facteurs $r_{Q|P}(c)$ s'expriment en termes des fonctions $L$ partielles. On s'attend au même phénomène pour certains revêtements. Malgré l'importance de cette interprétation en pratique, on n'en a pas besoin dans cet article: il suffit d'extraire ses propriétés analytiques à l'aide des opérateurs d'entrelacement globaux.
\end{remarque}

Pour $c \in \mathcal{C}^V_{\text{disc},-}(\tilde{M})$, on introduit la $(G,M)$-famille
\begin{gather}
  r_Q(\Lambda, c) := r_{Q|\bar{Q}}(c)^{-1} r_{Q|\bar{Q}}\left(c_{\frac{\Lambda}{2}}\right), \quad Q \in \mathcal{P}(M), \Lambda \in i\mathfrak{a}^*_M .
\end{gather}
La théorie des $(G,M)$-familles fournit alors la fonction méromorphe en $\lambda \in \mathfrak{a}^*_{M,\C}$
$$ r^G_M(c_\lambda). $$

Afin de passer au cas adélique, considérons la situation suivante. On suppose que $\mathring{\sigma} = \sigma \otimes c$ pour un certain $\mathring{\sigma} \in \Pi_{\text{disc},t,-}(\tilde{M}^1)$ et $\sigma = \mathring{\sigma}_V$. En utilisant les facteurs normalisants adéliques \eqref{eqn:r-adelique}, on définit les avatars $r_Q(\Lambda, \sigma_\lambda)$ et $r_Q(\Lambda, \mathring{\sigma}_\lambda)$ de $r_Q(\Lambda, c_\lambda)$, qui vérifient
\begin{gather}\label{eqn:r-3familles}
  r_Q(\Lambda, \mathring{\sigma}_\lambda) = r_Q(\Lambda, \sigma_\lambda) r_Q(\Lambda, c_\lambda).
\end{gather}
On définit ainsi $r^G_M(\mathring{\sigma}_\lambda)$ et $r^G_M(\sigma_\lambda)$.

Fixons $P \in \mathcal{P}(M)$. Dans \cite{Li13} on considère une autre $(G,M)$-famille
$$ r_Q(\Lambda, \mathring{\sigma}_\lambda, P) := r_{Q|P}(\mathring{\sigma}_\lambda)^{-1} r_{Q|P}(\mathring{\sigma}_{\lambda+\Lambda}), \quad Q \in \mathcal{P}(M). $$

On introduit la $(G,M)$-famille $\nu_Q(\Lambda, \mathring{\sigma}_\lambda, P)$ de sorte que
\begin{gather}\label{eqn:r_Q-descente}
  r_Q(\Lambda, \mathring{\sigma}_\lambda) = \nu_Q(\Lambda, \mathring{\sigma}_\lambda, P) r_Q(\Lambda, \mathring{\sigma}_\lambda, P).
\end{gather}

Soient $L \in \mathcal{L}(M)$, $Q_L \in \mathcal{P}(L)$. Toutes les $(G,M)$-familles en vue sont des familles radicielles considérées dans \cite[\S 7]{Ar82-Eis2}. En particulier, \cite[Corollary 7.4]{Ar82-Eis2} affirme que $\nu^{Q_L}_M(\cdots)$, etc., ne dépendent pas du choix de $Q_L$. On les notera $\nu^L_M(\cdots)$, etc. On a
$$ \nu^L_M(\mathring{\sigma}_\lambda, P) := \lim_{\Lambda \to 0} \sum_{\substack{Q \in \mathcal{P}(M) \\ Q \subset Q_L}} \nu_Q(\Lambda, \mathring{\sigma}_\lambda, P) \theta_{Q \cap L}(\Lambda)^{-1}. $$
A priori, c'est défini pour $\lambda$ en position générale.

\begin{lemme}\label{prop:mu-nu}
  Avec les conventions précédentes, on a
  $$ \nu^L_M(\mathring{\sigma}_\lambda, P) = \begin{cases} 1, & \text{si } L=M, \\ 0, & \text{si } L \neq M, \end{cases} $$
  pour tout $L \in \mathcal{L}(M)$.
\end{lemme}
\begin{proof}
  On définit la $(G,M)$-famille
  \begin{gather}\label{eqn:mu-global}
    \mu_Q(\Lambda, \mathring{\sigma}_\lambda, P) := \mu_{Q|P}(\mathring{\sigma}_\lambda)^{-1} \mu_{Q|P}\left(\mathring{\sigma}_{\lambda+\frac{\Lambda}{2}}\right), \quad Q \in \mathcal{P}(M)
  \end{gather}
  avec
  \begin{gather}\label{eqn:mu_QP-global}
    \mu_{Q|P}(\mathring{\sigma}_\lambda) := \left( r_{Q|P}(\mathring{\sigma}_\lambda) r_{P|Q}(\mathring{\sigma}_\lambda) \right)^{-1}.
  \end{gather}
  Comme pour $\nu_Q(\Lambda, \mathring{\sigma}_\lambda, P)$, on en déduit $\mu^L_M(\mathring{\sigma}_\lambda, P)$ pour $\lambda$ en position générale. Montrons que
  \begin{gather}\label{eqn:mu-nu}
    \nu^L_M(\mathring{\sigma}_\lambda, P) = \mu^L_M(\mathring{\sigma}_\lambda, P).
  \end{gather}

  On adapte les arguments pour \cite[Lemma 2.1]{Ar98} au cadre adélique. On pose
  \begin{align*}
    c_Q(\Lambda, \mathring{\sigma}_\lambda, P) & := \left( r_{Q|P}(\mathring{\sigma}_\lambda)^{-1} r_{Q|P}(\mathring{\sigma}_{\lambda+\Lambda}) \right)^{-1} \left( r_{Q|P}(\mathring{\sigma}_\lambda) r_{Q|P} \left( \mathring{\sigma}_{\lambda+\frac{\Lambda}{2}} \right) \right)^2, \\
    \bar{\nu}_Q(\Lambda, \mathring{\sigma}_\lambda, P) & := r_{Q|\bar{Q}}(\mathring{\sigma}_\lambda)^{-1} r_{Q|\bar{Q}}\left( \mathring{\sigma}_{\lambda+\frac{\Lambda}{2}} \right) r_{Q|P}(\mathring{\sigma}_\lambda)^2 r_{Q|P} \left( \mathring{\sigma}_{\lambda+\frac{\Lambda}{2}} \right)^{-2} .
  \end{align*}
  On vérifie que $\nu_Q(\Lambda, \mathring{\sigma}_\lambda, P) = c_Q(\Lambda, \mathring{\sigma}_\lambda, P) \bar{\nu}_Q(\Lambda, \mathring{\sigma}_\lambda, P)$. La formule de descente \cite[Corollaire 4.2.5]{Li14a} dit que
  $$ \nu^L_M(\mathring{\sigma}_\lambda, P) = \sum_{L_1 \in \mathcal{L}^L(M)} c^{L_1}_M(\mathring{\sigma}_\lambda, P) \bar{\nu}^L_{L_1}(\mathring{\sigma}_\lambda, P). $$

  En appliquant \cite[Corollary 7.4]{Ar82-Eis2} à la $(G,M)$-famille radicielle $c_Q(\Lambda, \mathring{\sigma}_\lambda, P)$, on voit que $c^{L_1}_M(\mathring{\sigma}_\lambda, P)$ est un polynôme homogène de degré $\dim \mathfrak{a}^{L_1}_M$ en les dérivées $c'_\alpha(0)$, où on pose
  \begin{gather*}
    d_\alpha: \C \to \C \text{ telle que } d_\alpha(\angles{\Lambda, \alpha^\vee}) = r_\alpha(\mathring{\sigma}_\Lambda), \\
    c_\alpha(t) := (d_\alpha(0)^{-1} d_\alpha(t))^{-1} \left( d_\alpha(0) d_\alpha\left( \frac{t}{2}\right) \right)^2, \quad \alpha \in \Sigma^{\text{red}}_Q \cap \Sigma^{\text{red}}_{\bar{P}}.
  \end{gather*}
  Ici $r_\alpha(\mathring{\sigma}_\lambda) = \prod_v r_\alpha(\mathring{\sigma}_{v,\lambda})$ provient des facteurs normalisants locaux $r_\alpha$ satisfaisant à $r_{Q|P}(\cdots) = \prod_{\alpha \in \Sigma^{\text{red}}_Q \cap \Sigma^{\text{red}}_{\bar{P}}} \; r_\alpha(\cdots)$.

  On vérifie aisément que $c'_\alpha(0)=0$ pour tout $\alpha$ (le rôle du facteur $\frac{1}{2}$ y est assez visible). Donc $\nu^L_M(\mathring{\sigma}_\lambda, P) = \bar{\nu}^L_M(\mathring{\sigma}_\lambda, P)$.

  Rappelons que $r_{Q|\bar{Q}} = r_{Q|P} r_{P|\bar{Q}}$ et $r_{P|\bar{P}} = r_{P|Q} r_{Q|\bar{P}}$. Ces identités permettent d'écrire $\bar{\nu}_Q(\Lambda, \mathring{\sigma}_\lambda, P)$ comme le produit de 
  $$ r_{P|\bar{P}}(\mathring{\sigma}_\lambda)^{-1} r_{P|\bar{P}}\left( \mathring{\sigma}_{\lambda+\frac{\Lambda}{2}} \right) \; \text{ et } \; \mu_Q(\Lambda, \mathring{\sigma}_\lambda, P). $$

  Le premier terme est indépendant de $Q$ et tend vers $1$ quand $\Lambda \to 0$. Il en résulte que $\bar{\nu}^L_M(\mathring{\sigma}_\lambda, P) = \mu^L_M(\mathring{\sigma}_\lambda, P)$, ce qui démontre \eqref{eqn:mu-nu}.

  Enfin, montrons que \eqref{eqn:mu-nu} entraîne l'assertion du Lemme. Pour cela, nous avons besoins des résultats de \cite[\S 6]{Li13}. C'est loisible de supposer que $\lambda$ est en position générale. Notons $R_{P|Q}(\cdots)$ l'opérateur d'entrelacement normalisé global et multiplions les deux côtés de \eqref{eqn:mu_QP-global} par $(R_{Q|P}(\mathring{\sigma}_\lambda) R_{P|Q}(\mathring{\sigma}_\lambda))^{-1}$, qui vaut l'identité. Le côté à droite devient $(M_{Q|P}(\mathring{\sigma}_\lambda) M_{P|Q}(\mathring{\sigma}_\lambda))^{-1}$, ce qui vaut aussi l'identité d'après l'équation fonctionnelle des opérateurs d'entrelacement globaux. Il en résulte que $\mu_{Q|P}(\mathring{\sigma}_\lambda)=1$ pour tout $\lambda$, ce qui permet de conclure.
\end{proof}

\begin{lemme}[Cf. {\cite[Lemma 3.2]{Ar02}}]\label{prop:majoration-r^G_M}
  Pour $M,c$ comme ci-dessus, $\lambda \mapsto r^G_M(c_\lambda)$ est analytique sur $i\mathfrak{a}_M^*$ et vérifie
  $$ \exists N \; \text{ tel que } \int_{i(\mathfrak{a}^G_M)^*} r^G_M(c_\lambda)(1+\|\lambda\|)^{-N} \dd\lambda < +\infty ,$$
  où $\|\cdot\|$ est une norme euclidienne sur $i(\mathfrak{a}^G_M)^*$.
\end{lemme}
\begin{proof}
  On suppose toujours que $\mathring{\sigma} = \sigma \otimes c$ pour un certain $\mathring{\sigma} \in \Pi_{\text{disc},t,-}(\tilde{M}^1)$. Vu les identités \eqref{eqn:r-3familles}, \eqref{eqn:r_Q-descente} et la formule de descente, le Lemme \ref{prop:mu-nu} donne
  $$ \sum_{L_1, L_2 \in \mathcal{L}(M)} d^G_M(L_1,L_2) r^{L_1}_M(\sigma_\lambda) r^{L_2}_M(c_\lambda) = r^G_M(\mathring{\sigma}_\lambda) = r^G_M(\mathring{\sigma}_\lambda, P). $$

  Le seul terme à gauche avec $L_2 = G$ est $r^G_M(c_\lambda)$. On a démontré dans \cite[Théorème 6.3]{Li13} que $r^G_M(\mathring{\sigma}_\lambda, P)$ vérifie les propriétés du Lemme. Par récurrence, il en résulte que $r^G_M(c_\lambda)$ est analytique en $\lambda \in i\mathfrak{a}^*_M$. Il reste à établir la majoration.

  Prenons $N=N_1+N_2$. À l'aide de l'isomorphisme
  $$ \frac{i\mathfrak{a}^*_M}{i\mathfrak{a}^*_G} \stackrel{\sim}{\twoheadrightarrow} \frac{i\mathfrak{a}^*_M}{i\mathfrak{a}^*_{L_1}} \oplus \frac{i\mathfrak{a}^*_M}{i\mathfrak{a}^*_{L_2}}, \quad \text{ si } d^G_M(L_1, L_2) \neq 0, $$
  l'intégrale dans l'assertion est majorée par la somme de
  $$ \int_{i(\mathfrak{a}^G_M)^*} |r^G_M(\mathring{\sigma}_\lambda, P)| (1+\|\lambda\|)^{-N} \dd\lambda $$
  et
  $$ \sum_{L_1,L_2} d^G_M(L_1,L_2) \int_{i(\mathfrak{a}^{L_1}_M)^*} |r^{L_1}_M(\sigma_\lambda)| (1+\|\lambda\|)^{-N_1} \dd\lambda \cdot  \int_{i(\mathfrak{a}^{L_2}_M)^*} |r^{L_2}_M(c_\lambda)| (1+\|\lambda\|)^{-N_2} \dd\lambda, $$
  où $L_1, L_2 \in \mathcal{L}(M)$ sont tels que $L_2 \neq G$, $\mathfrak{a}^G_M = \mathfrak{a}^{L_1}_M \oplus \mathfrak{a}^{L_2}_M$, puisque sinon $d^G_M(L_1,L_2)=0$.

  On vient de remarquer que l'intégrale contenant $r^G_M(\mathring{\sigma}_\lambda, P)$ est finie. Les intégrales contenant $r^{L_2}_M(c_\lambda)$ sont aussi finies par récurrence. Donnons une esquisse pour la majoration de  $r^{L_1}_M(\sigma_\lambda)$. Cf. \cite[p.10]{Ar98}. On raisonne en plusieurs étapes
  \begin{enumerate}
    \item Rappelons que $r^{L_1}_M(\sigma_\lambda)$ se déduit de la $(G,M)$-famille
      $$ r_Q(\Lambda, \sigma_\lambda) = r_{Q|\bar{Q}}(\sigma_\lambda)^{-1} r_{Q|\bar{Q}}\left( \sigma_{\lambda+\frac{\Lambda}{2}} \right) = \prod_{\alpha \in \Sigma_Q^\text{red}} r_\alpha(\sigma_\lambda)^{-1} r_\alpha\left( \sigma_{\lambda+\frac{\Lambda}{2}} \right) $$
      où $Q \in \mathcal{P}(M)$, $\Lambda \in i(\mathfrak{a}^G_M)^*$. C'est une $(G,M)$-famille radicielle scalaire considérée dans \cite[\S 7]{Ar82-Eis2}. On renvoie à \cite[\S 3]{Li12b} pour les définitions précises des termes ci-dessus.
    \item La formule \cite[Corollary 7.4]{Ar82-Eis2} donne
      $$ r^{L_1}_M(\sigma_\lambda) = \sum_F \mes(\mathfrak{a}^{L_1}_M/\Z F^\vee_M) \prod_{\beta \in F} r_\beta(c_\lambda)^{-1} \dot{r}_\beta(\sigma_\lambda) $$
      où $F$ parcourt les sous-ensembles de $\Sigma^\text{red}_M$, l'ensemble de racines réduites restreintes à $\mathfrak{a}_M$, tels que $F^\vee_M := \{ \beta^\vee_M : \beta \in F\}$ est une base de $\mathfrak{a}^{L_1}_M$.
      Expliquons les notations:
      \begin{itemize}
        \item $\beta^\vee_M$ signifie la projection de $\beta^\vee$ sur $\mathfrak{a}_M$,
        \item $\Z F^\vee_M$ signifie le réseau engendré par $F^\vee_M$,
        \item pour $\sigma$ et $\lambda$ fixés, $r_\beta\left (\sigma_{\lambda+\frac{\Lambda}{2}} \right)$ ne dépend que de $\angles{\Lambda,\beta^\vee}$; on la regarde donc comme une fonction en $i\R$ puisque $\Lambda$ est supposé imaginaire, ce qui permet de parler de la dérivée $\dot{r}_\beta$.
      \end{itemize}
    \item Comme dans \cite[pp.1329-1330]{Ar82-Eis2}, ladite formule nous ramène au cas où $M$ est un Lévi propre maximal dans $L_1$. Notons $\pm\alpha$ les éléments dans $\Sigma^\text{red}_M$. Il y a exactement deux possibilités pour $F$: $F=\{\alpha\}$ ou $F=\{-\alpha\}$. On regarde $r_{\pm\alpha}$ comme des fonctions en $\lambda \in i\R$. Alors $r_{-\alpha}(\lambda) = \overline{r_\alpha(\lambda)}$ pour tout $\lambda \in i\R$ d'après \cite[Définition 3.3.1 \textbf{(R2)}]{Li12b}. Il reste donc à majorer $\Re(r_\alpha(\lambda)^{-1} \dot{r}_\alpha(\lambda))$.

    \item Comme $r^{L_1}_M(\sigma_{v,\lambda})$ est analytique pour $\lambda$ imaginaire en chaque $v \in V$ (c'est \cite[Corollary 2.4]{Ar98}), on sait que $r^{L_1}_M(\sigma_{v,\lambda})$ est bornée pour $v$ non-archimédienne. On se ramène ainsi au cas $V=V_\infty$.

    \item D'après la construction des facteurs normalisants archimédiens \cite[\S 3.2]{Li12b}, on sait que $r_\alpha(\lambda)$ est un produit des fonctions de la forme $\Gamma(\alpha\lambda + \beta)^{\pm 1}$ à une constante multiplicative près, où $\alpha \in \R$, $\beta \in \C$. On se ramène ainsi à la majoration de la dérivée logarithmique de la fonction $\Gamma$ dans un domaine de la forme
    $$ \mathcal{A} = \mathcal{A}(a_1,a_2,b) := \{ z=\sigma+it \in \C : a_1 < \sigma < a_2, \; |t| > b \}, \quad a_1, a_2 \in \R,\; b \in \R_{>0}. $$

    \item Prenons un autre domaine $\mathcal{A}^+ = \mathcal{A}(a'_1, a'_2, b') \supset \mathcal{A}$ avec $a'_1 < a_1$, $a_2 < a'_2$ et $0 < b' < b$. Posons
    $$ \delta(z) := \log\Gamma(z) -  \left( \left(z - \frac{1}{2}\right) \log z - z \right), \quad z \in \mathcal{A}^+. $$
    C'est une fonction holomorphe sur $\mathcal{A}^+$; il y a un choix de $\log$ sur $\mathcal{A}^+$, mais peu importe.

    La formule de Stirling fournit une constante $C$ telle que $|\delta(z)| \leq C$ pour $z \in \mathcal{A}^+$. D'après la formule intégrale de Cauchy, on en déduit une nouvelle constante $C'$ telle que $|\delta'(z)| \leq C'$ pour $z \in \mathcal{A}$. Puisque la dérivée de $(z - \frac{1}{2}) \log z - z$ est évidemment à croissance modérée en la partie imaginaire de $z \in \mathcal{A}$, on en déduit la même propriété pour la dérivée logarithmique de $\Gamma$. La majoration cherchée pour $\lambda \mapsto \Re(r_\alpha(\lambda)^{-1} \dot{r}_\alpha(\lambda))$ ($\lambda \in i\R$) en résulte.
  \end{enumerate}
\end{proof}

\subsection{Compression des coefficients spectraux}
Soient $t \geq 0$, $V \supset V_\text{ram}$ comme précédemment. Soit $M \in \mathcal{L}(M_0)$. Rappelons que les éléments de $\Pi_-(\tilde{M}^1_V)$ sont regardés comme des représentations de $\tilde{M}_V$ sur lesquelles $A_{M,\infty}$ opère trivialement. On pose
$$ \Pi_{\text{disc},t,-}^{\tilde{G}^1}(\tilde{M}, V) := \{\sigma_\lambda : \sigma \in \Pi_{\text{disc},t,-}(\tilde{M}^1,V), \lambda \in i(\mathfrak{a}^G_M)^* \} \; \subset \Pi_{\text{unit},-}(\tilde{M}_V). $$

Étant donnés $P \in \mathcal{P}(M)$ et $\sigma \in \Pi_{\text{disc},t,-}^{\tilde{G}^1}(\tilde{M},V)$, l'induite parabolique normalisée
$$ \sigma^G := \mathcal{I}_{\tilde{P}}(\sigma) $$
est une représentation unitaire de longueur finie de $\tilde{G}_V$ sur laquelle $A_{G,\infty}$ opère trivialement. Pour $\pi \in \Pi_-(\tilde{G}_V^1)$, on pose
$$ (\sigma^G : \pi) := \text{ la multiplicité de } \pi \text{ dans } \sigma^G . $$

On définit
\begin{gather}
  \Pi_{t,-}(\tilde{G}^1, V) := \left\{ \pi \in \Pi_{\text{unit},-}(\tilde{G}_V^1) : \exists M \in \mathcal{L}(M_0), \sigma \in  \Pi_{\text{disc},t,-}^{\tilde{G}^1}(\tilde{M}, V) \text{ tels que } (\sigma^G : \pi) \neq 0 \right\}.
\end{gather}

Adoptons la convention suivante: une fonction $h: \Pi_{t,-}(\tilde{G}^1, V) \to \C$ se prolonge par zéro à $\Pi_{\text{unit}, -}(\tilde{G}_V)$, puis par linéarité aux représentations unitaires spécifiques de longueur finie de $\tilde{G}_V$.

\begin{definition-fr}\label{def:coeff-spec}
  On définit une mesure sur $\Pi_{t,-}(\tilde{G}^1,V)$, notée abusivement par $\pi \mapsto a^{\tilde{G}}(\pi)\dd\pi$ (à expliquer dans la remarque suivante), telle que
  \begin{multline}\label{eqn:mesure-Pi}
    \int_{\Pi_{t,-}(\tilde{G}^1,V)} h(\pi) a^{\tilde{G}}(\pi) \dd\pi = \\
    \sum_{M \in \mathcal{L}(M_0)} \frac{|W^M_0|}{|W^G_0|} \sum_{\substack{\sigma \in \Pi_{\text{disc},t,-}(\tilde{M}^1, V) \\ c \in \mathcal{C}^V_{\text{disc},t,-}(\tilde{M})}} \; \int_{i(\mathfrak{a}^G_M)^*} a^{\tilde{M}}_{\text{disc}}(\sigma_\lambda \otimes c_\lambda) r^G_M(c_\lambda) h(\sigma^L_\lambda) \dd\lambda,
  \end{multline}
  si $h: \Pi_{t,-}(\tilde{G}^1,V) \to \C$ induit une fonction
  $$ h' : \bigsqcup_{M \in \mathcal{L}(M_0)} \Pi_{\text{unit},-}(\tilde{M}_V) \to \C $$
  par $h'(\sigma)=h(\sigma^G)$, qui vérifie
  \begin{itemize}
    \item pour tout $M$, il existe un ensemble fini $\Gamma$ de $\tilde{K} \cap \tilde{M}_V$-types tel que $h'(\sigma)=0$ si $\sigma$ ne contient pas de $\tilde{K} \cap \tilde{M}_V$-types dans $\Gamma$;
    \item pour tout $M$, la fonction $\lambda \mapsto h'(\sigma_\lambda)$ est à décroissance rapide sur $i(\mathfrak{a}^G_M)^*$.
  \end{itemize}
 
  Pour de telles fonctions $h$, la convergence de l'intégrale ci-dessus en $(\sigma,c,\lambda)$ résulte du Lemme \ref{prop:majoration-r^G_M} et de \cite[Proposition 7.4]{Li13}.
\end{definition-fr}
\begin{remarque}\label{rem:coeff-spec}
  Cette notation suggère que l'on puisse bien définir le symbole $\pi \mapsto a^G(\pi)$ (le ``coefficient spectral'') comme une certaine dérivée de Radon-Nikodym de notre mesure. C'est ce qu'Arthur fait dans \cite[\S 3]{Ar02}, pour l'essentiel. Nous n'abordons pas ce point de vue. Cette notation est adoptée seulement en raison de compatibilité. Pour la même raison, on écrira aussi $\int a^{\tilde{G}}(\pi) h(\pi) \dd\pi$ au lieu de $\int h(\pi) a^{\tilde{G}}(\pi) \dd\pi$.
\end{remarque}

\begin{proposition-fr}\label{prop:caractere-pondere-local-global}
  Soit $f^1 \in \mathcal{H}_{\asp}(\tilde{G}_V, A_{G,\infty})$, on pose $\mathring{f}^1 := f^1 f_{K^V} \in \mathcal{H}_{\asp}(\tilde{G}, A_{G,\infty})$. Pour tous $M \in \mathcal{L}(M_0)$, $\lambda \in i(\mathfrak{a}^G_M)^*$ et
  $$ \mathring{\sigma} = \sigma \otimes c, \quad \sigma \in \Pi_{\mathrm{disc},t,-}(\tilde{M}^1, V), c \in \mathcal{C}^V_{\mathrm{disc},-}(\tilde{M}), $$
  on a
  $$ J_{\tilde{M}}(\mathring{\sigma}_\lambda, \mathring{f}^1) = \sum_{L \in \mathcal{L}(M)} r^L_M(c_\lambda) J_{\tilde{L}}(\sigma_\lambda^L, f^1), $$
  où $J_{\tilde{L}}(\sigma_\lambda^L, \cdot)$ est le caractère pondéré défini par rapport à la donnée centrale $(A_{G,\infty},\mathfrak{a}_G)$ de $\tilde{G}_V$.
\end{proposition-fr}
\begin{proof}
  Reproduisons les arguments dans \cite[pp.207-208]{Ar02}. Bien évidemment, il faut comparer les $(G,M)$-familles définissant les caractères pondérés locaux et globaux. Fixons $P \in \mathcal{P}(M)$. Au vu de la définition de $\mathring{f}^1$, il suffit de regarder leurs actions sur l'espace
  $$ \mathcal{I}_{\tilde{P}}(\sigma_\lambda) \otimes \bigotimes_{v \notin V} (\text{ le vecteur sphérique en } v ). $$

  On définit une $(G.M)$-famille
  $$ r_Q(\Lambda, c_\lambda, \tilde{P}) := r_{Q|P}(c_\lambda)^{-1} r_{Q|P}(c_{\lambda+\Lambda}), \quad Q \in \mathcal{P}(M), \Lambda \in i(\mathfrak{a}^G_M)^*. $$

  D'après la construction des facteurs normalisants non ramifiés, l'opérateur $\mathcal{J}_Q(\Lambda, \mathring{\sigma}_\lambda, \tilde{P})$ défini dans \eqref{eqn:J_Q-global} agit sur le produit tensoriel des vecteurs sphériques en $v \notin V$ par le scalaire $r_Q(\Lambda, c_\lambda, \tilde{P})$. Il en résulte que
  $$ \mathcal{J}_Q(\Lambda, \mathring{\sigma}_\lambda, \tilde{P}) = r_Q(\Lambda, c_\lambda, \tilde{P}) \mu_Q(\Lambda, \sigma_\lambda, \tilde{P})^{-1} \mathcal{M}_Q(\Lambda, \sigma_\lambda, \tilde{P}), $$
  soit encore
  $$ \left( r_Q(\Lambda, c_\lambda, \tilde{P}) \mu_Q(\Lambda, c_\lambda, \tilde{P}) \right) \mathcal{M}_Q(\Lambda, \sigma_\lambda, \tilde{P}) $$
  d'après l'équation fonctionnelle des opérateurs d'entrelacement globaux, où
  $$ \mu_Q(\Lambda, c_\lambda, \tilde{P}) := \prod_{v \notin V}\mu_{Q|P}(c_{v,\lambda})^{-1} \mu_{Q|P}\left( c_{v,\lambda+\frac{\Lambda}{2}} \right). $$
  On a déjà observé la convergence absolue de ce produit infini dans \eqref{eqn:mu-global}.

  Notons
  $$ \rho_Q(\Lambda, c_\lambda, \tilde{P}) := r_Q(\Lambda, c_\lambda, \tilde{P}) \mu_Q(\Lambda, c_\lambda, \tilde{P}),$$
  qui est encore une $(G,M)$-famille. Ici encore, les $(G,M)$-familles scalaires en vue sont toutes radicielles au sens de \cite[\S 7]{Ar82-Eis2}. Soit $L \in \mathcal{L}(M)$, on choisit $Q_L \in \mathcal{P}(L)$ pour définir les $(L,M)$-familles $r^{Q_L}_R(\cdots)$, etc. D'après \cite[Corollary 7.4]{Ar82-Eis2}, $r^{Q_L}_M(\cdots)$ ne dépend pas du choix de $Q_L$ donc c'est loisible de les noter $r^L_M(\cdots)$, etc.

  Appliquons maintenant la formule de descente \cite[Lemme 4.2.7]{Li14a} à la $(L,M)$-famille déduite de $\rho_Q$, ce qui donne
  $$ \rho^L_M(c_\lambda, \tilde{P}) = \sum_{L_1, L_2 \in \mathcal{L}^L(M)} d^L_M(L_1,L_2) r^{L_1}_M(c_\lambda,\tilde{P}) \mu^{L_2}_M(c_\lambda, \tilde{P}). $$

  D'autre part, comme dans la démonstration du Lemme \ref{prop:mu-nu} on définit une $(G.M)$-famille radicielle  $\nu_Q(\Lambda, c_\lambda, \tilde{P})$ vérifiant
  $$ r_Q(\Lambda, c_\lambda) = \nu_Q(\Lambda, c_\lambda, \tilde{P}) r_Q(\Lambda, c_\lambda, \tilde{P}). $$
  Alors la formule de descente donne
  $$ r^L_M(c_\lambda) = \sum_{L_1, L_2 \in \mathcal{L}^L(M)} d^L_M(L_1,L_2) r_M^{L_1}(c_\lambda, \tilde{P}) \nu^{L_2}_M(c_\lambda, \tilde{P}). $$

  Une variante du Lemme \ref{prop:mu-nu} dit que $\nu^{L_2}_M(c_\lambda,\tilde{P})=1$ si $L_2=M$, sinon il vaut zéro. D'où $\rho^L_M(c_\lambda, \tilde{P}) = r^L_M(c_\lambda)$. En effet, la démonstration s'y adapte sans modifications. Appliquons la version \cite[Corollaire 4.2.5]{Li14a} de la formule de descente au produit des $(G,M)$-familles $\rho_Q(\Lambda, c_\lambda, \tilde{P}) \mathcal{M}_Q(\Lambda, \sigma_\lambda, \tilde{P})$, alors des arguments standards (cf. la preuve de \cite[(7.8)]{Ar81}) entraînent
  \begin{align*}
    J_{\tilde{M}}(\mathring{\sigma}_\lambda, \mathring{f}^1) &= \sum_{L \in \mathcal{L}(M)} \rho^L_M(c_\lambda, \tilde{P}) J_{\tilde{L}}(\sigma_\lambda^L, f^1) \\
    & = \sum_{L \in \mathcal{L}(M)} r^L_M(c_\lambda) J_{\tilde{L}}(\sigma_\lambda^L, f^1).
  \end{align*}

  Cela permet de conclure.
\end{proof}

\begin{theoreme}\label{prop:dev-spec-nouveau}
  Soit $t \geq 0$. Pour tout $f^1 \in \mathcal{H}_{\asp}(\tilde{G}_V, A_{G,\infty})$, on a
  $$ J_t(f^1) = \sum_{L \in \mathcal{L}(M_0)} \frac{|W^L_0|}{|W^G_0|} \int_{\Pi_{t,-}(\tilde{L}^1, V)} a^{\tilde{L}}(\pi) J_{\tilde{L}}(\pi, 0, f^1) \dd\pi, $$
  où $J_{\tilde{L}}(\pi, 0, \cdot)$ est le coefficient de Fourier en $X=0$ du caractère pondéré de $\pi$ défini par rapport à la donnée centrale $(A_{G,\infty},\mathfrak{a}_G)$ de $\tilde{G}_V$.
\end{theoreme}
\begin{proof}
  Notons tout d'abord que l'application $h: \pi \mapsto J_{\tilde{L}}(\pi,0,f^1)$ vérifie les conditions de la Définition \ref{def:coeff-spec} avec $L$ au lieu de $G$. La condition concernant le support de $h$ est évidente. Quant à la décroissance rapide de $\lambda \mapsto J_{\tilde{L}}(\sigma^L_\lambda,0,f^1)$, c'est contenu dans le sens facile du Théorème \ref{prop:phi-M} (voir aussi la Définition \ref{def:PW-equiv} de l'espace de Paley-Wiener), ce qui ne dépend pas de l'Hypothèse \ref{hyp:PW}. Ceci justifie l'intégrale dans l'assertion.

  Rappelons le Théorème \ref{prop:dev-spec-fin} qui fournit le développement pour $J_t(f^1)$
  $$ \sum_{M \in \mathcal{L}(M_0)} \frac{|W^M_0|}{|W^G_0|} \sum_{\mathring{\sigma} \in \Pi_{\mathrm{disc},t,-}(\tilde{M}^1)} \; \int_{i(\mathfrak{a}^G_M)^*} a^{\tilde{M}}_{\mathrm{disc}}(\mathring{\sigma}_\lambda) J_{\tilde{M}}(\mathring{\sigma}_\lambda, \mathring{f}^1) \dd\lambda $$
  où $\mathring{f}^1 := f f_{K^V}$. À cause de la présence de $f_{K^V}$, on peut se limiter aux représentations $\mathring{\sigma}$ de la forme
  $$ \mathring{\sigma} = \sigma \otimes c, \quad \sigma \in \Pi_{\text{disc},t,-}(\tilde{M}^1,V), c \in \mathcal{C}^V_{\text{disc},-}(\tilde{M}). $$

  Appliquons la Proposition \ref{prop:caractere-pondere-local-global} aux $J_{\tilde{M}}(\mathring{\sigma}_\lambda, \mathring{f}^1)$. On obtient donc
  \begin{multline*}
    \sum_{L \in \mathcal{L}(M_0)} \frac{|W^L_0|}{|W^G_0|} \sum_{M \in \mathcal{L}^L(M_0)} \frac{|W^M_0|}{|W^L_0|} \\
    \sum_{\sigma \in \Pi_{\text{disc},t,-}(\tilde{M}^1,V)} \; \sum_{c \in \mathcal{C}^V_{\text{disc},-}(\tilde{M})} \; \int_{i(\mathfrak{a}^G_M)^*} a^{\tilde{M}}_{\text{disc}}(\sigma_\lambda \otimes c_\lambda) r^L_M(c_\lambda) J_{\tilde{L}}(\sigma_\lambda^L, f^1) \dd\lambda .
  \end{multline*}

  La somme portant sur $c$ est finie pour $\sigma$ fixé. On observe que $a^{\tilde{M}}_\text{disc}(\sigma_\lambda \otimes c_\lambda)$ et $r^L_M(c_\lambda)$ sont invariants par la translation $\lambda \mapsto \lambda+\Lambda$ où $\Lambda \in i(\mathfrak{a}^G_L)^*$, cf. \eqref{eqn:coeff-spec-translation}. L'intégrale sur $i(\mathfrak{a}^G_M)^*$ se décompose en celle sur $(\lambda,\Lambda) \in i(\mathfrak{a}^L_M)^* \oplus i(\mathfrak{a}^G_L)^*$. On arrive à
  \begin{multline*}
    \sum_{L \in \mathcal{L}(M_0)} \frac{|W^L_0|}{|W^G_0|} \sum_{M \in \mathcal{L}^L(M_0)} \frac{|W^M_0|}{|W^L_0|} \sum_{\sigma \in \Pi_{\text{disc},t,-}(\tilde{M}^1,V)} \\
    \int_{\lambda \in i(\mathfrak{a}^L_M)^*} \left( \sum_{c \in \mathcal{C}^V_{\text{disc},-}(\tilde{M})} a^{\tilde{M}}_{\text{disc}}(\sigma_\lambda \otimes c_\lambda) r^L_M(c_\lambda) \right) \int_{\Lambda \in i(\mathfrak{a}^G_L)^*} J_{\tilde{L}}(\sigma^L_{\lambda+\Lambda}, f^1) \dd\Lambda \dd\lambda .
  \end{multline*}

  L'intégrale en $\Lambda$ donne $J_{\tilde{L}}(\pi,0,f^1)$. Pour conclure, il reste à rappeler la Définition \ref{def:coeff-spec} de la mesure sur $\Pi_{t,-}(\tilde{L}^1,V)$.
\end{proof}

\section{La formule des traces invariante}\label{sec:I}
Dans cette section, on conserve les mêmes conventions que dans \S\ref{sec:developpements} sur $\tilde{G}$, $\tilde{K}$, $M_0$, $K$ et les mesures. De tels choix permettent de bien définir la distribution dans la formule des traces grossière pour $\tilde{G}$, notée $J: \mathcal{H}_{\asp}(\tilde{G}, A_{G,\infty}) \to \C$. On fixe un ensemble fini de places $V \supset V_\text{ram}$ et on suppose vérifiée l'Hypothèse \ref{hyp:PW}, ce qui permet d'utiliser les résultats dans \S\ref{sec:dists-locales}.

Comme expliqué dans \S\ref{sec:intro}, nous n'utilisons pas les données $(Z,\zeta)$ de \cite[\S 1]{Ar02}. Une variante de la formule des traces à la \cite{Ar88-TF2} sera donnée dans \S\ref{sec:id-supp}.

\subsection{Distributions invariantes globales}
Rappelons que la distribution $J$ admet une décomposition $J = \sum_{t \geq 0} J_t$ selon la norme de caractères infinitésimaux des données automorphes cuspidales.

Rappelons aussi que l'espace $\mathcal{H}_{\asp}(\tilde{G}_V, A_{G,\infty})$ se plonge dans $\mathcal{H}_{\asp}(\tilde{G}, A_{G,\infty})$ par
$$ f^1 \longmapsto \mathring{f}^1 := f^1 f_{K^V} $$
où $f_{K^V}$ est l'unité de l'algèbre de Hecke sphérique anti-spécifique hors de $V$. On en déduit une distribution $J: \mathcal{H}_{\asp}(\tilde{G}_V, A_{G,\infty}) \to \C$, donnée par $J(f^1) := J(\mathring{f}^1)$.

L'énoncé suivant donne la formule des traces invariante que l'on cherche.

\begin{theoreme}[Cf. {\cite[Proposition 2.2 et Proposition 3.3]{Ar02}}]\label{prop:I}
  Pour tous $L \in \mathcal{L}(M_0)$ et $t \geq 0$, il existe des distributions
  $$ I^{\tilde{L}}, I^{\tilde{L}}_t: \mathcal{H}_{\asp}(\tilde{L}_V, A_{L,\infty}) \to \C $$
  qui sont invariantes et supportées par $I\mathcal{H}_{\asp}(\tilde{L}_V, A_{L,\infty})$, telles que si l'on note $I := I^{\tilde{G}}$, $I_t := \tilde{I}^{\tilde{G}}_t$, alors
  \begin{align}
    \label{eqn:I}
    I(f^1) & = J(f^1) - \sum_{\substack{L \in \mathcal{L}^G(M_0) \\ L \neq G}} \frac{|W^L_0|}{|W^G_0|} I^{\tilde{L}}(\phi^1_{\tilde{L}}(f^1)), \\
    \label{eqn:I_t}
    I_t(f^1) & = J_t(f^1) - \sum_{\substack{L \in \mathcal{L}^G(M_0) \\ L \neq G}} \frac{|W^L_0|}{|W^G_0|} I^{\tilde{L}}_t(\phi^1_{\tilde{L}}(f^1)), \\
    \label{eqn:I-I_t}
    I(f^1) & = \sum_{t \geq 0} I_t(f^1), \\
    \label{eqn:I-geom}
    I(f^1) & = \sum_{M \in \mathcal{L}^G(M_0)} \frac{|W^M_0|}{|W^G_0|} \sum_{\tilde{\gamma} \in \Gamma(\tilde{M}^1, V)} a^{\tilde{M}}(\tilde{\gamma}) I_{\tilde{M}}(\tilde{\gamma}, f^1), \\
    \label{eqn:I-spec}
    I_t(f^1) & = \sum_{M \in \mathcal{L}^G(M_0)} \frac{|W^M_0|}{|W^G_0|} \int_{\Pi_{t,-}(\tilde{M}^1,V)} a^{\tilde{M}}(\pi) I_{\tilde{M}}(\pi,0,f^1) \dd\pi
  \end{align}
  pour tout $f^1 \in \mathcal{H}_{\asp}(\tilde{G}_V, A_{G,\infty})$, où $I_{\tilde{M}}(\tilde{\gamma},\cdot)$ et $I_{\tilde{M}}(\pi,\cdot)$ sont les distributions invariantes introduites dans \S\ref{sec:dists-locales} définies par rapport à la donnée centrale $(A_{G,\infty}, \mathfrak{a}_G)$ de $\tilde{G}_V$. Toutes les intégrales sont absolument convergentes.

  Les mêmes identités sont satisfaites si l'on remplace $G$ par un élément quelconque de $\mathcal{L}^G(M_0)$.
\end{theoreme}
Évidemment, lesdites identités déterminent les distributions $I^{\tilde{L}}$ et $I^{\tilde{L}}_t$. La définition de $\phi^1_{\tilde{L}}$ se trouve dans \S\ref{sec:phi^1}.

\begin{proof}
  Lorsque $G$ est anisotrope modulo le centre, on prend tout simplement $I := J$ et $I_t := J_t$. Raisonnons par récurrence et regardons les identités \eqref{eqn:I} et \eqref{eqn:I_t} comme les définitions de $I$ et $I_t$. Il faut vérifier que les autres identités sont satisfaites et que $I$, $I_t$ sont invariantes et supportées par $I\mathcal{H}_{\asp}(\tilde{G}_V, A_{G,\infty})$.

  D'après l'hypothèse de récurrence et le fait $J = \sum_t J_t$, \eqref{eqn:I-I_t} est satisfaite.

  Montrons \eqref{eqn:I-geom}. On a
  \begin{multline*}
    I(f^1) := J(f^1) - \sum_{\substack{L \in \mathcal{L}(M_0) \\ L \neq G}} \frac{|W^L_0|}{|W^G_0|} I^L(\phi^1_{\tilde{L}}(f^1)) = \sum_{M \in \mathcal{L}(M_0)} \frac{|W^M_0|}{|W^L_0|} \sum_{\tilde{\gamma} \in \Gamma(\tilde{M}^1, V)} a^{\tilde{M}}(\tilde{\gamma}) J_{\tilde{M}}(\tilde{\gamma}, f^1) \\
    - \sum_{\substack{L \in \mathcal{L}(M_0) \\ L \neq G}} \frac{|W^L_0|}{|W^G_0|} \sum_{M \in \mathcal{L}^L(M_0)} \frac{|W^M_0|}{|W^L_0|} \sum_{\tilde{\gamma} \in \Gamma(\tilde{M}^1, V)} a^{\tilde{M}}(\tilde{\gamma}) I^{\tilde{L}}_{\tilde{M}}(\tilde{\gamma}, \phi^1_{\tilde{L}}(f^1)) \\
    = \sum_{M \in \mathcal{L}(M_0)} \frac{|W^M_0|}{|W^G_0|} \sum_{\tilde{\gamma} \in \Gamma(\tilde{M}^1, V)} a^{\tilde{M}}(\tilde{\gamma}) \left( J_{\tilde{M}}(\tilde{\gamma}, f^1) - \sum_{\substack{L \in \mathcal{L}(M) \\ L \neq G}} I^{\tilde{L}}_{\tilde{M}}(\tilde{\gamma}, \phi^1_{\tilde{L}}(f^1)) \right) \\
    = \sum_{M \in \mathcal{L}(M_0)} \frac{|W^M_0|}{|W^G_0|} \sum_{\tilde{\gamma} \in \Gamma(\tilde{M}^1, V)} a^{\tilde{M}}(\tilde{\gamma}) I_{\tilde{M}}(\tilde{\gamma}, f^1),
  \end{multline*}
  où on a utilisé le Théorème \ref{prop:dev-geom-nouveau}, l'équation \eqref{eqn:I-geom} pour $L$ et la Proposition \ref{prop:I-phi^1}. La convergence absolue provient du Théorème \ref{prop:dev-geom-nouveau}.

  On établit \eqref{eqn:I-spec} de la même façon, en utilisant le Théorème \ref{prop:dev-spec-nouveau} et l'équation \eqref{eqn:I-spec} pour le côté spectral.

  Vu les développements \eqref{eqn:I-geom} et \eqref{eqn:I-spec}, les distributions $I$, $I_t$ sont toutes invariantes et supportées par $I\mathcal{H}_{\text{ac}}(\tilde{G}_V, A_{G,\infty})$ car $I_{\tilde{M}}(\tilde{\gamma},\cdot)$ et $I_{\tilde{M}}(\pi,0,\cdot)$ le sont d'après les résultats de \S\ref{sec:dist-locale-preuve}. Cela achève la démonstration.
\end{proof}

\subsection{Identités supplémentaires}\label{sec:id-supp}
Dans cette sous-section, on donnera des variantes des développements \eqref{eqn:I-geom} et \eqref{eqn:I-spec}. Ils réconcilient des énoncés différents de la formule des traces invariante.

\begin{theoreme}\label{prop:I-1}
  Fixons $(A,\mathfrak{a})$ une donnée centrale de $\tilde{G}_V$. Soient $f^A \in \mathcal{H}_{\asp}(\tilde{G}_V,A)$ et $f^1 \in \mathcal{H}_{\asp}(\tilde{G}_V, A_{G,\infty})$ la fonction déduite de $f^A$ par la Proposition \ref{prop:f-moyenne}. On regarde $I$ et $I_t$ comme des formes linéaires sur $\mathcal{H}_{\asp}(\tilde{G}_V, A)$ en posant $I(f^A) = I(f^1)$, $I_t(f^A)=I_t(f^1)$, pour tout $t \geq 0$. Alors on a
  \begin{align*}
    I(f^A) & = \sum_{t \geq 0} I_t(f^A), \\
    I(f^A) & = \sum_{M \in \mathcal{L}^G(M_0)} \frac{|W^M_0|}{|W^G_0|} \sum_{\tilde{\gamma} \in \Gamma(\tilde{M}^1, V)} a^{\tilde{M}}(\tilde{\gamma}) \int_{z \in A_{G,\infty}/A} I_{\tilde{M}}(z\tilde{\gamma}, f^A) \dd z, \\
    I_t(f^A) & = \sum_{M \in \mathcal{L}^G(M_0)} \frac{|W^M_0|}{|W^G_0|} \int_{\Pi_{t,-}(\tilde{M}^1,V)} a^{\tilde{M}}(\pi) \int_{\mathfrak{a}_G/\mathfrak{a}}  I_{\tilde{M}}(\pi,X,f^A) \dd X \dd\pi ,
  \end{align*}
  où les distributions locales $I_{\tilde{M}}(\cdots)$ sont définies par rapport à $(A,\mathfrak{a})$.
\end{theoreme}
En pratique, on l'applique souvent au cas $A=\{1\}$ pour se ramener aux distributions locales définies sans référence à données centrales.
\begin{proof}
  On a toujours $A \subset A_{G,\infty}$. Il suffit donc d'utiliser le Théorème \ref{prop:I} et les propriétés données dans Hypothèses \ref{hyp:recurrence-spec} et \ref{hyp:recurrence-geom} (qui ne sont plus hypothétiques...)
\end{proof}

\begin{lemme}
  Soient $f^1 \in \mathcal{H}_{\asp}(\tilde{G}_V, A_{G,\infty})$ et $f^\dagger \in \mathcal{H}_{\mathrm{ac},\asp}(\tilde{G}_V)$ son image sous l'inclusion naturelle. Alors on a
  \begin{align*}
    I(f^1) & = \sum_{M \in \mathcal{L}(M_0)} \frac{|W^M_0|}{|W^G_0|} \sum_{\tilde{\gamma} \in \Gamma(\tilde{M}^1, V)} a^{\tilde{M}}(\tilde{\gamma}) I_{\tilde{M}}(\tilde{\gamma}, f^\dagger), \\
    I_t(f^1) & = \sum_{M \in \mathcal{L}(M_0)} \frac{|W^M_0|}{|W^G_0|} \int_{\Pi_{t,-}(\tilde{M}^1,V)} a^{\tilde{M}}(\pi) I_{\tilde{M}}(\pi,0,f^\dagger) \dd\pi,
  \end{align*}
  où les distributions $I_{\tilde{M}}(\cdots)$ sont définies par rapport à la donnée centrale triviale $(\{1\},\{0\})$, dont l'évaluation en $f^\dagger$ est loisible car elles sont toutes concentrées en $0$.
\end{lemme}
\begin{proof}
  Comme dans le cas précédant, cela résulte aussi des propriétés données dans Hypothèses \ref{hyp:recurrence-spec} et \ref{hyp:recurrence-geom}.
\end{proof}

Avant d'entamer le résultat suivant, rappelons que $\tilde{G}_V = \tilde{G}^1_V \times A_{G,\infty}$. Donc le produit tensoriel algébrique $\mathcal{H}_{\asp}(\tilde{G}^1_V) \otimes C^\infty_c(A_{G,\infty})$ se plonge dans de $\mathcal{H}_{\asp}(\tilde{G}_V)$.

\begin{theoreme}\label{prop:I-2}
  Soient $f^1 \in \mathcal{H}_{\asp}(\tilde{G}_V, A_{G,\infty})$ et $f^\star \in \mathcal{H}_{\asp}(\tilde{G}^1_V) \otimes C^\infty_c(A_{G,\infty}) \subset \mathcal{H}_{\asp}(\tilde{G}_V)$ tels que
  $$ f^\star|_{\tilde{G}^1_V} = f^1|_{\tilde{G}^1_V}. $$
  Alors on a
  \begin{align*}
    I(f^\star) := I(f^1) & = \sum_{M \in \mathcal{L}(M_0)} \frac{|W^M_0|}{|W^G_0|} \sum_{\tilde{\gamma} \in \Gamma(\tilde{M}^1, V)} a^{\tilde{M}}(\tilde{\gamma}) I_{\tilde{M}}(\tilde{\gamma}, f^\star), \\
    I_t(f^\star) := I_t(f^1) & = \sum_{M \in \mathcal{L}(M_0)} \frac{|W^M_0|}{|W^G_0|} \int_{\Pi_{t,-}(\tilde{M}^1,V)} a^{\tilde{M}}(\pi) I_{\tilde{M}}(\pi,0,f^\star).
  \end{align*}
\end{theoreme}
\begin{proof}
  On désigne par  $f^\dagger \in \mathcal{H}_{\text{ac},\asp}(\tilde{G}_V)$ l'image de $f^1$ par l'inclusion naturelle. On se ramène aisément au cas $f^\star = (f^\dagger)^b$ où $b \in C^\infty_c(\mathfrak{a}_G)$, $b(0)=1$. Puisque $I_{\tilde{M}}(\tilde{\gamma},\cdot)$ (resp. $I_{\tilde{M}}(\pi,0,\cdot)$) est concentrée en $0$, on a $I_{\tilde{M}}(\tilde{\gamma},f^\star)=I_{\tilde{M}}(\tilde{\gamma},f^\dagger)$ (resp. $I_{\tilde{M}}(\pi,0,f^\star)=I_{\tilde{M}}(\pi,0,f^\dagger)$). On conclut par le Lemme précédent.
\end{proof}

\begin{remarque}
  On revient ainsi au point de vue adopté dans \cite{Ar88-TF2}, pour l'essentiel, de travailler avec les fonctions $f^\star$ qui se restreignent en $f^1$ sur $\tilde{G}^1_V$.
\end{remarque}

\subsection{Formes simples de la formule des traces}

Un élément semi-simple $\delta \in G(F)$ est dit $F$-elliptique si $\mathfrak{a}_{G_\delta} = \mathfrak{a}_G$; la même définition s'adapte immédiatement au cas local, puis au cas des revêtements $\tilde{G}_V$ (on passe à l'image par $\rev$).

\begin{definition-fr}\label{def:cuspidale}
  Soient momentanément $F$ un corps local et $\tilde{G} \to G(F)$ un revêtement local. Une fonction $f \in \mathcal{H}_{\asp}(\tilde{G})$ est dite cuspidale si $\Tr\pi(f)=0$ pour toute induite parabolique propre $\pi$. C'est équivalent à la propriété que $I_{\tilde{G}}(\tilde{\gamma}, f)=0$ pour tout $\tilde{\gamma}$ semi-simple régulier non-elliptique.

  Revenons au cas global. Une fonction test décomposable $f = \prod_{v \in V} f_v \in \mathcal{H}_{\asp}(\tilde{G}_V)$ est dite cuspidale en une place $v$ si $f_v$ l'est.
\end{definition-fr}

\begin{theoreme}[Cf. {\cite[\S 7]{Ar88-TF2}}]\label{prop:formules-simples}
  Soient $f = \prod_{v \in V} f_v \in \mathcal{H}_{\asp}(\tilde{G}_V)$, $f \mapsto f^1 \mapsto \mathring{f}^1 = f^1 f_{K^V} \in \mathcal{H}_{\asp}(\tilde{G})$.
  \begin{enumerate}
    \item Si $f$ est cuspidale en une place, alors
      \begin{gather*}
        I_t(f) = \sum_{\mathring{\pi} \in \Pi_{\mathrm{disc},t,-}(\tilde{G}, A_{G,\infty})} a^{\tilde{G}}_{\mathrm{disc}}(\mathring{\pi}) J_{\tilde{G}}(\mathring{\pi}, \mathring{f}^1)
      \end{gather*}
      pour tout $t \geq 0$.
    \item Si $f$ est cuspidale en une place $v_1$ telle que $\Tr\pi_{v_1} (f_{v_1})=0$ pour toute représentation $\pi$ qui est un constituant de $\mathcal{I}_{\tilde{P}}(\sigma)$, où $P=MU$ est un sous-groupe parabolique propre de $G$ et $\sigma \in \Pi_{\mathrm{unit},-}(\tilde{M}_{v_1})$, alors
      $$ I_t(f) = \Tr(R_{\mathrm{disc},t,-}(\mathring{f}^1)) $$
      où $R_{\mathrm{disc},t,-}$ est la $t$-part du spectre discret spécifique de $\tilde{G}$.
    \item Si $f$ est cuspidale en deux places, alors
      $$ I(f) = \sum_{\tilde{\gamma} \in \Gamma(\tilde{G}^1, V)} a^{\tilde{G}}(\tilde{\gamma}) I_{\tilde{G}}(\tilde{\gamma}, f^1). $$
    \item Supposons que
      \begin{itemize}
        \item $f^1 \in \mathcal{H}_{\mathrm{adm}, \asp}(\tilde{G}_V, A_{G,\infty})$, l'espace défini par \eqref{eqn:H-adm};
        \item $f$ est cuspidale en une place $v_1$;
        \item il existe une autre place $v_2 \in V$ telle que $I_{\tilde{G}}(\tilde{\gamma}_{v_2}, f_{v_2})=0$ sauf si $\tilde{\gamma}_{v_2}$ est semi-simple $F_{v_2}$-elliptique (pas forcément régulier).
      \end{itemize}
      Alors
      $$ I(f^1) = \sum_{\gamma \in G(F)^{\mathrm{bon}}_{\mathrm{ell}}/\mathrm{conj}} \frac{\mes(G_\gamma(F) A_{G,\infty} \backslash G_\gamma(\A))}{[G^\gamma(F) : G_\gamma(F)]} \int_{G_\gamma(\A) \backslash G(\A)} \mathring{f}^1(x^{-1} \gamma x) \dd x $$
      où $G(F)^{\mathrm{bon}}_{\mathrm{ell}}$ signifie l'ensemble des éléments $F$-elliptiques semi-simples dans $G(F)$ qui sont bons en tant qu'éléments de $\tilde{G}$. 
  \end{enumerate}
\end{theoreme}
\begin{proof}
  La preuve est très similaire à celles dans \cite[\S 7]{Ar88-TF2}. Donnons-en une esquisse.

  Montrons la première assertion. Nous allons simplifier le développement spectral \eqref{eqn:I-spec}. Comme dans op.cit, la cuspidalité entraîne que $I_{\tilde{M}}(\pi,X,f)=0$ pour tout $M \neq G$ et tout $(\pi,X)$. Donc $I_{\tilde{M}}(\pi,0,f^1)=0$ d'après l'une des propriétés dans l'Hypothèse \ref{hyp:recurrence-spec}. Seulement les termes avec $M=G$ survivent. Considérons ensuite des représentations de la forme $\pi=\rho^G_\lambda$, où $\rho \in \Pi_{\text{unit},-}(\tilde{L}^1,V)$ et $\lambda \in i(\mathfrak{a}^G_L)^*$, $L \in \mathcal{L}(M_0)$. Appliquons la cuspidalité encore une fois pour obtenir $I_{\tilde{G}}(\rho^G_\lambda, f^1)=0$ si $L \neq G$. L'égalité s'en suit d'après la Définition \ref{def:coeff-spec} et le fait que $J_{\tilde{G}}(\pi \times c, \mathring{f}^1) = I_{\tilde{G}}(\pi, f^1)$ pour tout $c \in \mathcal{C}_-(\tilde{G}^V)$.

  La deuxième assertion résulte rapidement de la première. En effet, regardons \eqref{eqn:coeff-spec-caracterisation}. Dans le côté à gauche, la présence de $f_{v_1}$ annule tous les termes avec $P \neq G$. On retrouve donc la trace de $R_{\text{disc},t,-}$.

  Montrons la troisième assertion. Comme dans \cite[p.539]{Ar88-TF2}, des formules de descente permettent de montrer que $I_{\tilde{M}}(z\tilde{\gamma},f)=0$ si $M \neq G$ et $z \in A_{G,\infty}$. D'où $I_{\tilde{M}}(\tilde{\gamma}, f^1)=0$ si $M \neq G$ d'après l'une des propriétés dans l'Hypothèse \ref{hyp:recurrence-geom}.

  Pour établir la quatrième assertion, on observe tout d'abord que $f$ est cuspidale en $v_1$ et $v_2$. On a donc $I(f) = \sum_{\tilde{\gamma}} a^{\tilde{G}}(\tilde{\gamma}) I_{\tilde{G}}(\tilde{\gamma}, f^1)$. La Remarque \ref{rem:V-S-convention} dit qu'en prenant $V=S$, ce qui est loisible, on a $a^{\tilde{G}}(\tilde{\gamma})=a^{\tilde{G}}_\text{ell}(\tilde{\gamma})$. Grâce au Lemme \ref{prop:a_ell}, on retrouve la partie $I_\text{ell}(f^1)$ définie dans \eqref{eqn:I_ell} du développement géométrique fin. La distribution $I_\text{ell}(f^1)$ est facile à décrire: d'une part, les coefficients $a^{\tilde{G}}(V, \tilde{\gamma})$ se calculent à l'aide de  la Définition \ref{def:coeff-geom-0} et du fait que $\tilde{\gamma}_{v_2}$ est $F_{v_2}$-elliptique et semi-simple; d'autre part, la $(G,V)$-équivalence se réduit à $G(F)$-conjugaison pour les éléments semi-simples. Enfin, la condition d'admissibilité sur le support de $f^1$ entraîne que si $\gamma \in (G(F))^{K,\text{bon}}_{\tilde{G},V}$ (confondu avec un 
représentant admissible), $\gamma \leadsto \tilde{\gamma}
$ 
par \eqref{eqn:leadsto}, alors
  $$ J_{\tilde{G}}(\tilde{\gamma}, f^1) = \int_{G_\gamma(\A) \backslash G(\A)} \mathring{f}^1(x^{-1} \gamma x) \dd x, \quad \text{si $\gamma$ est bon dans $\tilde{G}$ }. $$
  On arrive ainsi à l'égalité cherchée.
\end{proof}

\bibliographystyle{abbrv-fr}
\bibliography{metaplectic}

\begin{flushleft}
  Wen-Wei Li \\
  Institute of Mathematics, \\
  Academy of Mathematics and Systems Science, Chinese Academy of Sciences, \\
  55, Zhongguancun East Road, \\
  100190 Beijing, China. \\
  Adresse électronique: \texttt{wwli@math.ac.cn}
\end{flushleft}

\end{document}